\documentclass[draft=false]{article}

\usepackage{graphicx} 
\usepackage{amsmath,amsfonts,amssymb,amscd,amsthm,mathrsfs}
\usepackage{ cmll }
\usepackage{tikz,tikz-cd}
\usetikzlibrary{graphs}
\usepackage{tangle}
\usepackage[inner=3cm,outer=3cm,top=2.5cm,bottom=2.5cm]{geometry}
\usepackage{pdflscape}

\usepackage[mathscr]{eucal}
\usepackage{bbold}
\usepackage[pdffitwindow=false,colorlinks=true,linkcolor=blue,urlcolor=blue,citecolor=violet]{hyperref}
\usepackage{todonotes}

\usepackage{tikzit}

\tikzstyle{grey dot}=[fill={rgb,255: red,225; green,225; blue,225}, draw=black, shape=circle, minimum width=1mm]
\tikzstyle{test}=[fill=white, draw=black, shape=circle]
\tikzstyle{rectangle}=[fill=white, draw=black, shape=rectangle, minimum width=3cm, minimum height=2cm]
\tikzstyle{small}=[fill=none, draw=none, shape=circle, font={\footnotesize}]
\tikzstyle{black rectangle}=[fill={rgb,255: red,64; green,64; blue,64}, draw=black, shape=rectangle]
\tikzstyle{grey rectangle}=[fill={rgb,255: red,225; green,225; blue,225}, draw=black, shape=rectangle]
\tikzstyle{small font circle}=[fill={rgb,255: red,225; green,225; blue,225}, draw=black, shape=circle, inner sep=1pt, font={\scriptsize}]
\tikzstyle{big rectangle}=[fill=white, draw=black, shape=rectangle, minimum width=4cm, minimum height=12mm]

{\theoremstyle{definition}\newtheorem{df}         {Definition}  [section]}
{\theoremstyle{definition}\newtheorem{setting}[df]{Setting}     }
{\theoremstyle{remark}    \newtheorem{rk}   [df]  {Remark}      }
{\theoremstyle{remark}    \newtheorem{ex}   [df]  {Example}     }
{\theoremstyle{remark}    \newtheorem{nota} [df]  {Notation}    }
{\theoremstyle{remark}      }
{\theoremstyle{remark}    \newtheorem{construction} [df]  {Construction}  }
{\theoremstyle{plain}     \newtheorem{thm}  [df]  {Theorem}     }
{\theoremstyle{plain}     \newtheorem{lm}   [df]  {Lemma}       }
{\theoremstyle{plain}            }
{\theoremstyle{plain}     \newtheorem{cor}  [df]  {Corollary}   }
{\theoremstyle{plain}     \newtheorem{prop} [df]  {Proposition} }
{\theoremstyle{plain}       }

\newtheorem{thmintro}{Theorem}

\ExplSyntaxOn
\NewDocumentCommand{\createbunch}{ m O{} m }
{ \clist_map_inline:nn { #3 } { \cs_new_protected:cpn { #2 ##1 } { #1 { ##1 } } } }
\ExplSyntaxOff

\createbunch{\mathbb}   [b]{A,B,C,D,E,F,G,H,I,J,K,L,M,N,O,P,Q,R,S,T,U,V,W,X,Y,Z}
\createbunch{\mathcal}  [c]{A,B,C,D,E,F,G,H,I,J,K,L,M,N,O,P,Q,R,S,T,U,V,W,X,Y,Z}
\createbunch{\mathsf}   [s]{A,B,C,D,E,F,G,H,I,J,K,L,M,N,O,P,Q,R,S,T,U,V,W,X,Y,Z}
\createbunch{\mathbf}   [f]{A,B,C,D,E,F,G,H,I,J,K,L,M,N,O,P,Q,R,S,T,U,V,W,X,Y,Z}
\createbunch{\mathtt}   [t]{A,B,C,D,E,F,G,H,I,J,K,L,M,N,O,P,Q,R,S,T,U,V,W,X,Y,Z}
\createbunch{\mathrm}   [r]{A,B,C,D,E,F,G,H,I,J,K,L,M,N,O,P,Q,R,S,T,U,V,W,X,Y,Z}
\createbunch{\mathfrak} [k]{A,B,C,D,E,F,G,H,I,J,K,L,M,N,O,P,Q,R,S,T,U,V,W,X,Y,Z}
\createbunch{\mathscr}  [z]{A,B,C,D,E,F,G,H,I,J,K,L,M,N,O,P,Q,R,S,T,U,V,W,X,Y,Z}
\createbunch{\mathsf}   {sSet,Top,Ch,Mod,Ho,Set,Ab,cdga,dga,Fin,Cat,Op,Gpd,Tup,LMod,Grp,Grpd,Span,Cospan,BiAlg}
\createbunch{\mathrm}[i]{sSet,Top,Ch,Mod,Ho,Set,Ab,cdga,dga,Fin,Cat,Op,Gpd,Tup,LMod,Grp,Grpd,Span,Cospan,BiAlg}
\createbunch{\mathtt}[t]{sSet,Top,Ch,Mod,Ho,Set,Ab,cdga,dga,Fin,Cat,Op,Gpd,Tup,LMod,Grp,Grpd,Span,Cospan,BiAlg}

\createbunch{\operatorname}{Hom,End,Tor,Ext,Mul,Ob,tr,ad,Ad,Id,id,im,Map,Fun,Sing,Spec,Sym,rad,Ind,Perf,Coh,QCoh,HH,ev,coev,Lie,Env}

\DeclareMathOperator{\colim}{colim}
\renewcommand{\o}{\otimes}

\newcommand{\too}   {\longrightarrow}
\newcommand{\baru}{\overline{\zU}}

\newcommand{\1}{\mathbb{1}}
\newcommand{\inn}{\mathrm{in}}
\newcommand{\out}{\mathrm{out}}

\newcommand{\Diop}{\mathbf{Diop}}
\newcommand{\Prpd}{\mathbf{Prpd}}
\newcommand{\Alg}{\mathbf{Alg}}
\newcommand{\Cato}{\mathbf{Cat}^\otimes}
\newcommand{\Frob}{\mathbf{Frob}}
\title{Dioperads, Frobenius monoidal functors and duality}

\author{
    Valerio Melani\thanks{ DIMAI, Universit\`a degli Studi di Firenze, Italy
    \\
    \href{mailto:valerio.melani@unifi.it}{valerio.melani@unifi.it}},
    Hugo Pourcelot\thanks{ Univ Angers, CNRS, LAREMA, SFR MATHSTIC, F-49000 Angers, France  \\
    \href{mailto:hugo.pourcelot@univ-angers.fr}{hugo.pourcelot@univ-angers.fr}}
}

\date{\today}

\allowdisplaybreaks[1]
\begin{document}

\maketitle
\begin{abstract}
Motivated by duality phenomena for derived global sections on derived local systems on compact oriented manifolds, we introduce the notion of a $d$-duality context between symmetric monoidal enriched categories. In this setting, the right adjoint of a symmetric monoidal functor carries compatible lax and colax structures twisted by an invertible object $d$.

For any enriched dioperad $\zP$, we define a $d$-twist $\zP\{d\}$ and prove that, in a $d$-duality context, the right adjoint sends $\zP$-algebras to $\zP\{-d\}$-algebras. To achieve this, the key conceptual result is that Frobenius monoidal functors between symmetric monoidal categories are precisely those functors inducing morphisms between the underlying dioperads. We also develop a dioperadic Day convolution, yielding an alternative proof of the main theorem and suggesting an $\infty$-categorical extension of the theory. 
\end{abstract}
\tableofcontents

\section{Introduction}

Operads are mathematical objects designed to model algebraic structures whose operations have a finite number of inputs and a single output. In addition to their theoretical importance, operads have also proven to be a useful tool in more geometric contexts. Indeed, besides their first historical appearance in algebraic topology in the recognition principle for iterated loop spaces via $\mathbb{E}_n$-algebras, they play a crucial role in string topology via operadic actions on free loop spaces, in the theory of deformation quantization or in the study of configuration spaces of manifolds. When working in the context of \emph{higher} or \emph{homotopical} geometry, it is often useful to employ operads to describe geometric structures in terms of algebraic operations (see for example \cite{AyalaFrancis_FactHomology}, \cite{Toen_Branes}, \cite{MannRobalo_BraneActions}, \cite{Melani_PoissonAffine}, \cite{CPTVV}).

Dioperads are structures similar to operads: they are mathematical tools designed to model (bi)algebraic structures that involve operations with multiple inputs as well as multiple outputs. 
They were originately defined by Szabo in \cite{Szabo_Polycategories} under the name \emph{polycategories}, and later studied in  \cite{Cockett-Seely} for their connection to linear logic. Although they are of importance in theoretical computer science, in this work we will adopt a more algebraic perspective, in which the typical examples of dioperads are those governing Lie bialgebras and Frobenius algebras. 
As it is the case with operads, given a dioperad $\zP$ and a symmetric monoidal category $\zC$, one can consider the category of $\zP$-algebras in $\zC$, which can be defined as the category of dioperadic morphisms bewteen $\zP$ and $\zU \zC$, the underlying dioperad of $\zC$.
Dioperads form a 2-category in a natural way. 

The present work is motivated by the idea that dioperads should provide an appropriate language for certain duality phenomena in geometry, where algebraic and coalgebraic structures interact in a controlled way (see for example \cite{johnson-freyd_aksz}, \cite{JohnsonFreyd_Coisotropic}). 
This general principle should not be too surprising: indeed, in \cite{Boyarchenko-Drinfeld2013} Boyarchenko and Drinfeld explained that Verdier duality in algebraic geometry naturally yields examples of \emph{$*$-autonomous categories} (which they call \emph{Grothendieck--Verdier categories}), a categorical structure closely related to dioperads\footnote{
More precisely, $*$-autonomous categories can be seen as \emph{representable} dioperads with duals, in the same way rigid symmetric monoidal categories are \emph{representable} operads with duals. We refer to Section \ref{section:recollections_*-autonomous} for details on $*$-autonomous categories.
}.
This paper is a step forward to study the possible uses of dioperads in geometry.

Our starting point was the following elementary observation. Suppose $M$ is a $d$-dimensional compact oriented topological manifold, and let $D_{\mathrm{loc}}(M,k)$ be the subcategory of the derived category of sheaves of $k$-modules on $M$ for a commutative ring $k$, consisting of complexes whose cohomology sheaves are local systems on $M$. Inside $D_{\mathrm{loc}}(M,k)$, let $D_{\mathrm{loc}}^p(M,k)$ be the full subcategory of objects which are locally quasi-isomorphic to complexes of projective modules of finite type. Then the derived global section functor $\Gamma$ on $D_{\mathrm{loc}}^p(M,k)$ admits a lax symmetric monoidal structure (being a right adjoint), while by Poincaré duality its cohomological shift $\Gamma[d]$ is identified with a \emph{left} adjoint and therefore carries a \emph{colax} symmetric monoidal structure. As a result, $\Gamma$ preserves algebraic structures, while it preserves coalgebraic structures up to a cohomological shift. It is natural to investigate what kind of bialgebraic structures are preserved by $\Gamma$.

An $\infty$-categorical way of modeling this type of situation is suggested by the discussion in \cite[Section 2.1]{PTVV}, where the authors define notions of compactness and orientations for general derived stacks.
Although the $\infty$-categorical question is definitely interesting (and we will briefly discuss it at the end of this introduction), this paper is devoted to answering the 1-categorical version, which turned out to be interesting in its own right.

To start with, suppose that we are given two rigid symmetric monoidal categories $\zC$ and $\zD$, both enriched over a closed bicomplete symmetric monoidal category $\zV$. Let $d \in \zV$ be an invertible object. We say that a \emph{$d$-duality context} is a pair of adjoint functors $F \colon \zD \rightleftharpoons\zC \colon G$, where the left adjoint $F$ is symmetric monoidal, together with a $d$-orientation on the functor $G$ (see  Definition \ref{df:duality_context} for details). In the motivating example above, $\zV$ is the category of complexes of $k$-modules, and by abuse of notation $d$ denotes the complex given by $k[d]$.

Moreover, given a $\zV$-enriched dioperad $\zP$ we define its $d$-twisted version $\zP\{d\}$, which is another dioperad. Intuitively, $\zP\{d\}$ is obtained from $\zP$ by twisting every co-operadic operation (i.e., every operation with several outputs) by tensoring with an appropriate power of $d$. We refer to Definition \ref{df:TwistedDioperad} for details.
Our first main result is then the following.

\begin{thmintro}[Cor. \ref{cor:main}]\label{thm:IntroA}
    Let $\zP$ be a dioperad. Given a $d$-duality context $F \colon \zD \rightleftharpoons\zC \colon G$, the right adjoint $G$ sends $\zP$-algebras in $\zC$ to $\zP\{-d\}$-algebras in $\zD$.
\end{thmintro}

Our strategy for proving this statement is as follows. We start by identifying the functors that preserve dioperadic structures. It turns out that the correct notion is that of Frobenius monoidal functors, which was introduced and studied in \cite{Szlachnyi2000FiniteQG}, \cite{Szlachnyi2003AdjointableMF}, \cite{Day2008}. We recall the precise definition in Section \ref{sect:FrobMonFunct}, but roughly speaking a Frobenius monoidal functor between monoidal categories is a functor that is both lax and colax in a suitably compatible way. In the same way lax symmetric monoidal functors are exactly those preserving algebras over operads, we prove that Frobenius monoidal functors are exactly those preserving algebras over dioperads. More precisely, our second result takes the following form.

\begin{thmintro}[Thm. \ref{thm:frob_as_diop_map}]\label{thm:IntroB}
    Let $\zC, \zD$ be symmetric monoidal $\zV$-enriched categories, and let $\zU\zC, \zU\zD$ be their underlying dioperads. Then there is an equivalence of categories
    \[ \Diop_\zV(\zU\zC, \zU\zD) \simeq \mathbf{Frob}_\zV^{\otimes}(\zC, \zD) \]
    where the right hand side is the category of Frobenius monoidal functors from $\zC$ to $\zD$.
\end{thmintro}

We use Theorem \ref{thm:IntroB} to prove Theorem \ref{thm:IntroA}. In fact, in Section \ref{sec:orientation} we explicitly construct a $d$-twisted Frobenius monoidal structure on the functor $G$ coming from a $d$-duality context. The proof is mostly carried out via diagrammatic/graphical calculus.

Finally, in Section \ref{sec:day_convolution} we propose a different approach to Theorem \ref{thm:IntroA}. The idea is to use a dioperadic Day convolution. To state our result, recall the notion of a \emph{$*$-autonomous} category mentioned above, which consists of a closed symmetric monoidal category together with a \emph{dualizing object} (in the sense of Definition \ref{df:starAutCats}). Given such a $*$-autonomous category $\zC$, there is a standard way (described in Section \ref{sec:UndelyingDioperad}) to extract a dioperad from it, which we denote by $\baru\zC$. The following result was conjectured in \cite{Egger_StarAutonomousFctCategories}.

\begin{thmintro}[Thm. \ref{thm:day}]\label{thm:IntroC}
    Let $\zC$ be a $*$-autonomous category and $\zD$ be a dioperad. Then there exists a dioperad $\Fun(\zC,\zD)$ with the following universal property: for every dioperad $\zO$, there is an equivalence of categories
    \begin{equation*}
        \Diop(\zO,\Fun(\zC,\zD)) \simeq 
        \Diop(\zO\times \overline{\zU}\zC, \zD).
    \end{equation*}
\end{thmintro}

We believe Theorem \ref{thm:IntroC} may be interesting in its own right, but in Section \ref{sec:day_application} we sketch how it can also give a different proof of Theorem \ref{thm:IntroA}. The idea here is to use a straightforward enriched generalization of Theorem \ref{thm:IntroC}. In particular, one prove that if both $\zC$ and $\zD$ are enriched over $\zV$, then there is an equivalence of categories
\begin{equation*}
        \Diop_\zV(\zO,\Fun(\zC,\zD)) \simeq 
        \Diop_\zV(\zO\otimes \overline{\zU}\zC, \zD).
    \end{equation*}
where the tensor product on the right-hand side is the Hadamard tensor product of dioperads. In particular, if $\zO$ is the dioperad $\mathrm{Frob}$ (which encodes Frobenius algebras, and is the unit for the Hadamard tensor product), we get an equivalence
\begin{equation*}
        \Alg_{\mathrm{Frob}}(\Fun(\zC,\zD)) \simeq 
        \Diop_\zV(\overline{\zU}\zC, \zD).
    \end{equation*}
If moreover the dioperad $\zD$ is of the form $\zU\zE\{d\}$ for some rigid symmetric monoidal $\zV$-enriched category $\zE$ and some invertible object $d \in \zV$, we obtain
\begin{equation*}
     \Alg_\mathrm{Frob}\left( \Fun(\zC,\zU\zE\{d\}) \right) \simeq \Diop_\zV(\zU\zC,\zU\zE\{d\}).
\end{equation*}
In particular, $d$-twisted Frobenius monoidal functors $\zC \to \zE$ correspond to Frobenius algebras in $\Fun(\zC,\zU\zE\{d\})$. Section \ref{sec:day_application} sketches how to produce such a Frobenius algebra starting from a $d$-duality context in the sense of Definition \ref{df:orientation}, thus giving a different approach to Theorem \ref{thm:IntroA}.

\

{\bf{Related works and vistas.}} Recently, there has been some growing interest in the construction of Frobenius monoidal functors, as in \cite{Yadav_coHopfAdjunctions} and \cite{FLAKE_AmbiadjunctionsDrinfeldCenters}. In both papers, the authors construct Frobenius monoidal structures, starting with a fairly general categorical situation. This is very much in the spirit of the results in Section \ref{sect:FrobMonFunct}, even though our framework is somewhat different. It would be interesting to study the relation between their constructions and our results. 

As we already mentioned in this Introduction, we think it would be interesting to have an $\infty$-categorical version of our results. In particular, if $X$ is a $\zO$-compact $d$-oriented derived stack (in the sense of \cite[Section 2.1]{PTVV}) over a base field $k$, we expect the (derived) pushforward
$$\Perf(X) \to \Perf(k)$$
to be a Frobenius monoidal $\infty$-functor. However, to our knowledge, there is no definition of Frobenius monoidal $\infty$-functors in the literature. An idea is to use intuition from Theorem \ref{thm:IntroB}. More specifically, one could hope to study $\infty$-dioperads first, and prove that any symmetric monoidal $\infty$-category $\zC$ has an underlying $\infty$-dioperad $\zU\zC$. Then a Frobenius monoidal $\infty$-functor between symmetric monoidal $\infty$-categories could be defined as being a map between the underlying $\infty$-dioperads.
This is much in the spirit of the definition of lax symmetric monoidal $\infty$-functors appearing in \cite[Section 2.1.3]{Lurie_HA}. We plan to exploit this point of view in a forthcoming work.

As a possible application of the $\infty$-categorical version of our results, we think our approach could lead to a different proof of the AKSZ theorem for shifted Poisson structures, which was recently proved in \cite{Tomic_AKSZ}. 
To this purpose, we hope to prove an $\infty$-categorical version of Theorem \ref{thm:IntroC}, which may also be of independent interest (see \cite{Glasman_DayConvolution}, \cite[Section 2.2.6]{Lurie_HA} for Day convolution results for $\infty$-categories and $\infty$-operads). We then plan to follow the strategy outlined in Section \ref{sec:day_application} to show that, in an $\infty$-categorical duality context (as in \cite[Section 2.1]{PTVV}), the pushforward functor inherits a natural $d$-twisted Frobenius monoidal structure.

Finally, we think that a dioperadic approach to the Poisson AKSZ construction would also work in the relative setting, similarly to the relative versions of the symplectic AKSZ construction (see \cite{Calaque_LagrangianAKSZ}). In the Poisson case, we expect that the AKSZ theorem will produce shifted coisotropic structures (in the sense of \cite{MelaniSafronov_DerCoisotropic_1}, \cite{MelaniSafronov_DerCoisotropic_2}). We refer to \cite{JohnsonFreyd_Coisotropic} for a partial result in this direction.

\paragraph{Acknowledgements.}
We thank Félix Loubaton and Nikola Tomi\'c for the interesting discussions related to the topic of this article. 

The second author is partially supported by the Agence Nationale de la Recherche, project \href{https://anr.fr/Project-ANR-24-CE40-2583}{ANR-24-CE40-2583}. 

\subsection*{Notations and conventions}

\paragraph*{On symmetric monoidal categories.}
\label{nota:tensor}
To keep the computations in symmetric monoidal categories readable, we will systematically omit the unitors, associators and braidings in the notations, by virtue of MacLane's coherence theorem.
Moreover, to shorten notations, we may sometimes remove symbols $\otimes$, leaving the tensor products implicit; for example, $G(XY)GZ$ will stand for $G(X\otimes Y)\otimes G(Z)$.

\paragraph*{On enriched categories.}
\label{par:enriched_cats}
Throughout this paper, we work with categories, dioperads and properads enriched over a symmetric monoidal category $(\zV,\otimes, \1)$, which we further assume to be bicomplete with tensor product commuting with colimits in each variables. In particular, morphisms, functors, natural transformations, symmetric monoidal categories, dualizable objects in them, etc. are to be interpreted in the sense of enriched category theory. 

Given a $\zV$-enriched category $\zC$, we will abuse notation by writing $f\colon A\to B$ for a morphism $\1\to \zC(A,B)$ and say that $f$ is a morphism from $A$ to $B$ in $\zC$. In the same spirit, we will typically treat morphisms as in ordinary category theory, writing equations and composites of individual morphisms. In every such situation, it is straightforward to reformulate the relevant equalities as commutative diagrams in $\zV$.

The (large) $2$-category of small symmetric monoidal $\zV$-categories, symmetric monoidal functors and monoidal natural transformations between them will be denoted  $\Cato_\zV$.

\section{Recollections on properads and dioperads}
In this section, we recall the basic definitions and properties of $\zV$-dioperads and $\zV$-properads, their associated $2$-categories $\Diop_\zV$ and $\Prpd_\zV$ and their relation with symmetric monoidal $\zV$-categories through envelope functors. 
The results of this section are gathered from standard references (mainly \cite{Yau_Johnson_PROPsBook,Johnson_Yau_2dimCatsBook}), if only slightly extended to the enriched setting, with the exception of Proposition \ref{prop:envelope_E^p}, whose statement we were not able to locate in the literature and of which we provide a proof (note however that an $\infty$-categorical version of this proposition is established in \cite{Barkan-Steinebrunner_properads}).

\subsection{Graphs}

We recall some definitions and conventions on graphs, following \cite{Yau_Johnson_PROPsBook}.

\begin{df}
    \label{df:digraph}
    A \emph{digraph} is a directed graph with half-edges which has no directed cycles. Given  a digraph $G$, we let $V(G)$, $E(G)$, $\inn(G)$, $\out(G)$ denote respectively the sets of vertices, edges, input leaves and output leaves. Given a vertex $v\in V(G)$, we let $\inn_v$ and $\out_v$ denote its sets of inputs and outputs, which are either edges or leaves.
\end{df}

Fix a set $S$. We now introduce variants of digraphs whose edges are colored with elements in $S$.
\begin{df}
    \label{df:profile} 
    An \emph{$S$-profile} $\underline{x} = (x_1,\dots, x_n)$ is an (ordered) finite sequence of elements in $S$, possibly empty. An \emph{$S$-biprofile} is a pair of $S$-profiles, typically written as $(\underline{x},\underline{y})$ or $(x_1,\dots,x_n;y_1,\dots, y_m)$. Whenever the context is clear, we will forget the set $S$ in the notation.
\end{df}

\begin{nota}
    \label{nota:unordered_sets}
    Slightly abusing notation, we will sometimes write an $S$-profile $\underline{x}$ as $(x_i)_{i\in I}$, where $I = (I^\mathrm{u},<_I)$ is a linearly ordered finite set; we say that $I^\mathrm{u} \in \Fin$ is the underlying set of $I$. If $A\subseteq I^\mathrm{u}$ is a subset, we can form the subprofile $\underline{x}|_{A} := (x_i)_{i\in A}$ of $\underline{x}$, where we implicitly endow $A$ with the restriction of the linear order $<_I$. 
\end{nota}

\begin{df}
    \label{df:colored_graph} 
    An \emph{$S$-colored graph} is a digraph $G$ together with a coloring of edges, that is a function $E(G) \to S$, and orderings of the sets $\inn(G)$, $\out(G)$, $\inn_v$ and $\out_v$ for every vertex $v$. An $S$-colored graph $G$ therefore has an associated biprofile $(\inn(G),\out(G))$, as well as biprofiles $(\inn_v,\out_v)$ for every vertex $v$.
\end{df}

\begin{ex}
    \label{ex:corolla}
    Given a biprofile $(\underline{x},\underline{y}) = (x_1,\dots,x_n;y_1,\dots,y_m)$, the corolla $\mathfrak{c}_{\underline{x},\underline{y}}$ is the colored graph with one vertex, $n$ input leaves colored by $\underline{x}$ and $m$ output leaves colored by $\underline{y}$, with the obvious orderings.
\end{ex}


\subsection{Definition of properads and dioperads}

Recall that we have fixed  a closed bicomplete symmetric monoidal category $(\zV,\otimes,\1)$, in which the tensor product commutes with colimits in each variable.
\begin{df}
    A \emph{properad $\zP$ in $\zV$} (also called \emph{$\zV$-properad}) consists of 
    \begin{itemize}
        \item a set $S=\mathrm{Col}(\zP)$, whose elements are called \emph{colors},
        \item an object $\zP(\underline{x},\underline{y}) \in \zV$ for every biprofile $(\underline{x},\underline{y})$, which encodes \emph{operations} with inputs $\underline{x}$ and outputs $\underline{y}$,
        \item a morphism $\id_x\colon \1 \to \zP(x,x)$ in $\zV$ for every color $x\in S$,
        \item composition morphisms
            \begin{equation*}
                \label{eqn:def_properad_composition}
                \gamma_G \colon \zP[G]:= \bigotimes_{v\in V(G)} \zP(\mathrm{in}_v, \mathrm{out}_v) \longrightarrow \zP(\mathrm{in}(G), \mathrm{out}(G))
            \end{equation*}
            for every $S$-colored connected graph $G$, 
    \end{itemize}
    that satisfy appropriate unitality and associativity  conditions (see \cite[Definition 10.39]{Yau_Johnson_PROPsBook} for a precise definition).
\end{df}

\begin{df}
    A \emph{dioperad in $\zV$} (also called \emph{$\zV$-dioperad}) is defined in the same way, except that one considers only composition morphisms \eqref{eqn:def_properad_composition} indexed by $S$-colored \emph{simply connected} connected digraphs.
\end{df}

\begin{rk}
    Our dioperads and properads are implicitly colored. Note also that dioperads also appear under the name of \emph{polycategories} in the literature.
\end{rk}

\begin{df}
    A morphism between properads (or dioperads) $\zP \to \zP'$ is given by a map of underlying sets of colors $\varphi\colon S\to S'$ and a family of morphisms $\zP(\underline{x},\underline{y})\to \zP'(\varphi(\underline{x}),\varphi(\underline{y}))$ compatible with the identity and composition morphisms.
\end{df}

\begin{rk}
    [Biased definitions of properads and dioperads]
    In the above definition of a properad, the composition along a general $S$-colored connected graph $G$ is determined by compositions along graftings of two corollas, glued along $k>0$ common colors. Similarly, in the case of dioperads, general compositions are determined by compositions along grafting of two corollas, glued along a single common color. 
    It will be convenient when defining a properad (respectively a dioperad), or a morphism between such, to only specify these particular compositions.
    We refer to the book \cite[Definition 3.5, Remark 3.9]{hackney-robertson-yau_book} for details on these alternative definitions of properads and dioperads.
\end{rk}

\subsubsection{Polynatural transformations}

Recall from \cite[Section 2.5]{Johnson_Yau_2dimCatsBook} that given two parallel morphisms of dioperads $F,G : \zP \to \zQ$ there is a notion of polynatural transformation between $F$ and $G$. Following \cite[Definition 2.5.16]{Johnson_Yau_2dimCatsBook}, a \emph{polynatural transformation} $\theta: F \to G$ is given by a family of unary operations in $\zQ$
\[ \theta_X : FX \to GX \]
indexed by the colors $X$ of $\zP$. These morphisms have to satisfy the obvious naturality condition with respect to (not necessarily unary) operations in $\zP$. This definition generalize in a straightforward manner to the enriched case.

By \cite[Theorem 2.5.19]{Johnson_Yau_2dimCatsBook} (suitably extended to the enriched setting), $\zV$-dioperads, morphisms of $\zV$-dioperads and polynatural transformations between them form a (large) $2$-category, that we denote $\Diop_\zV$.

One can define polynatural transformations between parallel morphisms of $\zV$-properads in a completely similar manner; the resulting (large) $2$-category will be denoted $\Prpd_\zV$.

Given a $\zV$-properad, one can forget the non-simply connected compositions, while keeping the same set of colors and objects of operations, to obtain its \emph{underlying $\zV$-dioperad}. This construction defines a $2$-functor that we denote
\begin{equation}
    \label{eqn:U^d}
    \zU^d \colon \Prpd_\zV \longrightarrow \Diop_\zV.
\end{equation}

\subsubsection{Underlying properad and dioperad of a symmetric monoidal category}

Symmetric monoidal $\zV$-categories have underlying properads and dioperads, as we now explain. 
The \emph{underlying properad functor} is the $2$-functor
\begin{equation*}
    \label{eqn:U^p}
    \zU^p \colon \Cato_\zV \longrightarrow \Prpd_\zV,
\end{equation*}
that sends a symmetric monoidal $\zV$-category $\zC$ to the $\zV$-properad $\zU^p\zC$ with colors given by the objects of $\zC$ and with operations given by 
\begin{equation*}
    \zU^p\zC((x_i)_{i\in I}, (y_j)_{j\in J}) := \zC(\otimes_{i} x_i, \otimes_{j}y_j),
\end{equation*}
where composition is induced by that of $\zC$ in the obvious way.
Composing this $2$-functor with that of \eqref{eqn:U^d}, we obtain the \emph{underlying dioperad ($2$-)functor}
\begin{equation*}
    \zU:= \zU^d\circ \zU^p \colon \Cato_\zV \longrightarrow \Diop_\zV.
\end{equation*}

\subsection{Envelopes}

We will recall explicit descriptions of the symmetric monoidal and properadic envelope functors, respectively denoted $\zE^p$ and $\zE^d$, which can be thought of as left adjoints to the functors $\zU^p$ and $\zU^d$. This interpretation is literally correct for the latter functor (see Proposition \ref{prop:envelope_E^d}), but only true in a weak sense for the former (see Proposition \ref{prop:envelope_E^p} and subsequent Remark \ref{rk:biadjunction}).

\subsubsection{Properadic envelope of a dioperad}
We describe the properadic envelope functor $\zE^d$, which is to be thought of as taking the free properad on a given dioperad, following \cite{Yau_Johnson_PROPsBook}.

First we need some terminology. Recall that given an $S$-colored graph $G$ and an $S$-colored graph $H_v$ of biprofile $(\inn_v,\out_v)$ for every vertex $v$ of $G$, that we call a \emph{substitution data}, we can substitute $H_v$ into the vertex $v$ to form a new $S$-colored graph denoted $G[\{H_v\}_v]$ (we refer to \cite[Chapter 5]{Yau_Johnson_PROPsBook} for details on this construction). With this notion of graph substitution, we can form the following category, which controls the shape of operations in the properadic envelope of a dioperad. 

\begin{df}[{\cite[Definition 12.5]{Yau_Johnson_PROPsBook}}]
    \label{df:extension_category}
    Given an $S$-biprofile $(\underline{x},\underline{y})$, the \emph{extension category} $\cG(\underline{x},\underline{y})$ has objects given by connected digraphs $G$ with biprofile $(\underline{x},\underline{y})$, morphisms $K\to G$ are strict isomorphism classes\footnote{A strict isomorphism of $S$-colored graphs is a graph isomorphism that moreover preserves the colorings and orderings. Restricting to strict isomorphism classes of substitution data ensures that the $\Hom$ sets in the extension category are small.} of substitution data $\{H_v\}_{v\in V(G)}$ such that $K = G[\{H_v\}_v]$ and every digraph $H_v$ is  connected and simply connected. Composition is given by associativity of graph substitution. 
\end{df}
\begin{rk}
    In the general framework of \cite{Yau_Johnson_PROPsBook}, the category $\cG(\underline{x},\underline{y})$ would be called the extension category associated with the inclusion of pasting schemes $\mathsf{Gr}_\mathrm{di}^\uparrow \leqslant \mathsf{Gr}_{\mathrm{c}}^\uparrow$, from the simply-connected pasting scheme (encoding dioperads) to the connected one (encoding properads).
\end{rk}

\begin{df}
    Let $\zO$ be a $\zV$-dioperad. We define the \emph{properadic envelope} of $\zO$ to be the $\zV$-properad $\zE^d\zO$ whose colors are those of $\zO$ and operations are given by the formula 
    \begin{equation}
        \label{eqn:envelope_d}
        \zE^d\zO(\underline{x},\underline{y}) := \mathop{\colim}\limits_{H\in \cG(\underline{x},\underline{y})} \zO[H],
    \end{equation}
    where we use the notation $\zO[H] := \otimes_{v\in H} \zO(\inn_v,\out_v)$, and the transition morphisms in the colimit diagram are induced by composition in $\zO$.
    Composition in $\zE^d\zO$ is given by grafting of digraphs in $\cG$.  
\end{df}

\begin{prop}
    \label{prop:envelope_E^d}
    For every $\zV$-dioperad $\zO$ and $\zV$-properad $\zP$, there is an isomorphism of categories
    \begin{equation}
        \label{eqn:2-adjunction_E^d}
        \Prpd_\zV(\zE^d\zO,\zP)
        \cong 
        \Diop_\zV(\zO,\zU^d\zP)
    \end{equation}
    which is natural in $\zO$ and $\zP$. In particular, the properadic envelope functor $\zE^d$ is left adjoint to the forgetful functor $\zU^d$.
\end{prop}
\begin{proof}
    The second part of the statement, in the unenriched case, is a particular case of \cite[Theorem 12.1 and Lemma 12.8]{Yau_Johnson_PROPsBook} applied to the inclusion of pasting schemes $\mathsf{Gr}_\mathrm{di}^\uparrow \leqslant \mathsf{Gr}_{\mathrm{c}}^\uparrow$ from that of connected and simply-connected digraphs into all connected digraphs (see \cite[Example 12.10.(5)]{Yau_Johnson_PROPsBook}). 
    The proof therein consists in exhibiting the unit and counit for the adjunction; it is straightforward to extend these natural transformations to pseudonatural transformations, and then to reformulate this argument in the enriched setting. We leave these details to the reader. As a result, the adjunction between the underlying $1$-functors of $\zE^d$ and $\zU^p$ extends to a $2$-adjunction, in the sense of a natural isomorphism of categories of the form \eqref{eqn:2-adjunction_E^d}.
\end{proof}

\subsubsection{Symmetric monoidal envelope of a properad} 
\begin{df}
    Let $\zP$ be a $\zV$-properad. We define the \emph{symmetric monoidal envelope} of $\zP$ to be the symmetric monoidal $\zV$-category $\zE^p\zP$ whose objects are $\mathrm{Col}(\zP)$-profiles, i.e. finite sequences $\underline{x}=(x_i)_{i\in I}$ of colors of $\zP$, and whose morphism objects between $\underline{x}$ and $\underline{y}$ are given by the formula 
    \begin{equation}
        \label{eqn:def_monoidal_env}
        \zE^p\zP(\underline{x},\underline{y}) := \mathop{\colim}\limits_{A \in (\Fin_{{{}I^\mathrm{u}\amalg J^\mathrm{u}}/})^\simeq} \bigotimes_{a\in A} \zP\left(\underline{x}|_{I_a}; \underline{y}|_{J_a}\right)
    \end{equation}
    where the colimit is taken over the groupoid of finite sets under $I^\mathrm{u}\amalg J^\mathrm{u}$, and $I_a$ denotes the fiber of the map $I\to A$ over an element $a$ (and similarly for $J_a$).

    We can interpret morphisms as $A$-indexed unions of $\mathrm{Col}(\zP)$-colored corollas of the form $\mathfrak{c}_{\underline{x}|_{I_a}, \underline{y}|_{J_a}}$ together with a labelling of each corolla by an operation in $\zP$ of corresponding biprofile. When $A$ is a singleton, we will say that the morphism corresponding to the (unique) corolla is \emph{elementary}. Composition in the envelope $\zE^p\zP$ is obtained by first grafting the unions of corollas along the common colors, then contracting the inner edges while composing the operations accordingly, as illustrated in Figure \ref{fig:composition}.
    The symmetric monoidal structure on $\zE^{p}\zP$ is given by concatenation of sequences of colors.
\end{df}

\begin{figure}[!ht]
    \label{fig:composition}
    \begin{small}
    $
    \sigma=\begin{tikzpicture}
	\begin{pgfonlayer}{nodelayer}
		\node [style=small font circle] (0) at (-6.75, 0) {$\sigma_a$};
		\node [style=small] (1) at (-5.75, 4) {$x_2$};
		\node [style=small] (6) at (-8, -4) {$y_1$};
		\node [style=small] (7) at (-5.5, -4) {$y_3$};
		\node [style=small] (8) at (-3, -4) {$y_5$};
		\node [style=small font circle] (9) at (-4.25, 0) {$\sigma_b$};
		\node [style=small] (11) at (-7, 4) {$x_1$};
		\node [style=small] (12) at (-4.75, 4) {$x_3$};
		\node [style=small] (13) at (-3, 4) {$x_4$};
		\node [style=small] (15) at (-6.75, -4) {$y_2$};
		\node [style=small] (17) at (-4.25, -4) {$y_4$};
		\node [style=small font circle] (18) at (-5.5, 0) {$\sigma_c$};
		\node [style=small font circle] (19) at (-3, 0) {$\sigma_d$};
	\end{pgfonlayer}
	\begin{pgfonlayer}{edgelayer}
		\draw [in=90, out=-75] (0) to (7);
		\draw [in=90, out=-105] (0) to (6);
		\draw [in=270, out=90] (0) to (1);
		\draw [bend right=15] (12) to (9);
		\draw [in=90, out=-90] (9) to (17);
		\draw [bend right=15] (9) to (15);
		\draw [bend right=15, looseness=1.25] (9) to (13);
		\draw [bend right=15, looseness=1.25] (11) to (9);
		\draw [in=90, out=-45] (9) to (8);
	\end{pgfonlayer}
\end{tikzpicture}, \quad
    \tau=\begin{tikzpicture}
	\begin{pgfonlayer}{nodelayer}
		\node [style=small font circle] (18) at (2.5, 0) {$\tau_f$};
		\node [style=small] (19) at (1.75, 4) {$y_1$};
		\node [style=small] (23) at (3.25, 4) {$y_2$};
		\node [style=small] (24) at (1.5, -4) {$z_1$};
		\node [style=small] (26) at (3.5, -4) {$z_2$};
		\node [style=small font circle] (27) at (5.25, 0) {$\tau_e$};
		\node [style=small] (28) at (4.25, 4) {$y_3$};
		\node [style=small] (29) at (5.25, 4) {$y_4$};
		\node [style=small] (30) at (6.25, 4) {$y_5$};
	\end{pgfonlayer}
	\begin{pgfonlayer}{edgelayer}
		\draw [bend left=15] (18) to (26);
		\draw [bend right=15] (18) to (24);
		\draw [in=270, out=60] (18) to (23);
		\draw [in=270, out=120] (18) to (19);
		\draw (29) to (27);
		\draw [bend right=15, looseness=1.25] (27) to (30);
		\draw [bend right=15, looseness=1.25] (28) to (27);
	\end{pgfonlayer}
\end{tikzpicture},  \quad
    \begin{tikzpicture}
	\begin{pgfonlayer}{nodelayer}
		\node [style=small font circle] (0) at (-2.5, 1) {$\sigma_a$};
		\node [style=small] (1) at (-1.5, 4) {$x_2$};
		\node [style=small font circle] (5) at (0, 1) {$\sigma_b$};
		\node [style=small] (6) at (-2.75, 4) {$x_1$};
		\node [style=small] (7) at (-0.5, 4) {$x_3$};
		\node [style=small] (8) at (1.25, 4) {$x_4$};
		\node [style=small font circle] (11) at (-1.25, 1) {$\sigma_c$};
		\node [style=small font circle] (12) at (1.25, 1) {$\sigma_d$};
		\node [style=small font circle] (13) at (-2.25, -1.5) {$\tau_f$};
		\node [style=small] (16) at (-3, -4) {$z_1$};
		\node [style=small] (17) at (-1, -4) {$z_2$};
		\node [style=small font circle] (18) at (-0.25, -1.5) {$\tau_e$};
		\node [style=small] (20) at (4.5, 4) {$x_2$};
		\node [style=small font circle] (21) at (5, 0) {$\rho$};
		\node [style=small] (22) at (3.5, 4) {$x_1$};
		\node [style=small] (23) at (5.5, 4) {$x_3$};
		\node [style=small] (24) at (6.5, 4) {$x_4$};
		\node [style=small font circle] (25) at (4, 0) {$\sigma_c$};
		\node [style=small font circle] (26) at (6, 0) {$\sigma_d$};
		\node [style=small] (28) at (4, -4) {$z_1$};
		\node [style=small] (29) at (6, -4) {$z_2$};
		\node [style=none] (30) at (2.25, -1) {$\stackrel{\mathrm{contraction}}{\longrightarrow}$};
	\end{pgfonlayer}
	\begin{pgfonlayer}{edgelayer}
		\draw [in=270, out=90] (0) to (1);
		\draw [bend right=15] (7) to (5);
		\draw [bend right=15, looseness=1.25] (5) to (8);
		\draw [in=135, out=-60, looseness=1.25] (6) to (5);
		\draw [bend left=15] (13) to (17);
		\draw [bend right=15] (13) to (16);
		\draw [in=240, out=105] (13) to (0);
		\draw [in=210, out=30] (13) to (5);
		\draw [in=105, out=-105] (5) to (18);
		\draw [in=45, out=-60] (5) to (18);
		\draw [in=150, out=-45] (0) to (18);
		\draw [in=75, out=-90] (23) to (21);
		\draw [bend right=15, looseness=1.25] (21) to (24);
		\draw [bend right=15, looseness=1.25] (22) to (21);
		\draw [in=300, out=90] (29) to (21);
		\draw [in=240, out=90] (28) to (21);
		\draw [in=105, out=-90] (20) to (21);
	\end{pgfonlayer}
\end{tikzpicture}
    =\tau\circ\sigma
    $
    \end{small}
    \caption{Two morphisms $\sigma$ and $\tau$ in $\zE^p\zP$, represented as union of colored corollas, and their composition obtained by grafting and contraction of inner edges. Here the operation $\rho$ is obtained by the appropriate properadic composition of $\sigma_a$, $\sigma_b$, $\tau_e$ and $\tau_f$ in $\zP$.}
\end{figure}
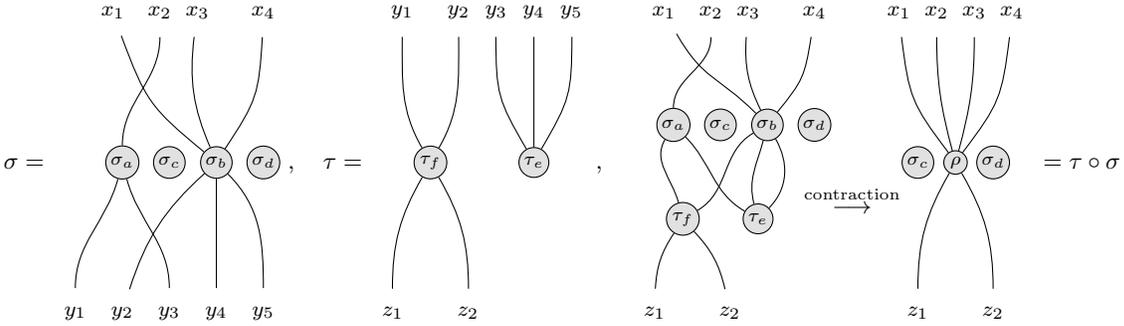

\begin{rk}
    Note that the construction of the symmetric monoidal envelope corresponds (in the unenriched setting) to that given in \cite[Section 2]{Beardsley-Hackney_labelled_cospan}, where properads are modelled in terms of Segal objects for an algebraic pattern of level graphs. Morever the authors enhance their construction to a biequivalence between the $2$-category of properads and that of \emph{labelled cospan categories} introduced by Steinebrunner in \cite{steinebrunner-surface_cat}. 
\end{rk}

For the purpose of this article, it will be enough to know that the symmetric monoidal envelope defines a $2$-functor $\zE^p\colon  \Prpd_\zV\to  \Cato_\zV $ which is left adjoint in a weak sense (see Remark \ref{rk:biadjunction}) to the forgetful functor (whence the terminology of envelope), as recorded in the following result.

\begin{prop}
    \label{prop:envelope_E^p}
    For every $\zV$-properad $\zP$ and every symmetric monoidal $\zV$-category $\zC$, there is an equivalence of categories
    \begin{equation}
        \label{eqn:biadjunction_E^p}
        \Cato_\zV(\zE^p\zP, \zC)
        \simeq 
        \Prpd_\zV(\zP,\zU^p\zC).
    \end{equation}
\end{prop}

\begin{rk}[Proposition \ref{prop:envelope_E^p} as a biadjunction]
    \label{rk:biadjunction}
    Note that $\zE^p$ and $\zU^p$, viewed as $1$-functors between the $1$-categories underlying $\Prpd_\zV$ and $\Cato_\zV$, do \emph{not} form an adjunction; in other words, the equivalence \eqref{eqn:biadjunction_E^p} is not a bijection on sets of objects. 
    The $2$-functors $\zE^p$ and $\zU^p$ are, however, part of a \emph{biadjunction}, which is to say that the equivalence of categories \eqref{eqn:biadjunction_E^p} is pseudonatural in $\zP$ and $\zC$. 
\end{rk}


\begin{rk}[$\infty$-categorical version of the symmetric monoidal envelope]
    \label{rk:env_properad_literature}
    An $\infty$-categorical version of Proposition \ref{prop:envelope_E^p} is established in \cite{Barkan-Steinebrunner_properads}; compare in particular formula \eqref{eqn:def_monoidal_env} with that of \cite[Lemma 3.1.7.]{Barkan-Steinebrunner_properads}. Instead of trying to deduce the $1$-categorical statement from its $\infty$-categorical counterpart, for completeness we prefer to give a more direct and elementary proof of the result.
\end{rk}

\begin{proof}[Proof of Proposition \ref{prop:envelope_E^p}]
    Given $\zP$ and $\zC$ as in the statement, we will define two functors
    \begin{equation*}
        \begin{tikzcd}
            \Gamma\colon \Cato_\zV(\zE^p\zP,\zC) 
            \ar[r, shift left=1] &
            \Prpd_\zV(\zP,\zU^p\zC) \colon \Upsilon
            \ar[l, shift left=1]
        \end{tikzcd}
    \end{equation*}
    and show they are quasi-inverse.

    We start by defining the functor $\Gamma$. Given a symmetric monoidal functor $F\colon \zE^p\zP\to \zC$, the morphism of properads $\Gamma F \colon \zP\to \zU^p\zC$ sends a color $x$ to $F((x))$, and an operation $\sigma \in \zP(\underline{x};\underline{y})$ to the following composite morphism in $\zC$ 
    \[
        \otimes_i F((x_i)) \cong F(\underline{x}) \xrightarrow{F(\sigma)}
        F(\underline{y}) \cong \otimes_j F((y_j)).
    \] 
    Given a morphism $\alpha\colon F\Rightarrow G$  in $\Cato_\zV(\zE^p\zP,\zC)$, we define $\Gamma\alpha$ componentwise by the formula $(\Gamma\alpha)_x = \alpha_{(x)}$; one then readily verifies that $\Gamma\alpha$ is polynatural, so that $\Gamma$ is well-defined.

    We turn to the definition of $\Upsilon$, which requires more care. Given an object $\phi\in \Prpd_\zV(\zP,\zU^p\zC)$, we define the functor $\Upsilon\phi$ as follows. It sends an object $\underline{x}=(x_i)_{i\in I}$ of the envelope $\zE^p\zP$ to $\otimes_i \phi(x_i)$ and an elementary morphism $f \in \zP(\underline{x};\underline{y})$ of the envelope to $\phi(f)$. Now suppose given a general morphism $f$ in $\zE^p\zP$, with associated map of finite sets  $p\colon I^\mathrm{u}\amalg J^\mathrm{u}\to A$ and labels $f_a\in \zP(\underline{x}|_{I_a},\underline{y}|_{J_a})$ for $a\in A$. To define $\phi(f)$, we first pick a linear order on $A$ and define a new linear order $\prec$ on $I^\mathrm{u}$ as follows: we say that $i \prec i'$ if $p(i) < p(i')$ (for the chosen order on $A$), or if $p(i)=p(i')$ and $i < i'$ in $I_a$ (for the order induced by $I$). Let $\sigma_I $ be the unique order-preserving permutation $I \to (I^\mathrm{u},\prec)$. We proceed similarly to obtain a linear order $\prec'$ on $J^\mathrm{u}$ and the associated order-preserving permutation $\tau \colon J\to (J^\mathrm{u},\prec')$. In particular, for every $a\in A$ the morphism $\phi(f_a)$ is of the form $\otimes_{i\in I_a} \phi(x_{\sigma(i)}) \to \otimes_{j\in J_a} \phi(y_{\tau(j)})$.
    The morphism $\phi(f)$ is then defined as the composite
    \[
        \otimes_{i\in I} \phi(x_i)\cong
        \otimes_{a\in A}\otimes_{i\in I_a} \phi(x_{\sigma(i)})
        \xrightarrow{\otimes_{a} \phi(f_a)}
        \otimes_{a\in A}\otimes_{j\in J_a} \phi(y_{\tau(j)}) \cong 
        \otimes_{j\in J} \phi(y_j),
    \]
    where the two isomorphisms are induced by the permutations $\sigma$ and $\tau^{-1}$ using the braiding of $\zC$. We leave the reader to check that this definition is independent of the choice of linear order on $A$, using in particular commutativity of the monoid of endomorphisms of $\1_\zC$.
    The symmetric monoidal structure on $\phi$ is the obvious one.
    Therefore we have defined the symmetric monoidal functor $\Upsilon \phi$.
    To specify $\Upsilon$ on morphisms, let $\alpha$ be a polynatural transformation between two morphisms of properads $\phi,\psi\colon \zP\to \zU^p\zC$. We define $\Upsilon\alpha \colon \Upsilon\phi\Rightarrow \Upsilon\psi$ by the formula $(\Upsilon\alpha)_{\underline{x}} = \otimes_i \alpha_{x_i}$; then $\Upsilon \alpha$ is natural, since $\alpha$ was polynatural, and it is monoidal by construction. This defines the functor $\Upsilon$.

    It is easy to verify that $\Gamma\circ \Upsilon=\id$. For the other composite, given a symmetric monoidal functor $F\in \Cato_\zV(\zE^p\zP,\zC)$, we have a monoidal natural isomorphism $ \Upsilon\circ \Gamma(F)\cong F$ given on components by the symmetric monoidal structure $\otimes  F((x_i))\cong F(\underline{x})$ of $F$; this natural isomorphism  is easily seen to be natural in $F$, so that we obtain a natural isomorphism of functors $\Upsilon\circ \Gamma \cong \id$, which concludes the proof.
\end{proof}

\subsubsection{Symmetric monoidal envelope of a dioperad}
By composing the above two constructions we obtain the $2$-functor
\[
    \zE := \zE^p\circ\zE^d \colon \Diop_\zV \to \Cato_\zV
\]
that we call the \emph{symmetric monoidal envelope of a dioperad}.
As a consequence of Propositions \ref{prop:envelope_E^d} and \ref{prop:envelope_E^p}, we deduce the following result.
\begin{cor}
    \label{cor:biadjunction_envelope}
    For every $\zV$-dioperad $\zO$ and every symmetric monoidal $\zV$-category $\zC$, there is an equivalence of categories
    \begin{equation*}
        \Cato_\zV(\zE\zO,\zC) \simeq 
        \Diop_\zV(\zO,\zU\zC).
    \end{equation*}
\end{cor}

Using formulas \eqref{eqn:envelope_d} and \eqref{eqn:def_monoidal_env}, the symmetric monoidal envelope $\zE\zO$ of a $\zV$-dioperad $\zO$ admits the following explicit description. Its objects are finite sequences $\underline{x}=(x_i)_{i\in I}$ of colors $x_i\in \zO$ and its operations are given by the formula
\begin{equation}
    \label{eqn:envelope_full}
    \zE\zO(\underline{x},\underline{y}) 
    = \mathop{\colim}\limits_{
        A \in (\Fin_{{{I^\mathrm{u}\amalg J^\mathrm{u}}/})^\simeq}
    }
    \bigotimes_{a\in A}
    \mathop{\colim}\limits_{
        G\in \cG(\underline{x}|_{I_a}, \underline{y}|_{J_a})
    }
    \zO[G].
\end{equation}
Analogously to the case of the envelope of a properad, we can interpret operations in \eqref{eqn:envelope_full} as (possibly disconnected) $\mathrm{Col}(\zO)$-colored digraphs, with a labelling of every vertex $v$ by an operation in $\zO$ of biprofile $(\inn_v,\out_v)$.

\section{Frobenius monoidal functors as morphisms of dioperads}
\label{sec:frob_monoidal_as_morphisms_dioperads}

In the same way that lax symmetric monoidal functors preserve algebras over operads and colax symmetric monoidal functors preserve coalgebras over operads, (separable) Frobenius monoidal functors preserve algebras over dioperads (resp. properads). In fact, in both cases such Frobenius structures are characterized by this property; this statement is made precise in Theorem \ref{thm:frob_as_diop_map}, which is the main result of this section.

\subsection{Frobenius monoidal structures}\label{sect:FrobMonFunct}

Recall that a \emph{lax symmetric monoidal structure} on a functor  $G\colon \zC\to \zD$ between symmetric monoidal $\zV$-categories is a pair $(\nabla, \nabla_0)$ consisting of a natural transformation $\nabla \colon G(-)\otimes G(-)\to G(-\otimes -)$ and a morphism $\nabla_0\colon \1_\zD\to G(\1_\zC)$ in $\zD$ such that $\nabla$ is associative, unital and compatible with the symmetry isomorphisms.
The notion of a \emph{colax symmetric  monoidal structure} $(\Delta, \Delta_0)$ is defined dually.

\begin{nota}
    Given $x_1,\dots, x_n\in \zC$, we let $\nabla_{x_1,\dots,x_n}\colon G(x_1) \otimes \dots \otimes G(x_n) \to G(x_1 \otimes \dots \otimes x_n)$ denote the morphism obtained by iterated composition of $\nabla$, which is well-defined up to a unique isomorphism.
\end{nota}

We recall the definition of a Frobenius monoidal functor.

\begin{df}
    Let $G\colon \zC\to \zD$ be a functor between symmetric monoidal $\zV$-categories.
    A \emph{Frobenius monoidal structure}\footnote{A Frobenius monoidal structure in the above sense should arguably be called \emph{commutative}; however we remove this adjective for simplicity, as we will not consider the non-commutative version in this paper.} on $G$ is the data of a lax symmetric monoidal structure $(\nabla,\nabla^0)$ and a colax symmetric monoidal structure $(\Delta,\Delta^0)$ satisfying the following condition: for any objects $A$, $B$ and $C$ in $\zC$, the following diagram commutes
            \begin{equation}
                \label{eqn:Frobenius_relation} 
                \tag{\text{Frobenius relation}}
                \begin{tikzcd}[column sep = large]
                    {GA\otimes G(B\otimes C)} & {G(A\otimes B\otimes C)} \\
                    {GA\otimes GB\otimes GC} & {G(A\otimes B)\otimes GC}
                    \arrow["{\nabla_{A,B\otimes C}}", from=1-1, to=1-2]
                    \arrow["{\id_A \otimes \Delta_{B,C}}"', from=1-1, to=2-1]
                    \arrow["{\Delta_{A\otimes B,C}}", from=1-2, to=2-2]
                    \arrow["{\nabla_{A,B}\otimes \id_C}"', from=2-1, to=2-2]
                \end{tikzcd}
            \end{equation}
    If moreover $\nabla_{X,Y}\circ \Delta_{X,Y} = \id_{G(X\otimes Y)}$, we say that the Frobenius monoidal structure is \emph{separable}.
\end{df}

We now define what it means for a natural transformation between Frobenius monoidal functors to be compatible with these structures.
\begin{df}
    \label{df:frob_natural_transformation}
    Let $F,G\colon \zC\to \zD$ be two Frobenius monoidal functors between symmetric monoidal $\zV$-categories. A natural transformation $\alpha\colon F\Rightarrow G$ is said to be \emph{Frobenius monoidal} if it is both monoidal with respect to the lax symmetric monoidal structures on $F$ and $G$ and comonoidal with respect to their colax symmetric monoidal structures. In more explicit terms, this means that for every $x,y\in \zC$ the following four squares commute:
    \begin{equation*}
        \begin{tikzcd}
        	{F(x)\otimes F(y)} & {G(x)\otimes G(y)} \\
        	{F(x\otimes y)} & {G(x\otimes y)}
        	\arrow["{\alpha_{x}\otimes \alpha_{y}}", from=1-1, to=1-2]
        	\arrow["{\nabla^F_{x,y}}"', from=1-1, to=2-1]
        	\arrow["{\nabla^G_{x,y}}", from=1-2, to=2-2]
        	\arrow["{\alpha_{x\otimes y}}"', from=2-1, to=2-2]
        \end{tikzcd}
        \qquad
        \begin{tikzcd}
        	{F(x\otimes y)} & {G(x\otimes y)} \\
        	{F(x)\otimes F(y)} & {G(x)\otimes G(y)}
        	\arrow["{\alpha_{x\otimes y}}", from=1-1, to=1-2]
        	\arrow["{\Delta^F_{x,y}}"', from=1-1, to=2-1]
        	\arrow["{\Delta^G_{x,y}}", from=1-2, to=2-2]
        	\arrow["{\alpha_{x}\otimes \alpha_{y}}"', from=2-1, to=2-2]
        \end{tikzcd}
    \end{equation*}
    \begin{equation*}
        \begin{tikzcd}
        	& {\1_\zD} & \\
        	{F(\1_\zC)} && {G(\1_\zC)}
        	\arrow["{\nabla_0^F}"', from=1-2, to=2-1]
        	\arrow["{\nabla_0^G}", from=1-2, to=2-3]
        	\arrow["{\alpha_{\1_\zC}}", from=2-1, to=2-3]
        \end{tikzcd} 
        \qquad 
        \begin{tikzcd}
        	{F(\1_\zC)} && {G(\1_\zC)} \\
        	& {\1_\zD}
        	\arrow["{\alpha_{\1_\zC}}", from=1-1, to=1-3]
        	\arrow["{\Delta_0^F}"', from=1-1, to=2-2]
        	\arrow["{\Delta_0^G}", from=1-3, to=2-2]
        \end{tikzcd}
    \end{equation*}
\end{df}

Symmetric monoidal $\zV$-categories, Frobenius monoidal functors and Frobenius monoidal natural transformations assemble into a $2$-category $\Frob^\otimes_\zV$.

\subsection{Characterization of Frobenius monoidal functors}
The goal of this section is to establish the following result.

\begin{thm}
    \label{thm:frob_as_diop_map}
    Let $\zC, \zD$ be symmetric monoidal $\zV$-categories. Then there are equivalences of categories
    \begin{enumerate}
        \item 
            between morphisms of dioperads $\zU\zC\to \zU\zD$ and Frobenius monoidal functors $\zC\to \zD$
            \[
                \Diop_\zV(\zU\zC,\zU\zD)\simeq \mathbf{Frob}^{\otimes}_\zV(\zC,\zD),
            \]
        \item 
            between morphisms of properads $\zU^p\zC\to \zU^p\zD$ and separable Frobenius monoidal functors $\zC\to \zD$
            \[
                \Prpd_\zV(\zU^p\zC,\zU^p\zD)\simeq \mathbf{Frob}^{\otimes,\mathrm{sep}}_\zV(\zC,\zD).
            \]
    \end{enumerate}
\end{thm}

\begin{rk}
    This theorem is somewhat similar in spirit to the main result of \cite{McCurdy-Street}, in which the authors show that Frobenius monoidal functors preserve equations where both sides are values of connected and simply connected string diagrams, while separable ones preserve all connected string diagrams.
\end{rk}

\begin{proof}
    We start by proving the second claim. 
    \paragraph{Proof of the properadic statement.}
    By Proposition \ref{prop:envelope_E^p}, there is an equivalence of categories
    \begin{equation*}
        \Prpd_\zV(\zU^p\zC, \zU^p\zD) \simeq \Cato_\zV(\zE^p\zU^p\zC,\zD)
    \end{equation*}
    so that the proof reduces to producing mutually inverse functors
    \begin{equation}
        \label{eqn:phi_psi}
        \begin{tikzcd}
            \Phi\colon \Cato_\zV(\zE^p\zU^p\zC,\zD) \ar[r,shift left=1] &
            \mathbf{Frob}^{\otimes,\mathrm{sep}}_\zV(\zC,\zD)\colon \Psi. \ar[l,shift left=1] 
        \end{tikzcd}
    \end{equation}

    \paragraph{Step 1: description of $\zE^p\zU^p\zC$.}
    By formula \eqref{eqn:def_monoidal_env}, we see that the symmetric monoidal category $\zE^p\zU^p\zC$ has objects given by sequences of objects in $\zC$ and operations by 
    \begin{equation*}
        \zE^p\zU^p\zC(\underline{x},\underline{y}) 
        = \mathop{\colim}\limits_{A \in (\Fin_{{I^\mathrm{u}\amalg J^\mathrm{u}}/})^\simeq} \bigotimes_{a\in A} \zP\left( \mathop{\otimes}\limits_{i\in I_a} x_i, \mathop{\otimes}\limits_{j\in J_a} y_j\right).
    \end{equation*}
    In other words, morphisms can be represented as unions of corollas with input leaves labelled with the $x_i$ and output leaves with the $y_j$, and vertices by morphisms in $\zC$ from the tensor product of the inputs to that of the outputs.
    The tensor structure on $\zE^p\zU^p\zC$ is given by concatenation, with the empty sequence $\emptyset$ as unit.

    In particular, given any object $\underline{x}\in \zE^p\zU^p\zC$, we can define two special morphisms $\Delta_{\underline{x}}$ and $\nabla_{\underline{x}}$ as follows. 
    The morphism $\nabla_{\underline{x}}\colon \underline{x}\to \otimes_i x_i$ is given by the corolla $\mathfrak{c}_{\underline{x},\otimes_i x_i}$ labelled with the morphism $\id_{\otimes_i x_i}$ in $\zC$. 
    Similarly the morphism $\Delta_{\underline x} \colon \otimes_i x_i \to \underline{x}$ is given by the corolla $\mathfrak{c}_{\otimes_i x_i,\underline{x}}$ labelled with $\id_{\otimes_i x_i}$. Using the description of composition in the monoidal envelope in terms of grafting and contraction of corollas, one readily sees that
    \begin{equation}
        \label{eqn:separability}
        \nabla_{x,y} \circ \Delta_{x,y} = \id_{(x,y)}
    \end{equation}
    and that the following square commutes
    \begin{equation}
        \label{eqn:frobenius_envelope}
        \begin{tikzcd}[column sep = large]
            (x \otimes y, z) \ar[r,"{\nabla_{x \otimes y, z}}"]\ar[d,"{(\Delta_{x,y}, \id_z)}"left]	& x \otimes  y \otimes z \ar[d,"{\Delta_{x,y \otimes z}}"] \\
            (x,y,z) \ar[r,"{(\id_x, \nabla_{y,z})}"]		& (x, y \otimes z)
        \end{tikzcd} 
    \end{equation}
    for any $x,y,z\in \zC$. There is also a morphism $\nabla_0\colon \emptyset\to \1_\zC$ that satisfies a unitality property with respect to $\nabla$, and dually a counital morphism $\Delta_0\colon \1_\zC\to \emptyset$ with respect to $\Delta$.

    \paragraph{Step 2: construction of the functors $\Phi$ and $\Psi$.}
    We now define the two functors of \eqref{eqn:phi_psi}, starting with $\Phi$. Given a symmetric monoidal functor $F\colon \zE^p\zU^p\zC\to \zD$, the functor $\Phi(F)\colon \zC\to \zD$ is simply given by restricting $F$ along the unit $\zC\to \zE^p\zU^p\zC$. Using the relations  \eqref{eqn:separability} and \eqref{eqn:frobenius_envelope} computed in $\zE^p\zU^p\zC$, we see that the following definitions, in which the isomorphisms are given by the symmetric monoidal structure of $F$,
    \begin{align}
        \label{eqn:def_frob_structure_Phi(F)}
        \nabla^{\Phi(F)}_{x,y} &\colon F(x)\otimes F(y) \xrightarrow{\cong} F(x,y) \xrightarrow{F(\nabla_{(x,y)})} F(x\otimes y)\\ 
        \notag
        \nabla_0^{\Phi(F)} &\colon \1_\zD \xrightarrow{\cong}F(\emptyset) \xrightarrow{F(\nabla_0)} F(\1_\zC) \\ 
        \notag
        \Delta^{\Phi(F)}_{x,y} &\colon F(x\otimes y) \xrightarrow{F(\Delta_{(x,y)})} F(x,y) \xrightarrow{\cong} F(x)\otimes F(y)\\ 
        \notag
        \Delta^{\Phi(F)}_0 &\colon F(\1_\zC) \xrightarrow{F(\Delta_0)} F(\emptyset) \xrightarrow{\cong} \1_\zD,  
    \end{align}
    endow $\Phi(F)$ with a separable Frobenius monoidal structure. At the level of morphisms, given two symmetric monoidal functors $F,G\in \Cato_\zV(\zE^p\zU^p\zC,\zD)$ and a monoidal natural transformation $\alpha\colon F\Rightarrow G $ between them, the natural transformation $\Phi\alpha$ is defined in components via the formula $(\Phi\alpha)_{x}=\alpha_{(x)}$. One then easily verifies that it is Frobenius monoidal, using that $\alpha$ is monoidal; hence we have constructed the functor $\Phi$.

    We turn to defining the functor $\Psi$. 
    Given a  functor $G\colon \zC\to \zD$ endowed with a separable Frobenius monoidal structure $(\nabla^G, \nabla_0^G, \Delta^G, \Delta^G_0)$, we define the functor $\Psi(G)\colon \zE^p\zU^p\zC\to \zD$ as follows. At the level of objects, it sends a sequence $\underline{x}=(x_i)_{i\in I}$ to $\otimes_i G(x_i)$; at the level of morphisms, $\Psi(G)$ sends an elementary morphism $\underline{x}\to \underline{y}$ labelled with $f\colon \otimes_i x_i \to \otimes_j y_j$, to the composite $\Delta^G_{y_1,\dots, y_m}\circ G(f) \circ \nabla^G_{x_1,\dots, x_n}$. To define $\Psi(G)$ on general morphisms, we proceed as in the proof of Proposition \ref{prop:envelope_E^p} (for the functor $\Upsilon$) to reduce to elementary morphisms.
    It follows from the separability assumption that this assignment  $\Psi(G)$ defines a functor, which carries an obvious symmetric monoidal structure. It remains to define $\Psi$ on morphisms. To do so, consider a Frobenius monoidal natural transformation $\alpha\colon F\Rightarrow G$ between two separable Frobenius monoidal functors. We define $\Psi\alpha$  via the formula $(\Psi\alpha)_{\underline{x}} = \otimes_i \alpha_{x_i}\colon \otimes_i F(x_i)\to \otimes_i G(x_i)$; naturality follows from $\alpha$ being Frobenius monoidal, while the fact that it is monoidal is obvious. 

    We have defined the two functors $\Psi$ and $\Phi$. One easily verifies that the composite  $\Phi\circ \Psi = \id$. One the other hand, there is a natural isomorphism $\Psi\circ\Phi\cong \id$ given for every symmetric monoidal functor $F\colon \zE^p\zU^p\zC\to \zD$ by its associated symmetric monoidal structure $\Psi\circ\Phi(\underline{x})=\otimes_i F(x_i)\cong F(\underline{x})$, hence the first statement is proven.

    \paragraph{Proof of the dioperadic statement.}
    We now turn to the dioperadic case. The strategy of proof is similar to the properadic case, the main difference being that the description of the $\zE\zU\zC$ is more involved in the dioperadic case. 
    By Corollary \ref{cor:biadjunction_envelope}, there is a equivalence of categories
    \begin{equation*}
        \Diop_\zV(\zU\zC, \zU\zD) \simeq \Cato_\zV(\zE\zU\zC,\zD).
    \end{equation*}
    We will produce mutually inverse functors
    \begin{equation*}
        \label{eqn:theta_lambda}
        \begin{tikzcd}
            \Theta\colon \Cato_\zV(\zE\zU\zC,\zD) \ar[r,shift left=1] &
            \mathbf{Frob}^{\otimes}_\zV(\zC,\zD)\colon \Lambda \ar[l,shift left=1]
        \end{tikzcd}
    \end{equation*}
    
    \paragraph{Step 1: description of $\zE\zU\zC$.}
    Recall that the objects of $\zE\zU\zC$ are sequences $\underline{x}=(x_i)_{i\in I}$ of objects $x_i\in \zC$, with symmetric monoidal structure given by concatenation. By equation \eqref{eqn:envelope_full}, operations in $\zE\zU\zC$ are given by (possibly disconnected) digraphs that are colored by sequences in $\zC$, together with a labelling of every vertex $v$ by a morphism in $\zC$ of the form $\otimes \,\inn_v \to \otimes\, \out_v$. 

    As in the proof of the properadic statement, for every object $\underline{x}\in \zE\zU\zC$ we can define two associated morphisms $\nabla_{\underline{x}}\colon \underline{x} \to \otimes_i x_i$ and $\Delta_{\underline{x}}\colon \otimes_i x_i \to \underline{x}$. 
    One then easily verifies that the Frobenius relation \eqref{eqn:frobenius_envelope} is satisfied, using that contraction of subtrees correspond to morphisms in the extension category $\cG$ of Definition \ref{df:extension_category}. 

    \paragraph{Step 2: construction of the functors $\Theta$ and $\Lambda$.}
    We first define the functor $\Theta$, in a completely similar way to the functor $\Phi$ in the properadic case. More precisely, let $F\colon \zE\zU\zC\to \zD$ a symmetric monoidal functor. The functor $\Theta(F)\colon \zC\to \zD$ is obtained by restricting $F$ along the unit $\zC\to \zE\zU\zC$. As before, using that the Frobenius relation \eqref{eqn:frobenius_envelope} holds in $\zE\zU\zC$, we obtain that defining 
    $\nabla^{\Theta(F)}_{x,y}$, $\nabla_0^{\Theta(F)}$, $\Delta^{\Theta(F)}_{x,y}$ and $\Delta_0^{\Theta(F)}$ 
    using the same formulas as equations \eqref{eqn:def_frob_structure_Phi(F)} yields a Frobenius monoidal structure on $\Theta(F)$. The definition of $\Theta$ on morphisms is as before.
    
    We turn to defining the functor $\Lambda$, which requires more care. Let $G\colon \zC\to \zD$ be a Frobenius monoidal functor. At the level of objects, the functor $\Lambda(G)\colon \zE\zU\zC\to \zD$ sends a sequence of objects $\underline{x}=(x_i)_{i\in I}$ to $\otimes_i G(x_i)$. Now let $f\in \zE\zU\zC(\underline{x},\underline{y})$ be a morphism represented by a colored and labelled digraph $K$. Since $K$ contains no directed cycles, we may choose a total order $v_1 < \dots < v_n$ on vertices of $K$ for which there is no edge from $v_i$ to $v_j$ when $i < j$; we then define $\Lambda(G)(f)$ as the composite
    \begin{equation}
        \label{eqn:def_Lambda(G)f}
        \left( \id\otimes (\Delta_{\out_{v_1}}\circ G(f_{v_1}) \circ \nabla_{\inn_{v_1}}) \otimes \id \right)\circ\dots\circ 
        \left( \id\otimes (\Delta_{\out_{v_n}}\circ G(f_{v_n}) \circ \nabla_{\inn_{v_n}}) \otimes \id \right).
    \end{equation}
    One can show that this definition is independent of the choice of the ordering. 
    To prove that the expression \eqref{eqn:def_Lambda(G)f} is invariant under substitution by simply-connected digraphs, we can reduce to the case of substituting a vertex by the grafting of two corollas along a single edge, for which the result follows from the Frobenius relation \eqref{eqn:frobenius_envelope}.
    Finally, it is straightforward to verify that the assignment $\Lambda(G)$ described above indeed defines a functor $\zE\zU\zC\to \zD$ (note that, as opposed to the above properadic case, the separability assumption is not needed anymore to ensure functoriality), which moreover admits an obvious symmetric monoidal structure. We extend the assignment $\Lambda$ into a functor exactly as we did for $\Psi$ above.

    As for the properadic statement, it is easy to show that  $\Theta\circ \Lambda = \id$ and that there is a natural monoidal isomorphism $\Lambda\circ \Theta(F)\cong F$ given by the associated symmetric monoidal structure of $F$, so that $\Lambda$ and $\Theta$ are mutually inverse functors. 
    This concludes the proof of Theorem \ref{thm:frob_as_diop_map}.
\end{proof}

\subsubsection{Twisted version}
We conclude this section by introducing a slight variant of Theorem \ref{thm:frob_as_diop_map} in which cooperations are twisted by an invertible object. This extra level of generality will be important in the application of next section.

\begin{df}\label{df:TwistedDioperad}
     Given an invertible object $d$ in $\zV$ and a dioperad (or properad) $\zO$, we define its associated \emph{$d$-twisted dioperad (or properad)} $\zO\{d\}$ as having the same colors and operations given by the formula
    \begin{equation*}
        \zO\{d\}(\underline{x},\underline{y}) := \zO(\underline{x},\underline{y}) \otimes d^{\otimes (n-1)}
    \end{equation*}
    where $n$ is the cardinality of $\underline{y}$. We stress that the twist is asymmetric, as it depends on the number of outputs, and not on that of inputs. 
\end{df}

\begin{nota}
    For $n \in \mathbb{Z}$, we will use an additive notation and abbreviate $\zO\{d^{\otimes n}\}$ to $\zO\{nd\}$.
\end{nota}

In the context of symmetric monoidal $\zV$-categories, we make the following definition.
\begin{df}[Degree of morphisms]
    \label{df:degree}
    Let $\zC$ be a symmetric monoidal $\zV$-category, $x,y$ be two objects, $d$ an invertible object in $\zV$ and $n\in \bZ$. The object of \emph{degree $nd$ morphisms from $x$ to $y$} in $\zC$ is defined as $\zC(x,y)\otimes d^{\o n}$. 

    As in notation \ref{par:enriched_cats}, we will refer to a morphism $f\colon \1_\zV \to \zC(x,y)\otimes d^{\o n}$ as a degree $nd$ morphism from $x$ to $y$; moreover it will be convenient to abuse notation by writing it $f\colon x\to y$.
    Note that composition adds degrees, in the sense that given two composable morphisms of respective degrees $nd$ and $md$, with $m,n\in \bZ$, their composite is of degree $(n+m)d$.
\end{df}


\begin{df}[Twisted Frobenius monoidal structures]
    \label{df:twisted_Frob_monoidal_structure}
    Let $G\colon \zC\to \zD$ be a functor between two symmetric monoidal $\zV$-categories and let $d\in \bZ$ be an invertible object.
    A \emph{$d$-twisted colax symmetric monoidal structure} on $G$ is the data of a coassociative natural transformation $\Delta_{x,y}\colon G(x\otimes y)\to G(x)\otimes G(y)$ of degree $d$ with a morphism $\Delta \colon G(\1_\zC) \to \1_\zD$ of degree $-d$ that is a counit for $\Delta$.

    We define a \emph{$d$-twisted Frobenius monoidal structure} on $G$ to be a lax symmetric monoidal structure on $G$ together with a $d$-twisted colax symmetric monoidal structure that satisfy the  \eqref{eqn:Frobenius_relation} (in which the vertical morphisms now have degree $d$).
    Given two symmetric monoidal $\zV$-categories $\zC$ and $\zD$, $d$-twisted Frobenius monoidal functors from $\zC$ to $\zD$ together with Frobenius monoidal natural transformations (defined as in Definition \ref{df:frob_natural_transformation}) form a category denoted $\Frob^{\otimes,d}_\zV(\zC,\zD)$.
\end{df}

The proof of Theorem \ref{thm:frob_as_diop_map} generalizes easily to take into account twists, hence we have the following result.
\begin{prop}
    \label{prop:twisted_Frob_monoidal}
    Let $\zC, \zD$ be symmetric monoidal $\zV$-categories and $d\in \zV$ an invertible object. Then there is a equivalence of categories
    \[
        \Diop_\zV(\zU\zC,\zU\zD\{d\})\simeq \mathbf{Frob}^{\otimes, d}_\zV(\zC,\zD)
    \]
    between morphisms of dioperads $\zU\zC\to \zU\zD\{d\}$ and $d$-twisted Frobenius monoidal functors $\zC\to \zD$.
    A similar statement holds for morphisms of properads $\zU^p\to \zU^p\zD\{d\}$ and separable $d$-twisted Frobenius monoidal functors.
\end{prop}

\section{Twisted Frobenius monoidal structures from orientations}
\label{sec:orientation}
In this section, we prove that the data of an orientation on a symmetric monoidal adjunction induces a twisted Frobenius monoidal structure on the right adjoint.

\subsection{The categorical setting}
\label{sub:frob_mon}

First we fix notation and hypothesis under which we work in this section. The following definitions are inspired by \cite[Section 2.1]{PTVV}.

\begin{setting}
    \label{par:setting}
    Let $\zC, \zD$ be symmetric monoidal $\zV$-categories, together with a pair of adjoint functors
    \begin{equation*}
        \label{equ:FG_adjoint}
        \begin{tikzcd}
            F\colon \zD  & \zC\colon G,
            \arrow[""{name=0, anchor=center, inner sep=0}, shift left=2, from=1-1, to=1-2]
            \arrow[""{name=1, anchor=center, inner sep=0}, shift left=2, from=1-2, to=1-1]
            \arrow["\dashv"{anchor=center, rotate=-90}, draw=none, from=0, to=1]
        \end{tikzcd}
    \end{equation*}
    in the sense of enriched categories.
    We further assume that
    \begin{enumerate}
        \item the functor $F$ is endowed with a symmetric monoidal structure,
        \item the $\zV$-categories $\zC$ and $\zD$ are both \emph{rigid}, that is, every object is dualizable. 
    \end{enumerate}
From the rest of this section, we fix an invertible object $d$  in $\zD$. 
\end{setting}

\begin{nota}
    \label{nota:duals}
    The duals in $\zC$ and in $\zD$ will be denoted respectively $(-)^\vee$ and $(-)^*$.
\end{nota}

We recall easy consequences of the above assumptions.
First, the functor $G$ inherits a lax symmetric monoidal structure given as 
\begin{equation}
    \label{eqn:nabla}
    \nabla_{X,Y}\colon G(X\otimes Y) \stackrel{u}{\too} GF(GX\otimes GY)\cong  G(FGX\otimes FGY) \stackrel{G(c\otimes c)}{\too} GX\otimes GY.
\end{equation}
where $u\colon \id\to GF$ and $c\colon FG\to \id$ denote the unit and counit of the adjunction. The morphism $\nabla^0\colon \1_\zD\to G(\1_\zC)$ is given by the transpose of the isomorphism $F(\1_\zD)\xrightarrow{\simeq} \1_\zC$.

The second consequence of setting \ref{par:setting} is the existence of a \emph{left} adjoint $G'$ to $F$, given by the formula
\[
    G'(X) := G(X^\vee)^*.
\]

\begin{df}
    \label{df:non-deg}
    A degree $nd$ morphism $b\colon A\otimes B \to \1_\zD$ in $\zD$ is \emph{non-degenerate} if the induced morphism
    \begin{equation*}
        b^\sharp := (b\otimes\id)\circ(\id\otimes\coev )\colon A\cong A\otimes \1_\zD \too A\otimes  B \otimes B^*\too \1_\zD\otimes  B^*\cong B^*
    \end{equation*}
    of degree $nd$ is an isomorphism.
\end{df} 

\begin{df}
    \label{df:orientation}
    A \emph{$d$-orientation} on $G$ is a morphism $\xi\colon G(\1_\zC)\to \1_\zD$ of degree $-d$ in $\zD$ such that for every $X$ in $\zC$, the induced morphism  of degree $-d$
    \begin{equation*}
        \beta  \colon G(X)\otimes G(X^\vee) 
        \xrightarrow{\nabla_{X,X^\vee}}
        G(X\otimes X^\vee) 
        \xrightarrow{G(\ev_X)} G(\1_\zC) 
        \xrightarrow{\xi} \1_\zD
    \end{equation*}
    is non-degenerate (in the sense of Definition \ref{df:non-deg}). 
\end{df}
In particular, an orientation $\xi$ induces a natural isomorphism of degree $-d$ 
\[
    \alpha \colon G \stackrel{\cong}{\too} G'
\]
between the left and right adjoints to F.

\begin{df}
    \label{df:duality_context}
    A \emph{duality context of degree $d$} (or simply a \emph{$d$-duality context}) is the data of a pair of adjoint functors $F \colon \zD \rightleftharpoons\zC\colon G$ satisfying the assumptions of Setting \ref{par:setting}, together with a $d$-orientation on the right adjoint $G$.
\end{df}

The following statement is the main result of this section.

\begin{thm}
    \label{thm:main}
    Any duality context of degree $d$ induces a $d$-twisted Frobenius monoidal structure on the right adjoint $G\colon \zC\to \zD$, extending its lax symmetric monoidal structure given by \eqref{eqn:nabla}.
\end{thm}

Using Theorem \ref{thm:frob_as_diop_map}, we obtain the following immediate consequence.
\begin{cor}\label{cor:main}
    Let $\zO$ be a dioperad and suppose given an $d$-orientation of $G$. Then $G$ sends $\zO$-algebras in $\zC$ to $\zO\{-d\}$-algebras in $\zD$. 
\end{cor}

\paragraph*{Strategy of proof of Theorem \ref{thm:main}.}
To prove Theorem \ref{thm:main}, we define a $d$-twisted colax symmetric monoidal structure $\Delta$ on $G$ using the orientation $\xi$ and the induced natural isomorphism $\alpha$. Namely, we let $\Delta$ denote the natural transformation of degree $d$
\begin{equation}
    \label{eqn:delta}
    \Delta_{X,Y}\colon 
    \begin{tikzcd}[column sep=large]
        {G(X\otimes Y)} & {G'(X\otimes Y)} & {G'X \otimes G'Y} & {GX\otimes GY}
        \arrow["{\alpha_{X\otimes Y}}", "\cong" below, from=1-1, to=1-2]
        \arrow["{\nabla^*_{Y^\vee, X^\vee}}", from=1-2, to=1-3]
        \arrow["{\alpha^{-1}_X\otimes\alpha^{-1}_Y}", "\cong" below, from=1-3, to=1-4]
    \end{tikzcd}
\end{equation} 
and $\Delta_0$ be the degree $-d$ morphism $G(\1_\zC)\xrightarrow{\alpha_\1} G'(\1_\zC) \xrightarrow{(\nabla_0)^*} \1_\zD$. 

It is straightforward to verify that $(\Delta,\Delta_0)$ defines a $d$-twisted colax symmetric monoidal structure on $G$.
The core of the proof consists in proving the  \eqref{eqn:Frobenius_relation}. To avoid drawing many commutative diagrams and make the proof more transparent, we will use a graphical calculus of string diagrams, which we introduce in Section \ref{sub:graphical_calculus}. The proof of the main theorem is then carried out in Section \ref{sub:proof_main}.

\subsection{Graphical calculus}
\label{sub:graphical_calculus}

\begin{nota}
    In the following, we use shortened notations for tensors, as explained in Section \ref{nota:tensor}. Also recall our convention \ref{nota:duals} on notations of duals. 
\end{nota}

We introduce a graphical representation of the morphisms appearing in the definitions of $\nabla$ and $\Delta$ (see their definitions \eqref{eqn:nabla} and \eqref{eqn:delta}). This graphical calculus is an ad hoc variation on the standard one of string diagrams.

To describe it, first observe that up to unitors and commutators, all such morphisms (with possibly some degree $\ell d$) are of the form 
\[
    f\colon G(X_{11} \cdots  X_{1n_1})^\star \cdots  G(X_{m1} \cdots  X_{mn_m})^\star \too
    G(Y_{11} \cdots  Y_{1p_1})^\star \cdots  G(Y_{q1} \cdots  Y_{qp_q})^\star 
\]
for some objects $X_i, Y_j\in \zD$, where $A^\star$ means either $A$ or $A^*$.
Such a morphism will be represented as a string diagram of the form
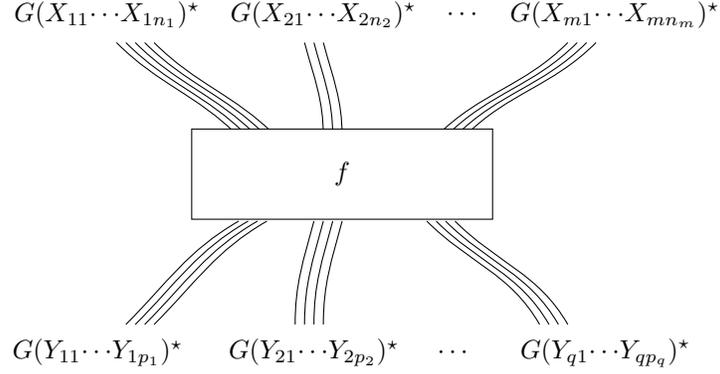
\begin{equation*}
     \begin{tikzpicture}
	\begin{pgfonlayer}{nodelayer}
		\node [style=none] (0) at (-5.75, 3.75) {};
		\node [style=none] (1) at (-5.5, 3.75) {};
		\node [style=none] (2) at (-5.25, 3.75) {};
		\node [style=none] (3) at (-5, 3.75) {};
		\node [style=none] (4) at (-4.75, 3.75) {};
		\node [style=none] (5) at (-0.75, 3.75) {};
		\node [style=none] (6) at (-0.5, 3.75) {};
		\node [style=none] (7) at (-0.25, 3.75) {};
		\node [style=none] (8) at (6.25, 3.75) {};
		\node [style=none] (9) at (6.5, 3.75) {};
		\node [style=none] (10) at (6.75, 3.75) {};
		\node [style=none] (11) at (7, 3.75) {};
		\node [style=none] (12) at (-5.5, -3.75) {};
		\node [style=none] (13) at (-5.25, -3.75) {};
		\node [style=none] (14) at (-5, -3.75) {};
		\node [style=none] (15) at (-4.75, -3.75) {};
		\node [style=none] (16) at (-1, -3.75) {};
		\node [style=none] (17) at (-0.75, -3.75) {};
		\node [style=none] (18) at (-0.5, -3.75) {};
		\node [style=none] (19) at (5.5, -3.75) {};
		\node [style=none] (20) at (5.75, -3.75) {};
		\node [style=none] (21) at (6, -3.75) {};
		\node [style=none] (22) at (6.25, -3.75) {};
		\node [style=none] (27) at (-2.5, -1) {};
		\node [style=none] (28) at (-2.25, -1) {};
		\node [style=none] (29) at (-2, -1) {};
		\node [style=none] (30) at (-1.75, -1) {};
		\node [style=none] (31) at (-0.5, -1) {};
		\node [style=none] (32) at (-0.25, -1) {};
		\node [style=none] (33) at (0, -1) {};
		\node [style=none] (34) at (2.5, -1) {};
		\node [style=none] (35) at (2.75, -1) {};
		\node [style=none] (36) at (3, -1) {};
		\node [style=none] (37) at (3.25, -1) {};
		\node [style=none] (38) at (-2.5, 1.25) {};
		\node [style=none] (39) at (-2.5, 1.25) {};
		\node [style=none] (40) at (-2.25, 1.25) {};
		\node [style=none] (41) at (-2, 1.25) {};
		\node [style=none] (42) at (-1.75, 1.25) {};
		\node [style=none] (43) at (-1.5, 1.25) {};
		\node [style=none] (44) at (-0.25, 1.25) {};
		\node [style=none] (45) at (0, 1.25) {};
		\node [style=none] (46) at (0.25, 1.25) {};
		\node [style=none] (47) at (2.75, 1.25) {};
		\node [style=none] (48) at (3, 1.25) {};
		\node [style=none] (49) at (3.25, 1.25) {};
		\node [style=none] (50) at (3.5, 1.25) {};
		\node [style=none] (51) at (0.25, -1) {};
		\node [style=none] (52) at (-0.25, -3.75) {};
		\node [style=none] (56) at (-6, 4.5) {$G(X_{11}\!\cdots\! X_{1n_1})^\star$};
		\node [style=big rectangle] (57) at (0.25, 0.25) {$f$};
		\node [style=none] (58) at (-0.25, 4.5) {$G(X_{21}\!\cdots \!X_{2n_2})^\star$};
		\node [style=none] (59) at (7.5, 4.5) {$G(X_{m1}\!\cdots \!X_{mn_m})^\star$};
		\node [style=none] (60) at (3.5, 4.5) {$\cdots$};
		\node [style=none] (61) at (-6.25, -4.5) {$G(Y_{11}\!\cdots\! Y_{1p_1})^\star$};
		\node [style=none] (62) at (7.25, -4.5) {$G(Y_{q1}\!\cdots\! Y_{qp_q})^\star$};
		\node [style=none] (63) at (-0.5, -4.5) {$G(Y_{21}\!\cdots\! Y_{2p_2})^\star$};
		\node [style=none] (64) at (3.25, -4.5) {$\cdots$};
	\end{pgfonlayer}
	\begin{pgfonlayer}{edgelayer}
		\draw [in=135, out=-45] (0.center) to (39.center);
		\draw [in=135, out=-45] (1.center) to (40.center);
		\draw [in=135, out=-45] (2.center) to (41.center);
		\draw [in=135, out=-45] (3.center) to (42.center);
		\draw [in=135, out=-45] (4.center) to (43.center);
		\draw [in=90, out=-75] (5.center) to (44.center);
		\draw [in=90, out=-75] (6.center) to (45.center);
		\draw [in=90, out=-75] (7.center) to (46.center);
		\draw [in=45, out=-135] (8.center) to (47.center);
		\draw [in=45, out=-135] (9.center) to (48.center);
		\draw [in=45, out=-135] (10.center) to (49.center);
		\draw [in=45, out=-135] (11.center) to (50.center);
		\draw [in=315, out=120] (19.center) to (34.center);
		\draw [in=315, out=120] (20.center) to (35.center);
		\draw [in=315, out=120] (21.center) to (36.center);
		\draw [in=315, out=120] (22.center) to (37.center);
		\draw [in=255, out=90] (18.center) to (33.center);
		\draw [in=255, out=90] (17.center) to (32.center);
		\draw [in=255, out=90] (16.center) to (31.center);
		\draw [in=210, out=45] (12.center) to (27.center);
		\draw [in=45, out=-150] (28.center) to (13.center);
		\draw [in=210, out=45] (14.center) to (29.center);
		\draw [in=45, out=-150] (30.center) to (15.center);
		\draw [in=255, out=90] (52.center) to (51.center);
	\end{pgfonlayer}
\end{tikzpicture}
\end{equation*}
where the map goes from top to bottom. Composition is given by vertical stacking and tensor product by horizontal juxtaposition.

\paragraph{String diagrams of the main morphisms.}
The building blocks are the following morphisms:
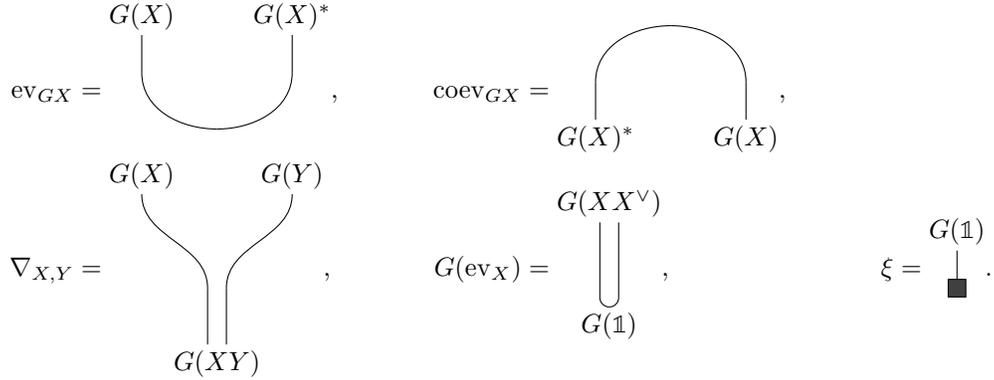
\begin{align*}
    \ev_{GX} &= \begin{tikzpicture}
	\begin{pgfonlayer}{nodelayer}
		\node [style=none] (0) at (-2, 1.5) {};
		\node [style=none] (1) at (-2, 0.5) {};
		\node [style=none] (4) at (0, -1) {};
		\node [style=none] (5) at (2, 1.5) {};
		\node [style=none] (6) at (2, 0.5) {};
		\node [style=none] (7) at (0, -1) {};
		\node [style=none] (8) at (-2, 2) {$G(X)$};
		\node [style=none] (9) at (2, 2) {$G(X)^*$};
		\node [style=none] (10) at (0, -1) {};
	\end{pgfonlayer}
	\begin{pgfonlayer}{edgelayer}
		\draw (0.center) to (1.center);
		\draw [in=-90, out=180] (4.center) to (1.center);
		\draw (5.center) to (6.center);
		\draw [in=-90, out=0] (7.center) to (6.center);
	\end{pgfonlayer}
\end{tikzpicture}, &
    \coev_{GX} &= \begin{tikzpicture}
	\begin{pgfonlayer}{nodelayer}
		\node [style=none] (0) at (-2, -0.75) {};
		\node [style=none] (1) at (-2, 0.25) {};
		\node [style=none] (4) at (0, 1.75) {};
		\node [style=none] (5) at (2, -0.75) {};
		\node [style=none] (6) at (2, 0.25) {};
		\node [style=none] (7) at (0, 1.75) {};
		\node [style=none] (8) at (-2, -1.25) {$G(X)^*$};
		\node [style=none] (9) at (2, -1.25) {$G(X)$};
		\node [style=none] (10) at (0, 1.75) {};
	\end{pgfonlayer}
	\begin{pgfonlayer}{edgelayer}
		\draw (0.center) to (1.center);
		\draw [in=90, out=-180] (4.center) to (1.center);
		\draw (5.center) to (6.center);
		\draw [in=90, out=0] (7.center) to (6.center);
	\end{pgfonlayer}
\end{tikzpicture}, & \\
    \nabla_{X,Y} &= \begin{tikzpicture}
	\begin{pgfonlayer}{nodelayer}
		\node [style=none] (4) at (-0.25, -0.5) {};
		\node [style=none] (6) at (-0.25, -2) {};
		\node [style=none] (8) at (-2, 2) {};
		\node [style=none] (11) at (0.25, -0.5) {};
		\node [style=none] (12) at (0.25, -2) {};
		\node [style=none] (13) at (2, 2) {};
		\node [style=none] (14) at (0, -2.5) {$G(XY)$};
		\node [style=none] (15) at (2, 2.5) {};
		\node [style=none] (16) at (2, 2.5) {$G(Y)$};
		\node [style=none] (17) at (-2, 2.5) {$G(X)$};
	\end{pgfonlayer}
	\begin{pgfonlayer}{edgelayer}
		\draw (4.center) to (6.center);
		\draw (11.center) to (12.center);
		\draw [in=90, out=-90] (13.center) to (11.center);
		\draw [in=270, out=90] (4.center) to (8.center);
	\end{pgfonlayer}
\end{tikzpicture}, & 
    G(\ev_X) &= \begin{tikzpicture}
	\begin{pgfonlayer}{nodelayer}
		\node [style=none] (4) at (-0.25, 1.25) {};
		\node [style=none] (6) at (-0.25, -0.75) {};
		\node [style=none] (11) at (0.25, 1.25) {};
		\node [style=none] (12) at (0.25, -0.75) {};
		\node [style=none] (14) at (0, 1.75) {$G(XX^\vee)$};
		\node [style=none] (16) at (0, -1) {};
		\node [style=none] (17) at (0, -1.5) {$G(\1)$};
	\end{pgfonlayer}
	\begin{pgfonlayer}{edgelayer}
		\draw (4.center) to (6.center);
		\draw (11.center) to (12.center);
		\draw [in=0, out=-90] (12.center) to (16.center);
		\draw [in=180, out=-90] (6.center) to (16.center);
	\end{pgfonlayer}
\end{tikzpicture}, &
    \xi &= \begin{tikzpicture}
	\begin{pgfonlayer}{nodelayer}
		\node [style=none] (16) at (0, 0.5) {};
		\node [style=black rectangle] (17) at (0, -0.5) {};
		\node [style=none] (20) at (0, 1) {$G(\1)$};
	\end{pgfonlayer}
	\begin{pgfonlayer}{edgelayer}
		\draw (16.center) to (17);
	\end{pgfonlayer}
\end{tikzpicture}
.
\end{align*}

Empty inputs or outputs mean the unit $\1$ in $\zD$.
Note that morphisms may carry some degree: for example, $\xi\colon G(\1)\to \1$ is of degree $-d$, while the morphism $\ev_{GX}\colon G(X)\otimes G(X)^* \to \1$ carries none (or, equivalently, is of degree $0$). 

Using these morphisms as building blocks, by composing we can obtain the morphisms $\beta$, $\alpha$, $\nabla^*$ and $\Delta$, which we write as follows:
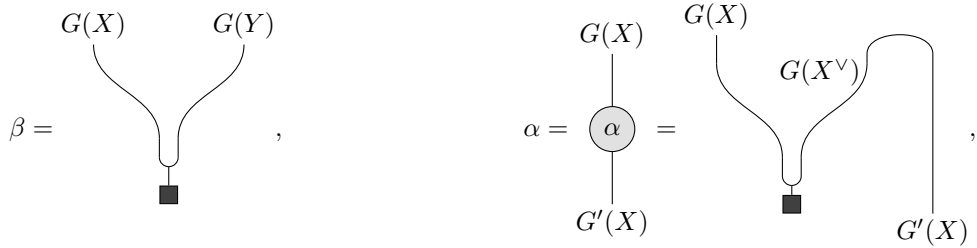
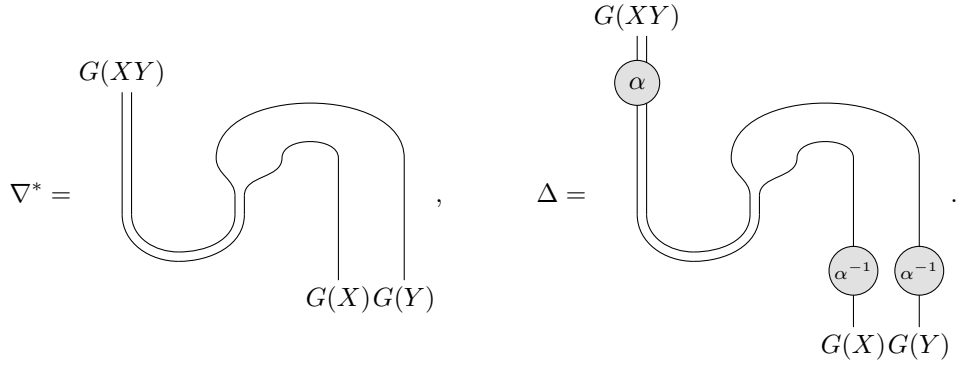
\begin{align*}
    \beta &= \begin{tikzpicture}
	\begin{pgfonlayer}{nodelayer}
		\node [style=none] (19) at (-0.25, -0.75) {};
		\node [style=none] (21) at (0.25, -0.75) {};
		\node [style=none] (23) at (0, -1) {};
		\node [style=black rectangle] (24) at (0, -1.75) {};
		\node [style=none] (26) at (-0.25, -0.25) {};
		\node [style=none] (27) at (-2, 2.25) {};
		\node [style=none] (28) at (0.25, -0.25) {};
		\node [style=none] (29) at (2, 2.25) {};
		\node [style=none] (30) at (2, 2.75) {};
		\node [style=none] (31) at (2, 2.75) {$G(Y)$};
		\node [style=none] (32) at (-2, 2.75) {$G(X)$};
	\end{pgfonlayer}
	\begin{pgfonlayer}{edgelayer}
		\draw (23.center) to (24);
		\draw [in=0, out=-90] (21.center) to (23.center);
		\draw [in=180, out=-90] (19.center) to (23.center);
		\draw [in=90, out=-90] (29.center) to (28.center);
		\draw [in=270, out=90] (26.center) to (27.center);
		\draw (26.center) to (19.center);
		\draw (21.center) to (28.center);
	\end{pgfonlayer}
\end{tikzpicture}, &
    \alpha &= \begin{tikzpicture}
	\begin{pgfonlayer}{nodelayer}
		\node [style=none] (0) at (0, 0.25) {};
		\node [style=grey dot] (1) at (0, 0) {$\alpha$};
		\node [style=none] (2) at (0, -0.25) {};
		\node [style=none] (3) at (0, 2) {};
		\node [style=none] (4) at (0, -2) {};
		\node [style=none] (5) at (0, 2.5) {$G(X)$};
		\node [style=none] (6) at (0, -2.5) {$G'(X)$};
	\end{pgfonlayer}
	\begin{pgfonlayer}{edgelayer}
		\draw (3.center) to (0.center);
		\draw (2.center) to (4.center);
	\end{pgfonlayer}
\end{tikzpicture} 
    = \begin{tikzpicture}
	\begin{pgfonlayer}{nodelayer}
		\node [style=none] (27) at (1.25, 2) {};
		\node [style=none] (29) at (3, -2.25) {};
		\node [style=none] (30) at (3, 2) {};
		\node [style=none] (33) at (3, -2.75) {$G'(X)$};
		\node [style=none] (36) at (0, 1.5) {$G(X^\vee)$};
		\node [style=none] (37) at (-1, -1.25) {};
		\node [style=none] (38) at (-0.5, -1.25) {};
		\node [style=none] (39) at (-0.75, -1.5) {};
		\node [style=black rectangle] (40) at (-0.75, -2) {};
		\node [style=none] (42) at (-1, -0.75) {};
		\node [style=none] (43) at (-2.75, 1.75) {};
		\node [style=none] (44) at (-0.5, -0.75) {};
		\node [style=none] (45) at (1.25, 1.75) {};
		\node [style=none] (48) at (-2.75, 3) {$G(X)$};
		\node [style=none] (49) at (-2.75, 2.5) {};
	\end{pgfonlayer}
	\begin{pgfonlayer}{edgelayer}
		\draw (29.center) to (30.center);
		\draw (39.center) to (40);
		\draw [in=0, out=-90] (38.center) to (39.center);
		\draw [in=180, out=-90] (37.center) to (39.center);
		\draw [in=90, out=-90] (45.center) to (44.center);
		\draw [in=270, out=90] (42.center) to (43.center);
		\draw (42.center) to (37.center);
		\draw (38.center) to (44.center);
		\draw (27.center) to (45.center);
		\draw (49.center) to (43.center);
		\draw [in=450, out=90] (30.center) to (27.center);
	\end{pgfonlayer}
\end{tikzpicture}, \\
    \nabla^* &= \begin{tikzpicture}
	\begin{pgfonlayer}{nodelayer}
		\node [style=none] (19) at (-2.75, -0.5) {};
		\node [style=none] (20) at (-3, -0.5) {};
		\node [style=none] (21) at (0, -0.5) {};
		\node [style=none] (23) at (0.25, -0.5) {};
		\node [style=none] (24) at (-1.5, -1.5) {};
		\node [style=none] (26) at (-1.5, -1.75) {};
		\node [style=none] (27) at (0, 0) {};
		\node [style=none] (28) at (0.25, 0) {};
		\node [style=none] (31) at (1.25, 1) {};
		\node [style=none] (32) at (-0.5, 1) {};
		\node [style=none] (33) at (2.75, 1) {};
		\node [style=none] (35) at (2.75, 1) {};
		\node [style=none] (36) at (4.5, 1) {};
		\node [style=none] (38) at (2.75, -2.25) {};
		\node [style=none] (39) at (4.5, -2.25) {};
		\node [style=none] (40) at (-3, 2.75) {};
		\node [style=none] (42) at (-2.75, 2.75) {};
		\node [style=none] (43) at (2.75, -2.75) {$G(X)$};
		\node [style=none] (44) at (4.5, -2.75) {$G(Y)$};
		\node [style=none] (45) at (-3, 3.25) {$G(XY)$};
	\end{pgfonlayer}
	\begin{pgfonlayer}{edgelayer}
		\draw [in=180, out=-90] (19.center) to (24.center);
		\draw [in=270, out=0] (24.center) to (21.center);
		\draw [in=0, out=-90] (23.center) to (26.center);
		\draw [in=180, out=-90] (20.center) to (26.center);
		\draw (27.center) to (21.center);
		\draw (28.center) to (23.center);
		\draw [in=270, out=90] (28.center) to (31.center);
		\draw [in=270, out=90] (27.center) to (32.center);
		\draw [in=450, out=90] (36.center) to (32.center);
		\draw [in=90, out=90] (31.center) to (35.center);
		\draw (36.center) to (39.center);
		\draw (38.center) to (35.center);
		\draw (19.center) to (42.center);
		\draw (40.center) to (20.center);
	\end{pgfonlayer}
\end{tikzpicture}, &
    \Delta &= \begin{tikzpicture}
	\begin{pgfonlayer}{nodelayer}
		\node [style=none] (19) at (-2.75, -0.5) {};
		\node [style=none] (20) at (-3, -0.5) {};
		\node [style=none] (21) at (0, -0.5) {};
		\node [style=none] (23) at (0.25, -0.5) {};
		\node [style=none] (24) at (-1.5, -1.5) {};
		\node [style=none] (26) at (-1.5, -1.75) {};
		\node [style=none] (27) at (0, 0) {};
		\node [style=none] (28) at (0.25, 0) {};
		\node [style=none] (31) at (1.25, 1) {};
		\node [style=none] (32) at (-0.5, 1) {};
		\node [style=none] (33) at (2.75, 1) {};
		\node [style=none] (35) at (2.75, 1) {};
		\node [style=none] (36) at (4.5, 1) {};
		\node [style=none] (38) at (2.75, -1.75) {};
		\node [style=none] (39) at (4.5, -1.5) {};
		\node [style=none] (40) at (-3, 2.75) {};
		\node [style=none] (42) at (-2.75, 2.75) {};
		\node [style=none] (43) at (2.75, -4) {$G(X)$};
		\node [style=none] (44) at (4.5, -4) {$G(Y)$};
		\node [style=none] (45) at (-3, 4.75) {$G(XY)$};
		\node [style=grey dot] (46) at (-3, 3) {$\alpha$};
		\node [style=none] (47) at (-3, 3.25) {};
		\node [style=none] (48) at (-2.75, 3.25) {};
		\node [style=none] (49) at (-3, 4.25) {};
		\node [style=none] (50) at (-2.75, 4.25) {};
		\node [style=small font circle] (51) at (2.75, -2) {${\alpha^{-1}}$};
		\node [style=small font circle] (52) at (4.5, -2) {$\alpha^{-1}$};
		\node [style=none] (53) at (2.75, -2.25) {};
		\node [style=none] (54) at (2.75, -3.5) {};
		\node [style=none] (55) at (4.5, -2.5) {};
		\node [style=none] (56) at (4.5, -3.5) {};
	\end{pgfonlayer}
	\begin{pgfonlayer}{edgelayer}
		\draw [in=180, out=-90] (19.center) to (24.center);
		\draw [in=270, out=0] (24.center) to (21.center);
		\draw [in=0, out=-90] (23.center) to (26.center);
		\draw [in=180, out=-90] (20.center) to (26.center);
		\draw (27.center) to (21.center);
		\draw (28.center) to (23.center);
		\draw [in=270, out=90] (28.center) to (31.center);
		\draw [in=270, out=90] (27.center) to (32.center);
		\draw [in=450, out=90] (36.center) to (32.center);
		\draw [in=90, out=90] (31.center) to (35.center);
		\draw (36.center) to (39.center);
		\draw (38.center) to (35.center);
		\draw (19.center) to (42.center);
		\draw (40.center) to (20.center);
		\draw (49.center) to (47.center);
		\draw (50.center) to (48.center);
		\draw (53.center) to (54.center);
		\draw (55.center) to (56.center);
	\end{pgfonlayer}
\end{tikzpicture}.
\end{align*}

We now collect the rules these string diagrams satisfy, in preparation for the proof of Theorem \ref{thm:main}.

\begin{lm}[Rules for graphical calculus]
    \label{lem:graphical}
    For every $X,Y,Z\in \zC$, every $A,A',B,B'\in \zD$ and every morphism $f\colon A\to A'$, $g\colon B\to B'$, the following relations hold.
    \begin{enumerate}
        \item (Exchange law)
            
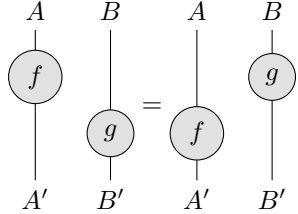
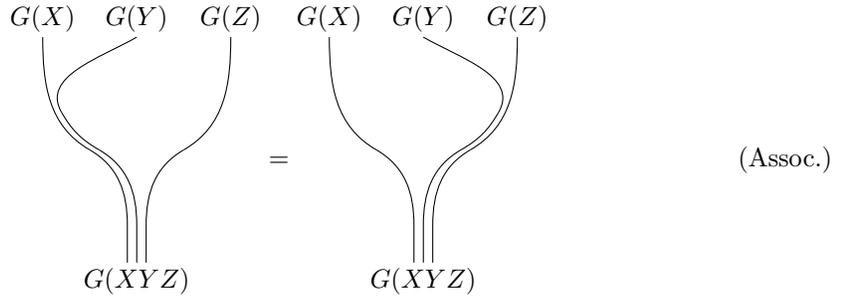
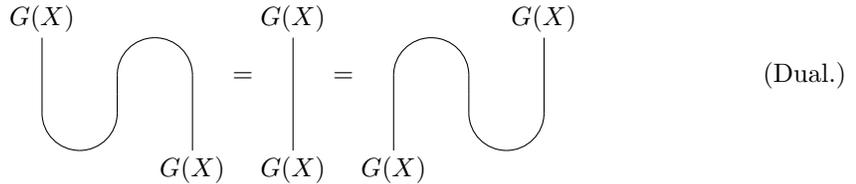
\begin{equation*}
                \begin{tikzpicture}
	\begin{pgfonlayer}{nodelayer}
		\node [style=grey dot] (0) at (-1, 0.75) {$f$};
		\node [style=grey dot] (1) at (1, -0.75) {$g$};
		\node [style=none] (2) at (-1, 1.25) {};
		\node [style=none] (3) at (1, -0.25) {};
		\node [style=none] (4) at (1, 2) {};
		\node [style=none] (5) at (-1, 2) {};
		\node [style=none] (6) at (-1, 0.25) {};
		\node [style=none] (7) at (-1, -2) {};
		\node [style=none] (8) at (1, -1.25) {};
		\node [style=none] (9) at (1, -2) {};
		\node [style=none] (10) at (-1, 2.5) {$A$};
		\node [style=none] (11) at (-1, -2.5) {$A'$};
		\node [style=none] (12) at (1, 2.5) {$B$};
		\node [style=none] (13) at (1, -2.5) {$B'$};
	\end{pgfonlayer}
	\begin{pgfonlayer}{edgelayer}
		\draw (5.center) to (2.center);
		\draw (6.center) to (7.center);
		\draw (8.center) to (9.center);
		\draw (4.center) to (3.center);
	\end{pgfonlayer}
\end{tikzpicture} =
                \begin{tikzpicture}
	\begin{pgfonlayer}{nodelayer}
		\node [style=grey dot] (0) at (-1, -0.75) {$f$};
		\node [style=grey dot] (1) at (1, 0.75) {$g$};
		\node [style=none] (2) at (-1, -0.25) {};
		\node [style=none] (3) at (1, 1.25) {};
		\node [style=none] (4) at (1, 2) {};
		\node [style=none] (5) at (-1, 2) {};
		\node [style=none] (6) at (-1, -1.25) {};
		\node [style=none] (7) at (-1, -2) {};
		\node [style=none] (8) at (1, 0.25) {};
		\node [style=none] (9) at (1, -2) {};
		\node [style=none] (10) at (-1, 2.5) {$A$};
		\node [style=none] (11) at (-1, -2.5) {$A'$};
		\node [style=none] (12) at (1, 2.5) {$B$};
		\node [style=none] (13) at (1, -2.5) {$B'$};
	\end{pgfonlayer}
	\begin{pgfonlayer}{edgelayer}
		\draw (5.center) to (2.center);
		\draw (6.center) to (7.center);
		\draw (8.center) to (9.center);
		\draw (4.center) to (3.center);
	\end{pgfonlayer}
\end{tikzpicture} 
            \end{equation*}
        \item (Associativity of $\nabla$)
            \begin{equation}
                \tag{\text{Assoc.}}
                \label{tag:ass}
                \begin{tikzpicture}
	\begin{pgfonlayer}{nodelayer}
		\node [style=none] (0) at (-1.25, 0.25) {};
		\node [style=none] (1) at (-0.25, -1.75) {};
		\node [style=none] (2) at (-0.25, -2.75) {};
		\node [style=none] (3) at (-2.5, 3.25) {};
		\node [style=none] (4) at (1.25, 0.25) {};
		\node [style=none] (5) at (0.25, -1.75) {};
		\node [style=none] (6) at (0.25, -2.75) {};
		\node [style=none] (7) at (2.5, 3.25) {};
		\node [style=none] (8) at (0, -3.25) {$G(XYZ)$};
		\node [style=none] (9) at (2.5, 3.75) {};
		\node [style=none] (10) at (2.5, 3.75) {$G(Z)$};
		\node [style=none] (11) at (-2.5, 3.75) {$G(X)$};
		\node [style=none] (12) at (0, -2.75) {};
		\node [style=none] (13) at (0, -1.75) {};
		\node [style=none] (14) at (-1, 0.25) {};
		\node [style=none] (15) at (0, 3.25) {};
		\node [style=none] (16) at (-2, 1.25) {};
		\node [style=none] (17) at (0, 3.75) {$G(Y)$};
	\end{pgfonlayer}
	\begin{pgfonlayer}{edgelayer}
		\draw (1.center) to (2.center);
		\draw (5.center) to (6.center);
		\draw (12.center) to (13.center);
		\draw [in=-150, out=120] (16.center) to (15.center);
		\draw [in=90, out=-150] (4.center) to (5.center);
		\draw [in=-30, out=90] (13.center) to (14.center);
		\draw [in=330, out=90] (1.center) to (0.center);
		\draw [in=270, out=30] (4.center) to (7.center);
		\draw [in=270, out=150] (0.center) to (3.center);
		\draw [in=-60, out=150] (14.center) to (16.center);
	\end{pgfonlayer}
\end{tikzpicture} =
                \begin{tikzpicture}
	\begin{pgfonlayer}{nodelayer}
		\node [style=none] (0) at (-1.25, 0.25) {};
		\node [style=none] (1) at (-0.25, -1.75) {};
		\node [style=none] (2) at (-0.25, -2.75) {};
		\node [style=none] (3) at (-2.5, 3.25) {};
		\node [style=none] (4) at (1.25, 0.25) {};
		\node [style=none] (5) at (0.25, -1.75) {};
		\node [style=none] (6) at (0.25, -2.75) {};
		\node [style=none] (7) at (2.5, 3.25) {};
		\node [style=none] (8) at (0, -3.25) {$G(XYZ)$};
		\node [style=none] (9) at (2.5, 3.75) {};
		\node [style=none] (10) at (2.5, 3.75) {$G(Z)$};
		\node [style=none] (11) at (-2.5, 3.75) {$G(X)$};
		\node [style=none] (12) at (0, -2.75) {};
		\node [style=none] (13) at (0, -1.75) {};
		\node [style=none] (14) at (1, 0.25) {};
		\node [style=none] (15) at (0, 3.25) {};
		\node [style=none] (16) at (2, 1.25) {};
		\node [style=none] (17) at (0, 3.75) {$G(Y)$};
	\end{pgfonlayer}
	\begin{pgfonlayer}{edgelayer}
		\draw (1.center) to (2.center);
		\draw (5.center) to (6.center);
		\draw (12.center) to (13.center);
		\draw [in=-30, out=60] (16.center) to (15.center);
		\draw [in=90, out=-150] (4.center) to (5.center);
		\draw [in=-150, out=90] (13.center) to (14.center);
		\draw [in=330, out=90] (1.center) to (0.center);
		\draw [in=270, out=30] (4.center) to (7.center);
		\draw [in=270, out=150] (0.center) to (3.center);
		\draw [in=-120, out=30] (14.center) to (16.center);
	\end{pgfonlayer}
\end{tikzpicture} 
            \end{equation}
        \item (Duality)
            \begin{equation}
                \tag{\text{Dual.}}
                \label{tag:dual}
                \begin{tikzpicture}
	\begin{pgfonlayer}{nodelayer}
		\node [style=none] (0) at (0, -0.5) {};
		\node [style=none] (1) at (0, 0) {};
		\node [style=none] (5) at (2, -2) {};
		\node [style=none] (6) at (2, 0) {};
		\node [style=none] (7) at (1, 1) {};
		\node [style=none] (9) at (2, -2.5) {$G(X)$};
		\node [style=none] (11) at (-2, 1) {};
		\node [style=none] (12) at (-2, -1) {};
		\node [style=none] (13) at (-1, -2) {};
		\node [style=none] (15) at (0, -1) {};
		\node [style=none] (17) at (-2, 1.5) {$G(X)$};
	\end{pgfonlayer}
	\begin{pgfonlayer}{edgelayer}
		\draw (0.center) to (1.center);
		\draw (5.center) to (6.center);
		\draw [in=90, out=0] (7.center) to (6.center);
		\draw (11.center) to (12.center);
		\draw [in=-90, out=180] (13.center) to (12.center);
		\draw (0.center) to (15.center);
		\draw [in=90, out=-180] (7.center) to (1.center);
		\draw [in=270, out=0] (13.center) to (15.center);
	\end{pgfonlayer}
\end{tikzpicture} =
                \begin{tikzpicture}
	\begin{pgfonlayer}{nodelayer}
		\node [style=none] (5) at (0, -2) {};
		\node [style=none] (9) at (0, -2.5) {$G(X)$};
		\node [style=none] (11) at (0, 1) {};
		\node [style=none] (17) at (0, 1.5) {$G(X)$};
	\end{pgfonlayer}
	\begin{pgfonlayer}{edgelayer}
		\draw (11.center) to (5.center);
	\end{pgfonlayer}
\end{tikzpicture}
 =
                \begin{tikzpicture}
	\begin{pgfonlayer}{nodelayer}
		\node [style=none] (0) at (0, -0.5) {};
		\node [style=none] (1) at (0, 0) {};
		\node [style=none] (5) at (-2, -2) {};
		\node [style=none] (6) at (-2, 0) {};
		\node [style=none] (7) at (-1, 1) {};
		\node [style=none] (9) at (-2, -2.5) {$G(X)$};
		\node [style=none] (11) at (2, 1) {};
		\node [style=none] (12) at (2, -1) {};
		\node [style=none] (13) at (1, -2) {};
		\node [style=none] (15) at (0, -1) {};
		\node [style=none] (17) at (2, 1.5) {$G(X)$};
	\end{pgfonlayer}
	\begin{pgfonlayer}{edgelayer}
		\draw (0.center) to (1.center);
		\draw (5.center) to (6.center);
		\draw [in=90, out=180] (7.center) to (6.center);
		\draw (11.center) to (12.center);
		\draw [in=-90, out=0] (13.center) to (12.center);
		\draw (0.center) to (15.center);
		\draw [in=90, out=0] (7.center) to (1.center);
		\draw [in=-90, out=180] (13.center) to (15.center);
	\end{pgfonlayer}
\end{tikzpicture}
            \end{equation}
        \item (Naturality of $\nabla$ and functoriality of $G$)
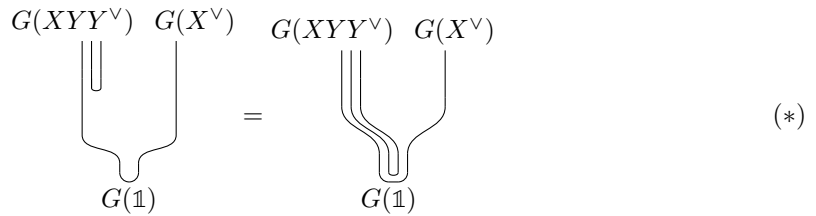
            \begin{equation}
                \tag{$*$}
                \label{tag:star}
                \begin{tikzpicture}
	\begin{pgfonlayer}{nodelayer}
		\node [style=none] (6) at (-1, 1) {};
		\node [style=none] (7) at (-1, 2) {};
		\node [style=none] (15) at (-1.25, 2) {};
		\node [style=none] (27) at (-0.75, 1.25) {};
		\node [style=none] (32) at (-0.75, 2) {};
		\node [style=none] (65) at (-0.75, 1) {};
		\node [style=none] (74) at (-1, 0.75) {};
		\node [style=none] (75) at (-0.75, 0.75) {};
		\node [style=none] (78) at (-0.25, -1.25) {};
		\node [style=none] (79) at (-1.25, -0.5) {};
		\node [style=none] (81) at (0.25, -1.25) {};
		\node [style=none] (82) at (1.25, -0.5) {};
		\node [style=none] (84) at (-0.25, -1.5) {};
		\node [style=none] (85) at (0.25, -1.5) {};
		\node [style=none] (86) at (0, -1.75) {};
		\node [style=none] (89) at (1.25, -0.5) {};
		\node [style=none] (97) at (-1.25, 1) {};
		\node [style=none] (101) at (-1, 1.25) {};
		\node [style=none] (102) at (-1.25, 1.25) {};
		\node [style=none] (103) at (1.25, 2) {};
		\node [style=none] (104) at (-1.5, 2.5) {$G(XYY^\vee)$};
		\node [style=none] (105) at (1.75, 2.5) {$G(X^\vee)$};
		\node [style=none] (106) at (0, -2.25) {$G(\1)$};
	\end{pgfonlayer}
	\begin{pgfonlayer}{edgelayer}
		\draw [in=270, out=90] (27.center) to (32.center);
		\draw [in=270, out=-90] (74.center) to (75.center);
		\draw (65.center) to (27.center);
		\draw (75.center) to (65.center);
		\draw (74.center) to (6.center);
		\draw [in=0, out=-90] (85.center) to (86.center);
		\draw [in=180, out=-90] (84.center) to (86.center);
		\draw (85.center) to (81.center);
		\draw (78.center) to (84.center);
		\draw (97.center) to (79.center);
		\draw [in=90, out=-90] (79.center) to (78.center);
		\draw [in=270, out=90] (81.center) to (89.center);
		\draw (102.center) to (97.center);
		\draw (101.center) to (74.center);
		\draw [in=270, out=90] (101.center) to (7.center);
		\draw [in=90, out=-90] (15.center) to (102.center);
		\draw (103.center) to (89.center);
	\end{pgfonlayer}
\end{tikzpicture} =
                \begin{tikzpicture}
	\begin{pgfonlayer}{nodelayer}
		\node [style=none] (27) at (-0.75, 1) {};
		\node [style=none] (32) at (-0.75, 1.75) {};
		\node [style=none] (103) at (-1, 1.75) {};
		\node [style=none] (104) at (-1.25, 1.75) {};
		\node [style=none] (129) at (1.5, 0.5) {};
		\node [style=none] (130) at (1.5, 1.75) {};
		\node [style=none] (133) at (-1.25, 0.25) {};
		\node [style=none] (134) at (-1, 0.25) {};
		\node [style=none] (136) at (-0.75, 0.25) {};
		\node [style=none] (137) at (1.5, 0.25) {};
		\node [style=none] (138) at (0, -1) {};
		\node [style=none] (139) at (0.25, -1) {};
		\node [style=none] (140) at (0.5, -1) {};
		\node [style=none] (141) at (-0.25, -1) {};
		\node [style=none] (142) at (0, -1.5) {};
		\node [style=none] (143) at (0.25, -1.5) {};
		\node [style=none] (144) at (0.5, -1.5) {};
		\node [style=none] (145) at (-0.25, -1.5) {};
		\node [style=none] (146) at (0, -1.75) {};
		\node [style=none] (148) at (-1.25, 0.25) {};
		\node [style=none] (149) at (0.25, -1.75) {};
		\node [style=none] (150) at (-1.5, 2.25) {$G(XYY^\vee)$};
		\node [style=none] (151) at (1.75, 2.25) {$G(X^\vee)$};
		\node [style=none] (152) at (0, -2.25) {$G(\1)$};
	\end{pgfonlayer}
	\begin{pgfonlayer}{edgelayer}
		\draw [in=270, out=90] (27.center) to (32.center);
		\draw [in=90, out=-90] (130.center) to (129.center);
		\draw [in=270, out=90] (140.center) to (137.center);
		\draw [in=270, out=90] (139.center) to (136.center);
		\draw [in=270, out=90] (138.center) to (134.center);
		\draw [in=270, out=90] (141.center) to (133.center);
		\draw (140.center) to (144.center);
		\draw (139.center) to (143.center);
		\draw (142.center) to (138.center);
		\draw (141.center) to (145.center);
		\draw [in=270, out=-90] (142.center) to (143.center);
		\draw [in=270, out=180] (146.center) to (145.center);
		\draw (27.center) to (136.center);
		\draw (137.center) to (129.center);
		\draw (148.center) to (104.center);
		\draw (103.center) to (134.center);
		\draw [in=0, out=-90] (144.center) to (149.center);
		\draw (149.center) to (146.center);
	\end{pgfonlayer}
\end{tikzpicture}
            \end{equation}
    \end{enumerate}
\end{lm}
\begin{proof}
    The proof consists simply in translating commutative diagrams into identities between corresponding string diagrams. We illustrate the principle with the last assertion: it comes from commutativity of the following diagram:
    \[\begin{tikzcd}[row sep=large, column sep = huge]
        {G(X\otimes Y \otimes Y^\vee)\otimes G(X^\vee)} & {G(X\otimes Y\otimes Y^\vee\otimes X^\vee)} & \\
        {G(X)\otimes G(X^\vee)} & {G(X\otimes X^\vee)} & {G(\1)}.
        \arrow["{{\nabla_{X\o Y\o Y^\vee,X^\vee}}}", from=1-1, to=1-2]
        \arrow["{G(\id_X\o \ev_Y)\o \id_{G(X^\vee)}}" left, from=1-1, to=2-1]
        \arrow["{G(\id_X\o \ev_Y \o {\id_{X^\vee}})}"', from=1-2, to=2-2]
        \arrow["{G(\ev_{X\o Y})}", from=1-2, to=2-3]
        \arrow["{{\nabla_{X,X^\vee}}}", from=2-1, to=2-2]
        \arrow["{G(\ev_X)}"', from=2-2, to=2-3]
    \end{tikzcd}\]
    which follows by naturality of $\nabla$ and functoriality of $G$.
\end{proof}

\subsection{Proof of main theorem}
\label{sub:proof_main}
\begin{proof}[Proof of Theorem \ref{thm:main}]
    We now show that the natural transformations $\nabla$ and $\Delta$ defined in \eqref{eqn:nabla} and \eqref{eqn:delta} satisfy the  \eqref{eqn:Frobenius_relation}. The verification takes the form of a computation, where each step follows from Lemma \ref{lem:graphical}. Recall that the steps involving the exchange law or unitality relations are implicit.

    \begin{align*}
        \Delta_{XY,Z} \circ \nabla_{X,YZ} 
        &\stackrel{(\text{def.})}{=} 
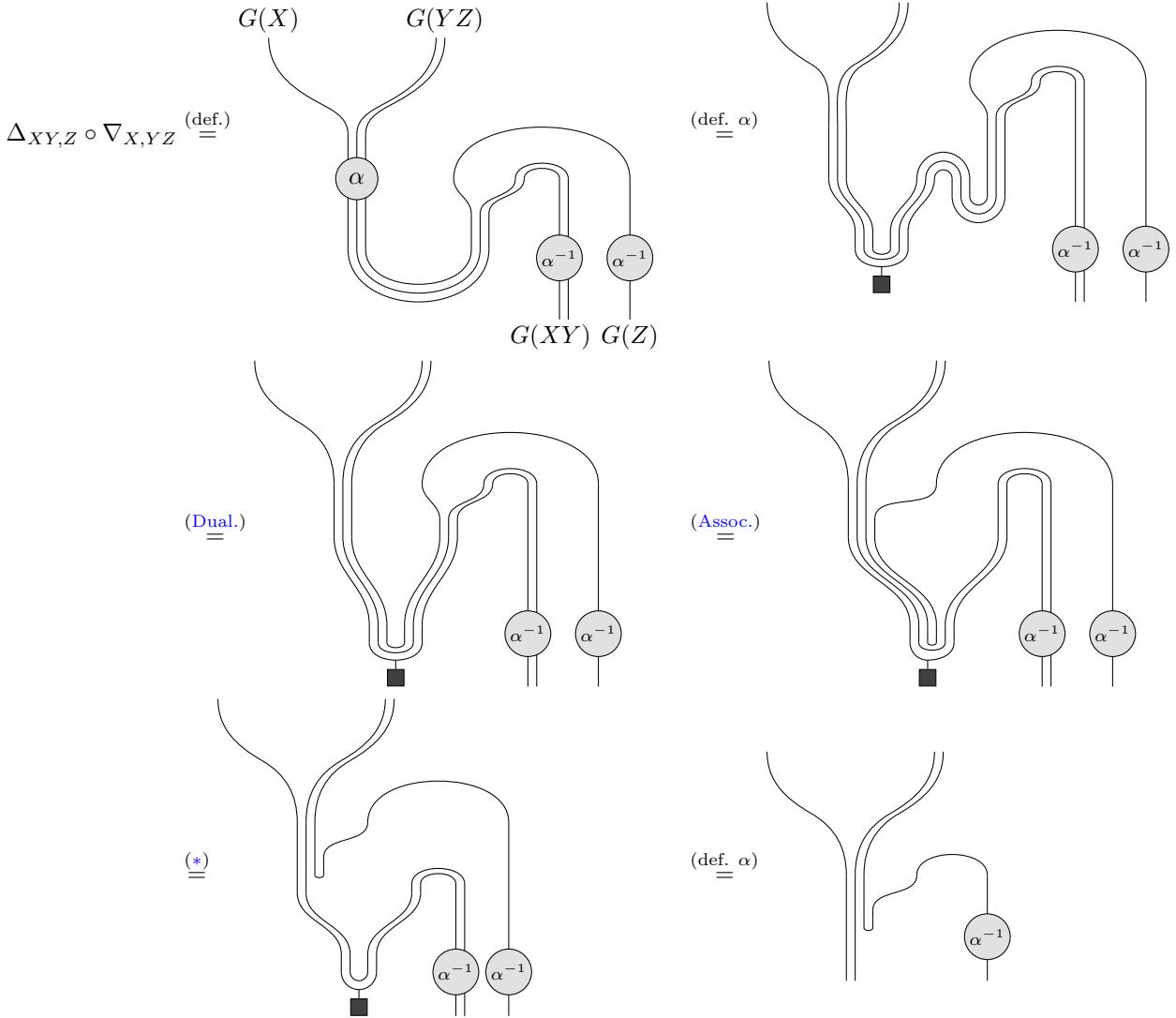
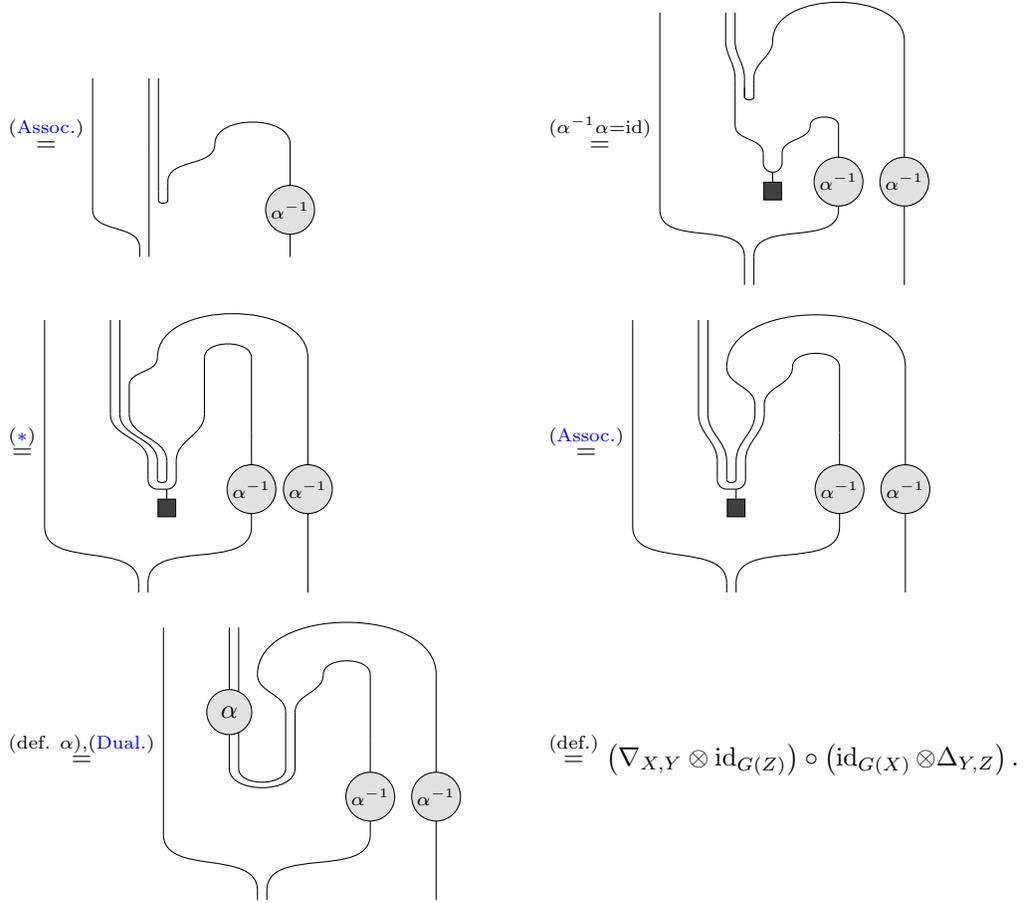
\begin{tikzpicture}
	\begin{pgfonlayer}{nodelayer}
		\node [style=none] (0) at (-3.5, 1) {};
		\node [style=none] (1) at (-2.5, 0) {};
		\node [style=none] (2) at (-2.5, -0.75) {};
		\node [style=none] (3) at (-4.75, 2.75) {};
		\node [style=none] (4) at (-1, 1) {};
		\node [style=none] (5) at (-2, 0) {};
		\node [style=none] (6) at (-2, -0.75) {};
		\node [style=none] (7) at (0.25, 2.75) {};
		\node [style=none] (9) at (0.25, 3.25) {};
		\node [style=none] (10) at (0.25, 3.25) {$G(YZ)$};
		\node [style=none] (11) at (-4.75, 3.25) {$G(X)$};
		\node [style=none] (12) at (-2.25, -0.75) {};
		\node [style=none] (13) at (-2.25, 0) {};
		\node [style=none] (14) at (-1.25, 1) {};
		\node [style=none] (15) at (0, 2.75) {};
		\node [style=grey dot] (17) at (-2.25, -1.25) {$\alpha$};
		\node [style=none] (18) at (-2.5, -3.25) {};
		\node [style=none] (19) at (-2, -3.25) {};
		\node [style=none] (20) at (-2.25, -3.25) {};
		\node [style=none] (21) at (1, -3.25) {};
		\node [style=none] (22) at (1.5, -3.25) {};
		\node [style=none] (23) at (1.25, -3.25) {};
		\node [style=none] (24) at (-0.5, -4.25) {};
		\node [style=none] (25) at (-0.5, -4.75) {};
		\node [style=none] (26) at (-0.5, -4.5) {};
		\node [style=none] (27) at (1, -2.25) {};
		\node [style=none] (28) at (1.25, -2.25) {};
		\node [style=none] (29) at (1.5, -2.25) {};
		\node [style=none] (30) at (2.5, -1.25) {};
		\node [style=none] (31) at (2.25, -1.25) {};
		\node [style=none] (32) at (0.5, -1.25) {};
		\node [style=none] (33) at (3.75, -1.25) {};
		\node [style=none] (34) at (3.5, -1.25) {};
		\node [style=none] (35) at (3.75, -1.25) {};
		\node [style=none] (36) at (5.5, -1.25) {};
		\node [style=none] (37) at (3.5, -3) {};
		\node [style=none] (38) at (3.75, -3) {};
		\node [style=none] (39) at (5.5, -5.25) {};
		\node [style=none] (40) at (-2.25, -1.75) {};
		\node [style=none] (41) at (-2.5, -1.75) {};
		\node [style=none] (42) at (-2, -1.75) {};
		\node [style=none] (43) at (3.25, -5.75) {$G(XY)$};
		\node [style=none] (44) at (5.5, -5.75) {$G(Z)$};
		\node [style=small font circle] (45) at (3.5, -3.5) {$\alpha^{-1}$};
		\node [style=none] (46) at (3.5, -3.75) {};
		\node [style=none] (47) at (3.75, -3.75) {};
		\node [style=none] (48) at (3.5, -5.25) {};
		\node [style=none] (49) at (3.75, -5.25) {};
		\node [style=none] (50) at (5.5, -3) {};
		\node [style=small font circle] (51) at (5.5, -3.5) {$\alpha^{-1}$};
		\node [style=none] (52) at (5.5, -3.75) {};
	\end{pgfonlayer}
	\begin{pgfonlayer}{edgelayer}
		\draw (1.center) to (2.center);
		\draw (5.center) to (6.center);
		\draw (12.center) to (13.center);
		\draw [in=90, out=-150] (4.center) to (5.center);
		\draw [in=-150, out=90] (13.center) to (14.center);
		\draw [in=330, out=90] (1.center) to (0.center);
		\draw [in=270, out=30] (4.center) to (7.center);
		\draw [in=270, out=150] (0.center) to (3.center);
		\draw [in=30, out=-90] (15.center) to (14.center);
		\draw [in=180, out=-90] (19.center) to (24.center);
		\draw [in=270, out=0] (24.center) to (21.center);
		\draw [in=0, out=-90] (23.center) to (26.center);
		\draw [in=180, out=-90] (20.center) to (26.center);
		\draw [in=180, out=-90] (18.center) to (25.center);
		\draw [in=270, out=0] (25.center) to (22.center);
		\draw (27.center) to (21.center);
		\draw (28.center) to (23.center);
		\draw (22.center) to (29.center);
		\draw [in=270, out=90] (29.center) to (30.center);
		\draw [in=270, out=90] (28.center) to (31.center);
		\draw [in=270, out=90] (27.center) to (32.center);
		\draw [in=450, out=90] (36.center) to (32.center);
		\draw [in=90, out=90] (31.center) to (35.center);
		\draw [in=90, out=90] (30.center) to (34.center);
		\draw (34.center) to (37.center);
		\draw (38.center) to (35.center);
		\draw (19.center) to (42.center);
		\draw (40.center) to (20.center);
		\draw (41.center) to (18.center);
		\draw (47.center) to (49.center);
		\draw (48.center) to (46.center);
		\draw (52.center) to (39.center);
		\draw (50.center) to (36.center);
	\end{pgfonlayer}
\end{tikzpicture} &
        &\stackrel{(\text{def. } \alpha)}{=} \begin{tikzpicture}
	\begin{pgfonlayer}{nodelayer}
		\node [style=none] (0) at (-2.75, 2) {};
		\node [style=none] (1) at (-2.25, 0.25) {};
		\node [style=none] (2) at (-2.25, -1.25) {};
		\node [style=none] (3) at (-4, 3.75) {};
		\node [style=none] (4) at (-1.25, 2) {};
		\node [style=none] (5) at (-1.75, 0.25) {};
		\node [style=none] (6) at (-1.75, -1.25) {};
		\node [style=none] (7) at (0, 3.75) {};
		\node [style=none] (12) at (-2, -1.25) {};
		\node [style=none] (13) at (-2, 0.25) {};
		\node [style=none] (14) at (-1.5, 2) {};
		\node [style=none] (15) at (-0.25, 3.75) {};
		\node [style=none] (18) at (1.25, -1.75) {};
		\node [style=none] (19) at (1.75, -1.75) {};
		\node [style=none] (20) at (1.5, -1.75) {};
		\node [style=none] (21) at (2.25, -1.75) {};
		\node [style=none] (22) at (2.75, -1.75) {};
		\node [style=none] (23) at (2.5, -1.75) {};
		\node [style=none] (24) at (2, -2) {};
		\node [style=none] (25) at (2, -2.5) {};
		\node [style=none] (26) at (2, -2.25) {};
		\node [style=none] (27) at (2.25, 0.5) {};
		\node [style=none] (28) at (2.5, 0.5) {};
		\node [style=none] (29) at (2.75, 0.5) {};
		\node [style=none] (30) at (3.75, 1.5) {};
		\node [style=none] (31) at (3.5, 1.5) {};
		\node [style=none] (32) at (1.75, 1.5) {};
		\node [style=none] (33) at (5, 1.5) {};
		\node [style=none] (34) at (4.75, 1.5) {};
		\node [style=none] (35) at (5, 1.5) {};
		\node [style=none] (36) at (6.75, 1.5) {};
		\node [style=none] (37) at (4.75, -2.75) {};
		\node [style=none] (38) at (5, -2.75) {};
		\node [style=none] (39) at (6.75, -4.75) {};
		\node [style=none] (40) at (1.5, -1.25) {};
		\node [style=none] (41) at (1.25, -1.25) {};
		\node [style=none] (42) at (1.75, -1.25) {};
		\node [style=small font circle] (45) at (4.75, -3.25) {$\alpha^{-1}$};
		\node [style=none] (46) at (4.75, -3.5) {};
		\node [style=none] (47) at (5, -3.5) {};
		\node [style=none] (48) at (4.75, -4.75) {};
		\node [style=none] (49) at (5, -4.75) {};
		\node [style=none] (50) at (6.75, -2.75) {};
		\node [style=small font circle] (51) at (6.75, -3.25) {$\alpha^{-1}$};
		\node [style=none] (52) at (6.75, -3.5) {};
		\node [style=none] (53) at (0.25, -1.25) {};
		\node [style=none] (54) at (0.75, -1.25) {};
		\node [style=none] (55) at (0.5, -1.25) {};
		\node [style=none] (56) at (1.25, -1.25) {};
		\node [style=none] (57) at (1.75, -1.25) {};
		\node [style=none] (58) at (1.5, -1.25) {};
		\node [style=none] (59) at (1, -1) {};
		\node [style=none] (60) at (1, -0.5) {};
		\node [style=none] (61) at (1, -0.75) {};
		\node [style=none] (65) at (0.25, -1.25) {};
		\node [style=none] (66) at (0.75, -1.25) {};
		\node [style=none] (67) at (0.5, -1.25) {};
		\node [style=none] (68) at (-1, -2.75) {};
		\node [style=none] (69) at (-0.5, -2.75) {};
		\node [style=none] (70) at (-0.25, -2.75) {};
		\node [style=none] (71) at (0, -2.75) {};
		\node [style=none] (72) at (-1.25, -2.75) {};
		\node [style=none] (73) at (-1.5, -2.75) {};
		\node [style=none] (74) at (-1, -3.25) {};
		\node [style=none] (75) at (-0.5, -3.25) {};
		\node [style=none] (76) at (-0.25, -3.25) {};
		\node [style=none] (77) at (0, -3.25) {};
		\node [style=none] (78) at (-1.25, -3.25) {};
		\node [style=none] (79) at (-1.5, -3.25) {};
		\node [style=none] (80) at (-0.75, -3.75) {};
		\node [style=black rectangle] (81) at (-0.75, -4.25) {};
	\end{pgfonlayer}
	\begin{pgfonlayer}{edgelayer}
		\draw (1.center) to (2.center);
		\draw (5.center) to (6.center);
		\draw (12.center) to (13.center);
		\draw [in=90, out=-150] (4.center) to (5.center);
		\draw [in=-150, out=90] (13.center) to (14.center);
		\draw [in=330, out=90] (1.center) to (0.center);
		\draw [in=270, out=30] (4.center) to (7.center);
		\draw [in=270, out=150] (0.center) to (3.center);
		\draw [in=30, out=-90] (15.center) to (14.center);
		\draw [in=180, out=-90] (19.center) to (24.center);
		\draw [in=270, out=0] (24.center) to (21.center);
		\draw [in=0, out=-90] (23.center) to (26.center);
		\draw [in=180, out=-90] (20.center) to (26.center);
		\draw [in=180, out=-90] (18.center) to (25.center);
		\draw [in=270, out=0] (25.center) to (22.center);
		\draw (27.center) to (21.center);
		\draw (28.center) to (23.center);
		\draw (22.center) to (29.center);
		\draw [in=270, out=90] (29.center) to (30.center);
		\draw [in=270, out=90] (28.center) to (31.center);
		\draw [in=270, out=90] (27.center) to (32.center);
		\draw [in=450, out=90] (36.center) to (32.center);
		\draw [in=90, out=90] (31.center) to (35.center);
		\draw [in=90, out=90] (30.center) to (34.center);
		\draw (34.center) to (37.center);
		\draw (38.center) to (35.center);
		\draw (19.center) to (42.center);
		\draw (40.center) to (20.center);
		\draw (41.center) to (18.center);
		\draw (47.center) to (49.center);
		\draw (48.center) to (46.center);
		\draw (52.center) to (39.center);
		\draw (50.center) to (36.center);
		\draw [in=-180, out=90] (54.center) to (59.center);
		\draw [in=-270, out=0] (59.center) to (56.center);
		\draw [in=0, out=90] (58.center) to (61.center);
		\draw [in=-180, out=90] (55.center) to (61.center);
		\draw [in=-180, out=90] (53.center) to (60.center);
		\draw [in=-270, out=0] (60.center) to (57.center);
		\draw [in=270, out=90] (71.center) to (66.center);
		\draw [in=90, out=-90] (67.center) to (70.center);
		\draw [in=270, out=90] (69.center) to (65.center);
		\draw [in=270, out=90] (68.center) to (6.center);
		\draw [in=90, out=-90] (12.center) to (72.center);
		\draw [in=270, out=90] (73.center) to (2.center);
		\draw (71.center) to (77.center);
		\draw (76.center) to (70.center);
		\draw (69.center) to (75.center);
		\draw (74.center) to (68.center);
		\draw (73.center) to (79.center);
		\draw (78.center) to (72.center);
		\draw [in=270, out=-90] (78.center) to (76.center);
		\draw [in=270, out=-90] (74.center) to (75.center);
		\draw [in=0, out=-90] (77.center) to (80.center);
		\draw [in=270, out=180] (80.center) to (79.center);
		\draw (80.center) to (81);
	\end{pgfonlayer}
\end{tikzpicture} &\\
        &\stackrel{\eqref{tag:dual}}{=} \begin{tikzpicture}
	\begin{pgfonlayer}{nodelayer}
		\node [style=none] (0) at (-3.5, 3.25) {};
		\node [style=none] (1) at (-2.5, 1.5) {};
		\node [style=none] (2) at (-2.5, 0) {};
		\node [style=none] (3) at (-4.75, 5) {};
		\node [style=none] (4) at (-1, 3.25) {};
		\node [style=none] (5) at (-2, 1.5) {};
		\node [style=none] (6) at (-2, 0) {};
		\node [style=none] (7) at (0.25, 5) {};
		\node [style=none] (12) at (-2.25, 0) {};
		\node [style=none] (13) at (-2.25, 1.5) {};
		\node [style=none] (14) at (-1.25, 3.25) {};
		\node [style=none] (15) at (0, 5) {};
		\node [style=none] (27) at (0.5, 0.5) {};
		\node [style=none] (28) at (0.75, 0.5) {};
		\node [style=none] (29) at (1, 0.5) {};
		\node [style=none] (30) at (2, 1.5) {};
		\node [style=none] (31) at (1.75, 1.5) {};
		\node [style=none] (32) at (0, 1.5) {};
		\node [style=none] (33) at (3.25, 1.5) {};
		\node [style=none] (34) at (3, 1.5) {};
		\node [style=none] (35) at (3.25, 1.5) {};
		\node [style=none] (36) at (5, 1.5) {};
		\node [style=none] (37) at (3, -2.25) {};
		\node [style=none] (38) at (3.25, -2.25) {};
		\node [style=none] (39) at (5, -4.25) {};
		\node [style=small font circle] (45) at (3, -2.75) {$\alpha^{-1}$};
		\node [style=none] (46) at (3, -3) {};
		\node [style=none] (47) at (3.25, -3) {};
		\node [style=none] (48) at (3, -4.25) {};
		\node [style=none] (49) at (3.25, -4.25) {};
		\node [style=none] (50) at (5, -2.25) {};
		\node [style=small font circle] (51) at (5, -2.75) {$\alpha^{-1}$};
		\node [style=none] (52) at (5, -3) {};
		\node [style=none] (53) at (0.5, 0) {};
		\node [style=none] (54) at (1, 0) {};
		\node [style=none] (55) at (0.75, 0) {};
		\node [style=none] (65) at (0.5, 0) {};
		\node [style=none] (66) at (1, 0) {};
		\node [style=none] (67) at (0.75, 0) {};
		\node [style=none] (68) at (-1, -2.5) {};
		\node [style=none] (69) at (-0.5, -2.5) {};
		\node [style=none] (70) at (-0.25, -2.5) {};
		\node [style=none] (71) at (0, -2.5) {};
		\node [style=none] (72) at (-1.25, -2.5) {};
		\node [style=none] (73) at (-1.5, -2.5) {};
		\node [style=none] (74) at (-1, -3) {};
		\node [style=none] (75) at (-0.5, -3) {};
		\node [style=none] (76) at (-0.25, -3) {};
		\node [style=none] (77) at (0, -3) {};
		\node [style=none] (78) at (-1.25, -3) {};
		\node [style=none] (79) at (-1.5, -3) {};
		\node [style=none] (80) at (-0.75, -3.5) {};
		\node [style=black rectangle] (81) at (-0.75, -4) {};
	\end{pgfonlayer}
	\begin{pgfonlayer}{edgelayer}
		\draw (1.center) to (2.center);
		\draw (5.center) to (6.center);
		\draw (12.center) to (13.center);
		\draw [in=90, out=-150] (4.center) to (5.center);
		\draw [in=-150, out=90] (13.center) to (14.center);
		\draw [in=330, out=90] (1.center) to (0.center);
		\draw [in=270, out=30] (4.center) to (7.center);
		\draw [in=270, out=150] (0.center) to (3.center);
		\draw [in=30, out=-90] (15.center) to (14.center);
		\draw [in=270, out=90] (29.center) to (30.center);
		\draw [in=270, out=90] (28.center) to (31.center);
		\draw [in=270, out=90] (27.center) to (32.center);
		\draw [in=450, out=90] (36.center) to (32.center);
		\draw [in=90, out=90] (31.center) to (35.center);
		\draw [in=90, out=90] (30.center) to (34.center);
		\draw (34.center) to (37.center);
		\draw (38.center) to (35.center);
		\draw (47.center) to (49.center);
		\draw (48.center) to (46.center);
		\draw (52.center) to (39.center);
		\draw (50.center) to (36.center);
		\draw [in=270, out=90] (71.center) to (66.center);
		\draw [in=90, out=-90] (67.center) to (70.center);
		\draw [in=270, out=90] (69.center) to (65.center);
		\draw [in=270, out=90] (68.center) to (6.center);
		\draw [in=90, out=-90] (12.center) to (72.center);
		\draw [in=270, out=90] (73.center) to (2.center);
		\draw (71.center) to (77.center);
		\draw (76.center) to (70.center);
		\draw (69.center) to (75.center);
		\draw (74.center) to (68.center);
		\draw (73.center) to (79.center);
		\draw (78.center) to (72.center);
		\draw [in=270, out=-90] (78.center) to (76.center);
		\draw [in=270, out=-90] (74.center) to (75.center);
		\draw [in=0, out=-90] (77.center) to (80.center);
		\draw [in=270, out=180] (80.center) to (79.center);
		\draw (80.center) to (81);
		\draw (66.center) to (29.center);
		\draw (28.center) to (67.center);
		\draw (65.center) to (27.center);
	\end{pgfonlayer}
\end{tikzpicture} &
        &\stackrel{\eqref{tag:ass}}{=} \begin{tikzpicture}
	\begin{pgfonlayer}{nodelayer}
		\node [style=none] (0) at (-3.5, 3.25) {};
		\node [style=none] (1) at (-2.5, 1.5) {};
		\node [style=none] (2) at (-2.5, 0) {};
		\node [style=none] (3) at (-4.75, 5) {};
		\node [style=none] (4) at (-1, 3.25) {};
		\node [style=none] (5) at (-2, 1.5) {};
		\node [style=none] (6) at (-2, 0) {};
		\node [style=none] (7) at (0.25, 5) {};
		\node [style=none] (12) at (-2.25, 0) {};
		\node [style=none] (13) at (-2.25, 1.5) {};
		\node [style=none] (14) at (-1.25, 3.25) {};
		\node [style=none] (15) at (0, 5) {};
		\node [style=none] (27) at (-1.75, 0.5) {};
		\node [style=none] (30) at (2, 1.5) {};
		\node [style=none] (31) at (1.75, 1.5) {};
		\node [style=none] (32) at (0, 1.5) {};
		\node [style=none] (33) at (3.25, 1.5) {};
		\node [style=none] (34) at (3, 1.5) {};
		\node [style=none] (35) at (3.25, 1.5) {};
		\node [style=none] (36) at (5, 1.5) {};
		\node [style=none] (37) at (3, -2.25) {};
		\node [style=none] (38) at (3.25, -2.25) {};
		\node [style=none] (39) at (5, -4.25) {};
		\node [style=small font circle] (45) at (3, -2.75) {$\alpha^{-1}$};
		\node [style=none] (46) at (3, -3) {};
		\node [style=none] (47) at (3.25, -3) {};
		\node [style=none] (48) at (3, -4.25) {};
		\node [style=none] (49) at (3.25, -4.25) {};
		\node [style=none] (50) at (5, -2.25) {};
		\node [style=small font circle] (51) at (5, -2.75) {$\alpha^{-1}$};
		\node [style=none] (52) at (5, -3) {};
		\node [style=none] (54) at (2, 0) {};
		\node [style=none] (55) at (1.75, 0) {};
		\node [style=none] (65) at (-1.75, 0) {};
		\node [style=none] (66) at (2, 0) {};
		\node [style=none] (67) at (1.75, 0) {};
		\node [style=none] (68) at (-0.25, -2.5) {};
		\node [style=none] (69) at (0, -2.5) {};
		\node [style=none] (70) at (0.25, -2.5) {};
		\node [style=none] (71) at (0.5, -2.5) {};
		\node [style=none] (72) at (-0.5, -2.5) {};
		\node [style=none] (73) at (-0.75, -2.5) {};
		\node [style=none] (74) at (-0.25, -3) {};
		\node [style=none] (75) at (0, -3) {};
		\node [style=none] (76) at (0.25, -3) {};
		\node [style=none] (77) at (0.5, -3) {};
		\node [style=none] (78) at (-0.5, -3) {};
		\node [style=none] (79) at (-0.75, -3) {};
		\node [style=none] (80) at (-0.25, -3.5) {};
		\node [style=black rectangle] (81) at (-0.25, -4) {};
	\end{pgfonlayer}
	\begin{pgfonlayer}{edgelayer}
		\draw (1.center) to (2.center);
		\draw (5.center) to (6.center);
		\draw (12.center) to (13.center);
		\draw [in=90, out=-150] (4.center) to (5.center);
		\draw [in=-150, out=90] (13.center) to (14.center);
		\draw [in=330, out=90] (1.center) to (0.center);
		\draw [in=270, out=30] (4.center) to (7.center);
		\draw [in=270, out=150] (0.center) to (3.center);
		\draw [in=30, out=-90] (15.center) to (14.center);
		\draw [in=270, out=90] (27.center) to (32.center);
		\draw [in=450, out=90] (36.center) to (32.center);
		\draw [in=90, out=90] (31.center) to (35.center);
		\draw [in=90, out=90] (30.center) to (34.center);
		\draw (34.center) to (37.center);
		\draw (38.center) to (35.center);
		\draw (47.center) to (49.center);
		\draw (48.center) to (46.center);
		\draw (52.center) to (39.center);
		\draw (50.center) to (36.center);
		\draw [in=270, out=90] (71.center) to (66.center);
		\draw [in=90, out=-90] (67.center) to (70.center);
		\draw [in=270, out=90] (69.center) to (65.center);
		\draw [in=270, out=90] (68.center) to (6.center);
		\draw [in=90, out=-90] (12.center) to (72.center);
		\draw [in=270, out=90] (73.center) to (2.center);
		\draw (71.center) to (77.center);
		\draw (76.center) to (70.center);
		\draw (69.center) to (75.center);
		\draw (74.center) to (68.center);
		\draw (73.center) to (79.center);
		\draw (78.center) to (72.center);
		\draw [in=270, out=-90] (78.center) to (76.center);
		\draw [in=270, out=-90] (74.center) to (75.center);
		\draw [in=0, out=-90] (77.center) to (80.center);
		\draw [in=270, out=180] (80.center) to (79.center);
		\draw (80.center) to (81);
		\draw (65.center) to (27.center);
		\draw [in=90, out=-90] (30.center) to (66.center);
		\draw [in=270, out=90] (67.center) to (31.center);
	\end{pgfonlayer}
\end{tikzpicture} &\\
        &\stackrel{\eqref{tag:star}}{=} \begin{tikzpicture}
	\begin{pgfonlayer}{nodelayer}
		\node [style=none] (0) at (-2.5, 3.25) {};
		\node [style=none] (1) at (-1.5, 1.5) {};
		\node [style=none] (2) at (-1.5, -0.5) {};
		\node [style=none] (3) at (-3.75, 5) {};
		\node [style=none] (4) at (0, 3.25) {};
		\node [style=none] (5) at (-1, 1.5) {};
		\node [style=none] (6) at (-1, 0.25) {};
		\node [style=none] (7) at (1.25, 5) {};
		\node [style=none] (12) at (-1.25, -0.5) {};
		\node [style=none] (13) at (-1.25, 1.5) {};
		\node [style=none] (14) at (-0.25, 3.25) {};
		\node [style=none] (15) at (1, 5) {};
		\node [style=none] (27) at (-0.75, 0.5) {};
		\node [style=none] (30) at (2, -0.25) {};
		\node [style=none] (31) at (1.75, -0.25) {};
		\node [style=none] (32) at (0.5, 1.5) {};
		\node [style=none] (33) at (3.25, -0.25) {};
		\node [style=none] (34) at (3, -0.25) {};
		\node [style=none] (35) at (3.25, -0.25) {};
		\node [style=none] (36) at (4.5, 1.5) {};
		\node [style=none] (37) at (3, -2.25) {};
		\node [style=none] (38) at (3.25, -2.25) {};
		\node [style=none] (39) at (4.5, -4) {};
		\node [style=small font circle] (45) at (3, -2.75) {$\alpha^{-1}$};
		\node [style=none] (46) at (3, -3) {};
		\node [style=none] (47) at (3.25, -3) {};
		\node [style=none] (48) at (3, -4) {};
		\node [style=none] (49) at (3.25, -4) {};
		\node [style=none] (50) at (4.5, -2.25) {};
		\node [style=small font circle] (51) at (4.5, -2.75) {$\alpha^{-1}$};
		\node [style=none] (52) at (4.5, -3) {};
		\node [style=none] (54) at (2, -0.5) {};
		\node [style=none] (55) at (1.75, -0.5) {};
		\node [style=none] (65) at (-0.75, 0.25) {};
		\node [style=none] (66) at (2, -0.5) {};
		\node [style=none] (67) at (1.75, -0.5) {};
		\node [style=none] (70) at (0.5, -2.25) {};
		\node [style=none] (71) at (0.75, -2.25) {};
		\node [style=none] (72) at (0, -2.25) {};
		\node [style=none] (73) at (-0.25, -2.25) {};
		\node [style=none] (74) at (-1, 0) {};
		\node [style=none] (75) at (-0.75, 0) {};
		\node [style=none] (76) at (0.5, -2.75) {};
		\node [style=none] (77) at (0.75, -2.75) {};
		\node [style=none] (78) at (0, -2.75) {};
		\node [style=none] (79) at (-0.25, -2.75) {};
		\node [style=none] (80) at (0.25, -3.25) {};
		\node [style=black rectangle] (81) at (0.25, -3.75) {};
		\node [style=none] (83) at (0.25, -3) {};
	\end{pgfonlayer}
	\begin{pgfonlayer}{edgelayer}
		\draw (1.center) to (2.center);
		\draw (5.center) to (6.center);
		\draw (12.center) to (13.center);
		\draw [in=90, out=-150] (4.center) to (5.center);
		\draw [in=-150, out=90] (13.center) to (14.center);
		\draw [in=330, out=90] (1.center) to (0.center);
		\draw [in=270, out=30] (4.center) to (7.center);
		\draw [in=270, out=150] (0.center) to (3.center);
		\draw [in=30, out=-90] (15.center) to (14.center);
		\draw [in=270, out=90] (27.center) to (32.center);
		\draw [in=450, out=90] (36.center) to (32.center);
		\draw [in=90, out=90] (31.center) to (35.center);
		\draw [in=90, out=90] (30.center) to (34.center);
		\draw (34.center) to (37.center);
		\draw (38.center) to (35.center);
		\draw (47.center) to (49.center);
		\draw (48.center) to (46.center);
		\draw (52.center) to (39.center);
		\draw (50.center) to (36.center);
		\draw [in=270, out=90] (71.center) to (66.center);
		\draw [in=90, out=-90] (67.center) to (70.center);
		\draw [in=90, out=-90] (12.center) to (72.center);
		\draw [in=270, out=90] (73.center) to (2.center);
		\draw (71.center) to (77.center);
		\draw (76.center) to (70.center);
		\draw (73.center) to (79.center);
		\draw (78.center) to (72.center);
		\draw [in=270, out=-90] (74.center) to (75.center);
		\draw [in=0, out=-90] (77.center) to (80.center);
		\draw [in=270, out=180] (80.center) to (79.center);
		\draw (80.center) to (81);
		\draw (65.center) to (27.center);
		\draw [in=90, out=-90] (30.center) to (66.center);
		\draw [in=270, out=90] (67.center) to (31.center);
		\draw (75.center) to (65.center);
		\draw (74.center) to (6.center);
		\draw [in=270, out=0] (83.center) to (76.center);
		\draw [in=180, out=-90] (78.center) to (83.center);
	\end{pgfonlayer}
\end{tikzpicture} &
        &\stackrel{(\text{def. } \alpha)}{=} \begin{tikzpicture}
	\begin{pgfonlayer}{nodelayer}
		\node [style=none] (0) at (-2.5, 1.75) {};
		\node [style=none] (1) at (-1.5, 0) {};
		\node [style=none] (2) at (-1.5, -2.25) {};
		\node [style=none] (3) at (-3.75, 3.5) {};
		\node [style=none] (4) at (0, 1.75) {};
		\node [style=none] (5) at (-1, 0) {};
		\node [style=none] (6) at (-1, -1.25) {};
		\node [style=none] (7) at (1.25, 3.5) {};
		\node [style=none] (12) at (-1.25, -2.25) {};
		\node [style=none] (13) at (-1.25, 0) {};
		\node [style=none] (14) at (-0.25, 1.75) {};
		\node [style=none] (15) at (1, 3.5) {};
		\node [style=none] (27) at (-0.75, -1) {};
		\node [style=none] (32) at (0.5, 0) {};
		\node [style=none] (36) at (2.5, 0) {};
		\node [style=none] (39) at (2.5, -3) {};
		\node [style=none] (48) at (-1.5, -3) {};
		\node [style=none] (49) at (-1.25, -3) {};
		\node [style=none] (50) at (2.5, -1.25) {};
		\node [style=small font circle] (51) at (2.5, -1.75) {$\alpha^{-1}$};
		\node [style=none] (52) at (2.5, -2) {};
		\node [style=none] (65) at (-0.75, -1.25) {};
		\node [style=none] (74) at (-1, -1.5) {};
		\node [style=none] (75) at (-0.75, -1.5) {};
	\end{pgfonlayer}
	\begin{pgfonlayer}{edgelayer}
		\draw (1.center) to (2.center);
		\draw (5.center) to (6.center);
		\draw (12.center) to (13.center);
		\draw [in=90, out=-150] (4.center) to (5.center);
		\draw [in=-150, out=90] (13.center) to (14.center);
		\draw [in=330, out=90] (1.center) to (0.center);
		\draw [in=270, out=30] (4.center) to (7.center);
		\draw [in=270, out=150] (0.center) to (3.center);
		\draw [in=30, out=-90] (15.center) to (14.center);
		\draw [in=270, out=90] (27.center) to (32.center);
		\draw [in=450, out=90] (36.center) to (32.center);
		\draw (52.center) to (39.center);
		\draw (50.center) to (36.center);
		\draw [in=270, out=-90] (74.center) to (75.center);
		\draw (65.center) to (27.center);
		\draw (75.center) to (65.center);
		\draw (74.center) to (6.center);
		\draw (48.center) to (2.center);
		\draw (12.center) to (49.center);
	\end{pgfonlayer}
\end{tikzpicture} &\\
        &\stackrel{\eqref{tag:ass}}{=} \begin{tikzpicture}
	\begin{pgfonlayer}{nodelayer}
		\node [style=none] (1) at (-1.5, -2.75) {};
		\node [style=none] (3) at (-2.75, -1.75) {};
		\node [style=none] (5) at (-1, 0) {};
		\node [style=none] (6) at (-1, -1.25) {};
		\node [style=none] (7) at (-1, 1.75) {};
		\node [style=none] (12) at (-1.25, -2.75) {};
		\node [style=none] (13) at (-1.25, 0) {};
		\node [style=none] (15) at (-1.25, 1.75) {};
		\node [style=none] (27) at (-0.75, -1) {};
		\node [style=none] (32) at (0.5, 0) {};
		\node [style=none] (36) at (2.5, 0) {};
		\node [style=none] (39) at (2.5, -3) {};
		\node [style=none] (48) at (-1.5, -3) {};
		\node [style=none] (49) at (-1.25, -3) {};
		\node [style=none] (50) at (2.5, -1.25) {};
		\node [style=small font circle] (51) at (2.5, -1.75) {$\alpha^{-1}$};
		\node [style=none] (52) at (2.5, -2) {};
		\node [style=none] (65) at (-0.75, -1.25) {};
		\node [style=none] (74) at (-1, -1.5) {};
		\node [style=none] (75) at (-0.75, -1.5) {};
		\node [style=none] (76) at (-2.75, 1.75) {};
	\end{pgfonlayer}
	\begin{pgfonlayer}{edgelayer}
		\draw (5.center) to (6.center);
		\draw (12.center) to (13.center);
		\draw [in=270, out=90] (27.center) to (32.center);
		\draw [in=450, out=90] (36.center) to (32.center);
		\draw (52.center) to (39.center);
		\draw (50.center) to (36.center);
		\draw [in=270, out=-90] (74.center) to (75.center);
		\draw (65.center) to (27.center);
		\draw (75.center) to (65.center);
		\draw (74.center) to (6.center);
		\draw (12.center) to (49.center);
		\draw (76.center) to (3.center);
		\draw (1.center) to (48.center);
		\draw [in=90, out=-90] (3.center) to (1.center);
		\draw (15.center) to (13.center);
		\draw (5.center) to (7.center);
	\end{pgfonlayer}
\end{tikzpicture} &
        &\stackrel{(\alpha^{-1}\alpha=\id)}{=} \begin{tikzpicture}
	\begin{pgfonlayer}{nodelayer}
		\node [style=none] (1) at (-1, -3) {};
		\node [style=none] (3) at (-3.25, -1.75) {};
		\node [style=none] (6) at (-1, 1.5) {};
		\node [style=none] (7) at (-1.25, 2.75) {};
		\node [style=none] (12) at (-0.75, -3) {};
		\node [style=none] (15) at (-1.5, 2.75) {};
		\node [style=none] (27) at (-0.75, 1.75) {};
		\node [style=none] (32) at (-0.25, 2.75) {};
		\node [style=none] (36) at (3.25, 2.75) {};
		\node [style=none] (39) at (3.25, -3.75) {};
		\node [style=none] (48) at (-1, -3.75) {};
		\node [style=none] (49) at (-0.75, -3.75) {};
		\node [style=none] (50) at (3.25, -0.75) {};
		\node [style=small font circle] (51) at (3.25, -1) {$\alpha^{-1}$};
		\node [style=none] (52) at (3.25, -1.25) {};
		\node [style=none] (65) at (-0.75, 1.5) {};
		\node [style=none] (74) at (-1, 1.25) {};
		\node [style=none] (75) at (-0.75, 1.25) {};
		\node [style=none] (76) at (-3.25, 3.5) {};
		\node [style=none] (78) at (-0.5, -0.25) {};
		\node [style=none] (79) at (-1.25, 0.5) {};
		\node [style=none] (81) at (0, -0.25) {};
		\node [style=none] (82) at (0.75, 0.5) {};
		\node [style=none] (84) at (-0.5, -0.5) {};
		\node [style=none] (85) at (0, -0.5) {};
		\node [style=none] (86) at (-0.25, -0.75) {};
		\node [style=black rectangle] (87) at (-0.25, -1.25) {};
		\node [style=none] (89) at (0.75, 0.5) {};
		\node [style=none] (92) at (1.5, -1.75) {};
		\node [style=none] (93) at (1.5, 0.5) {};
		\node [style=none] (97) at (-1.25, 1.5) {};
		\node [style=none] (98) at (1.5, -0.75) {};
		\node [style=small font circle] (99) at (1.5, -1) {$\alpha^{-1}$};
		\node [style=none] (100) at (1.5, -1.25) {};
		\node [style=none] (101) at (-1, 1.75) {};
		\node [style=none] (102) at (-1.25, 1.75) {};
		\node [style=none] (103) at (-1.25, 3.5) {};
		\node [style=none] (104) at (-1.5, 3.5) {};
	\end{pgfonlayer}
	\begin{pgfonlayer}{edgelayer}
		\draw [in=270, out=90] (27.center) to (32.center);
		\draw [in=450, out=90] (36.center) to (32.center);
		\draw (52.center) to (39.center);
		\draw (50.center) to (36.center);
		\draw [in=270, out=-90] (74.center) to (75.center);
		\draw (65.center) to (27.center);
		\draw (75.center) to (65.center);
		\draw (74.center) to (6.center);
		\draw (12.center) to (49.center);
		\draw (76.center) to (3.center);
		\draw (1.center) to (48.center);
		\draw [in=90, out=-90] (3.center) to (1.center);
		\draw (86.center) to (87);
		\draw [in=0, out=-90] (85.center) to (86.center);
		\draw [in=180, out=-90] (84.center) to (86.center);
		\draw (85.center) to (81.center);
		\draw (78.center) to (84.center);
		\draw (97.center) to (79.center);
		\draw [in=450, out=90] (93.center) to (89.center);
		\draw [in=90, out=-90] (79.center) to (78.center);
		\draw [in=270, out=90] (81.center) to (89.center);
		\draw [in=90, out=-90] (92.center) to (12.center);
		\draw (93.center) to (98.center);
		\draw (92.center) to (100.center);
		\draw (102.center) to (97.center);
		\draw (101.center) to (74.center);
		\draw [in=270, out=90] (101.center) to (7.center);
		\draw [in=90, out=-90] (15.center) to (102.center);
		\draw (104.center) to (15.center);
		\draw (7.center) to (103.center);
	\end{pgfonlayer}
\end{tikzpicture} &\\
        &\stackrel{\eqref{tag:star}}{=} \begin{tikzpicture}
	\begin{pgfonlayer}{nodelayer}
		\node [style=none] (1) at (-0.5, -3.5) {};
		\node [style=none] (3) at (-3, -2) {};
		\node [style=none] (12) at (-0.25, -3.5) {};
		\node [style=none] (27) at (-0.75, 1.75) {};
		\node [style=none] (32) at (0, 2.5) {};
		\node [style=none] (36) at (4, 2.5) {};
		\node [style=none] (39) at (4, -3.75) {};
		\node [style=none] (48) at (-0.5, -3.75) {};
		\node [style=none] (49) at (-0.25, -3.75) {};
		\node [style=none] (50) at (4, -0.75) {};
		\node [style=small font circle] (51) at (4, -1) {$\alpha^{-1}$};
		\node [style=none] (52) at (4, -1.25) {};
		\node [style=none] (76) at (-3, 3.5) {};
		\node [style=none] (92) at (2.5, -2) {};
		\node [style=none] (93) at (2.5, 0.5) {};
		\node [style=none] (98) at (2.5, -0.75) {};
		\node [style=small font circle] (99) at (2.5, -1) {$\alpha^{-1}$};
		\node [style=none] (100) at (2.5, -1.25) {};
		\node [style=none] (103) at (-1, 3.5) {};
		\node [style=none] (104) at (-1.25, 3.5) {};
		\node [style=none] (129) at (1.25, 1.25) {};
		\node [style=none] (130) at (1.25, 2.5) {};
		\node [style=none] (131) at (2.5, 2.5) {};
		\node [style=none] (133) at (-1.25, 1) {};
		\node [style=none] (134) at (-1, 1) {};
		\node [style=none] (136) at (-0.75, 1) {};
		\node [style=none] (137) at (1.25, 1) {};
		\node [style=none] (138) at (0, -0.25) {};
		\node [style=none] (139) at (0.25, -0.25) {};
		\node [style=none] (140) at (0.5, -0.25) {};
		\node [style=none] (141) at (-0.25, -0.25) {};
		\node [style=none] (142) at (0, -0.75) {};
		\node [style=none] (143) at (0.25, -0.75) {};
		\node [style=none] (144) at (0.5, -0.75) {};
		\node [style=none] (145) at (-0.25, -0.75) {};
		\node [style=none] (146) at (0, -1) {};
		\node [style=black rectangle] (147) at (0.25, -1.5) {};
		\node [style=none] (148) at (-1.25, 1) {};
		\node [style=none] (149) at (0.25, -1) {};
	\end{pgfonlayer}
	\begin{pgfonlayer}{edgelayer}
		\draw [in=270, out=90] (27.center) to (32.center);
		\draw [in=450, out=90] (36.center) to (32.center);
		\draw (52.center) to (39.center);
		\draw (50.center) to (36.center);
		\draw (12.center) to (49.center);
		\draw (76.center) to (3.center);
		\draw (1.center) to (48.center);
		\draw [in=90, out=-90] (3.center) to (1.center);
		\draw [in=90, out=-90] (92.center) to (12.center);
		\draw (93.center) to (98.center);
		\draw (92.center) to (100.center);
		\draw [in=450, out=90] (131.center) to (130.center);
		\draw (93.center) to (131.center);
		\draw [in=90, out=-90] (130.center) to (129.center);
		\draw [in=270, out=90] (140.center) to (137.center);
		\draw [in=270, out=90] (139.center) to (136.center);
		\draw [in=270, out=90] (138.center) to (134.center);
		\draw [in=270, out=90] (141.center) to (133.center);
		\draw (140.center) to (144.center);
		\draw (139.center) to (143.center);
		\draw (142.center) to (138.center);
		\draw (141.center) to (145.center);
		\draw [in=270, out=-90] (142.center) to (143.center);
		\draw [in=270, out=180] (146.center) to (145.center);
		\draw (27.center) to (136.center);
		\draw (137.center) to (129.center);
		\draw (148.center) to (104.center);
		\draw (103.center) to (134.center);
		\draw [in=0, out=-90] (144.center) to (149.center);
		\draw (149.center) to (146.center);
		\draw (149.center) to (147);
	\end{pgfonlayer}
\end{tikzpicture} &
        &\stackrel{\eqref{tag:ass}}{=} \begin{tikzpicture}
	\begin{pgfonlayer}{nodelayer}
		\node [style=none] (1) at (-0.5, -3.5) {};
		\node [style=none] (3) at (-3, -2) {};
		\node [style=none] (12) at (-0.25, -3.5) {};
		\node [style=none] (27) at (0.25, 1.25) {};
		\node [style=none] (32) at (-0.5, 2.25) {};
		\node [style=none] (36) at (4.25, 2.25) {};
		\node [style=none] (39) at (4.25, -3.75) {};
		\node [style=none] (48) at (-0.5, -3.75) {};
		\node [style=none] (49) at (-0.25, -3.75) {};
		\node [style=none] (50) at (4.25, -0.75) {};
		\node [style=small font circle] (51) at (4.25, -1) {$\alpha^{-1}$};
		\node [style=none] (52) at (4.25, -1.25) {};
		\node [style=none] (76) at (-3, 3.5) {};
		\node [style=none] (92) at (2.5, -2) {};
		\node [style=none] (93) at (2.5, 0.5) {};
		\node [style=none] (98) at (2.5, -0.75) {};
		\node [style=small font circle] (99) at (2.5, -1) {$\alpha^{-1}$};
		\node [style=none] (100) at (2.5, -1.25) {};
		\node [style=none] (103) at (-1, 3.5) {};
		\node [style=none] (104) at (-1.25, 3.5) {};
		\node [style=none] (129) at (0.5, 1.25) {};
		\node [style=none] (130) at (1.25, 2.25) {};
		\node [style=none] (131) at (2.5, 2.25) {};
		\node [style=none] (133) at (-1.25, 1) {};
		\node [style=none] (134) at (-1, 1) {};
		\node [style=none] (135) at (0.25, 1) {};
		\node [style=none] (136) at (0.25, 1) {};
		\node [style=none] (137) at (0.5, 1) {};
		\node [style=none] (138) at (-0.5, -0.25) {};
		\node [style=none] (139) at (-0.25, -0.25) {};
		\node [style=none] (140) at (0, -0.25) {};
		\node [style=none] (141) at (-0.75, -0.25) {};
		\node [style=none] (142) at (-0.5, -0.75) {};
		\node [style=none] (143) at (-0.25, -0.75) {};
		\node [style=none] (144) at (0, -0.75) {};
		\node [style=none] (145) at (-0.75, -0.75) {};
		\node [style=none] (146) at (-0.5, -1) {};
		\node [style=black rectangle] (147) at (-0.25, -1.5) {};
		\node [style=none] (148) at (-1.25, 1) {};
		\node [style=none] (149) at (-0.25, -1) {};
	\end{pgfonlayer}
	\begin{pgfonlayer}{edgelayer}
		\draw [in=270, out=90] (27.center) to (32.center);
		\draw [in=450, out=90] (36.center) to (32.center);
		\draw (52.center) to (39.center);
		\draw (50.center) to (36.center);
		\draw (12.center) to (49.center);
		\draw (76.center) to (3.center);
		\draw (1.center) to (48.center);
		\draw [in=90, out=-90] (3.center) to (1.center);
		\draw [in=90, out=-90] (92.center) to (12.center);
		\draw (93.center) to (98.center);
		\draw (92.center) to (100.center);
		\draw [in=450, out=90] (131.center) to (130.center);
		\draw (93.center) to (131.center);
		\draw [in=90, out=-90] (130.center) to (129.center);
		\draw [in=270, out=90] (140.center) to (137.center);
		\draw [in=270, out=90] (139.center) to (136.center);
		\draw [in=270, out=90] (138.center) to (134.center);
		\draw [in=270, out=90] (141.center) to (133.center);
		\draw (140.center) to (144.center);
		\draw (139.center) to (143.center);
		\draw (142.center) to (138.center);
		\draw (141.center) to (145.center);
		\draw [in=270, out=-90] (142.center) to (143.center);
		\draw [in=270, out=180] (146.center) to (145.center);
		\draw (27.center) to (136.center);
		\draw (137.center) to (129.center);
		\draw (148.center) to (104.center);
		\draw (103.center) to (134.center);
		\draw [in=0, out=-90] (144.center) to (149.center);
		\draw (149.center) to (146.center);
		\draw (149.center) to (147);
	\end{pgfonlayer}
\end{tikzpicture} &\\
        &\stackrel{(\text{def. } \alpha),\eqref{tag:dual}}{=} \begin{tikzpicture}
	\begin{pgfonlayer}{nodelayer}
		\node [style=none] (1) at (-0.5, -3.5) {};
		\node [style=none] (3) at (-3, -2) {};
		\node [style=none] (12) at (-0.25, -3.5) {};
		\node [style=none] (27) at (0.25, 1.25) {};
		\node [style=none] (32) at (-0.5, 2.25) {};
		\node [style=none] (36) at (4.25, 2.25) {};
		\node [style=none] (39) at (4.25, -3.75) {};
		\node [style=none] (48) at (-0.5, -3.75) {};
		\node [style=none] (49) at (-0.25, -3.75) {};
		\node [style=none] (50) at (4.25, -0.75) {};
		\node [style=small font circle] (51) at (4.25, -1) {$\alpha^{-1}$};
		\node [style=none] (52) at (4.25, -1.25) {};
		\node [style=none] (76) at (-3, 3.5) {};
		\node [style=none] (92) at (2.5, -2) {};
		\node [style=none] (93) at (2.5, 0.5) {};
		\node [style=none] (98) at (2.5, -0.75) {};
		\node [style=small font circle] (99) at (2.5, -1) {$\alpha^{-1}$};
		\node [style=none] (100) at (2.5, -1.25) {};
		\node [style=none] (103) at (-1, 3.5) {};
		\node [style=none] (104) at (-1.25, 3.5) {};
		\node [style=none] (129) at (0.5, 1.25) {};
		\node [style=none] (130) at (1.25, 2.25) {};
		\node [style=none] (131) at (2.5, 2.25) {};
		\node [style=none] (133) at (-1.25, 1.5) {};
		\node [style=none] (134) at (-1, 1.5) {};
		\node [style=none] (135) at (0.25, 1) {};
		\node [style=none] (136) at (0.25, -0.25) {};
		\node [style=none] (137) at (0.5, -0.25) {};
		\node [style=none] (148) at (-1.25, 1.5) {};
		\node [style=grey dot] (152) at (-1.25, 1.25) {$\alpha$};
		\node [style=none] (153) at (-1.25, 1) {};
		\node [style=none] (155) at (-1, 1) {};
		\node [style=none] (156) at (-1.25, -0.25) {};
		\node [style=none] (157) at (-1, -0.25) {};
	\end{pgfonlayer}
	\begin{pgfonlayer}{edgelayer}
		\draw [in=270, out=90] (27.center) to (32.center);
		\draw [in=450, out=90] (36.center) to (32.center);
		\draw (52.center) to (39.center);
		\draw (50.center) to (36.center);
		\draw (12.center) to (49.center);
		\draw (76.center) to (3.center);
		\draw (1.center) to (48.center);
		\draw [in=90, out=-90] (3.center) to (1.center);
		\draw [in=90, out=-90] (92.center) to (12.center);
		\draw (93.center) to (98.center);
		\draw (92.center) to (100.center);
		\draw [in=450, out=90] (131.center) to (130.center);
		\draw (93.center) to (131.center);
		\draw [in=90, out=-90] (130.center) to (129.center);
		\draw (27.center) to (136.center);
		\draw (137.center) to (129.center);
		\draw (148.center) to (104.center);
		\draw (103.center) to (134.center);
		\draw (155.center) to (157.center);
		\draw (153.center) to (156.center);
		\draw [in=270, out=-90] (157.center) to (136.center);
		\draw [in=270, out=-90] (156.center) to (137.center);
	\end{pgfonlayer}
\end{tikzpicture} &
        &\stackrel{(\text{def.)}}{=} \left( \nabla_{X,Y}\otimes \id_{G(Z)} \right) \circ \left( \id_{G(X)}\otimes \Delta_{Y,Z} \right) .
    \end{align*}
\end{proof}


\section{Day convolution for dioperads}
\label{sec:day_convolution}

The goal of this section is to define and study a notion of Day convolution for dioperads, by which we mean the construction of exponential objects in the category of dioperads. 

In the literature, the classical version of Day convolution is for monoidal categories and originates in \cite{day_convolution}. There is a variant for operads established in \cite{Pisani_convolution_operads}; from the point of view of the latter result, Day's original construction can be seen as the representable case. In the setting of dioperads, a parallel story goes, the analog of Day convolution result being the main theorem of \cite{Egger_StarAutonomousFctCategories}, where monoidal categories are replaced by linear distributive categories. Our Theorem \ref{thm:day} completes this picture by establishing a non-representable counterpart of Egger's result, which the author suggested in \cite[Remark 4.4]{Egger_StarAutonomousFctCategories}\footnote{Note however that Egger suggests that a slight generalization of Theorem \ref{thm:day} should hold true, in which the source category $\zC$ is allowed to be any \emph{bilinear} linearly distributive category, of which $*$-autonomous are special cases.}.

\begin{rk}[Restriction to the unenriched setting]
    In this section, as opposed to the rest of the paper, we work with \emph{unenriched} dioperads. The results and proofs remain valid in the enriched setting \emph{mutatis mutandis}; the restriction to the unenriched situation hopefully makes the constructions and arguments more readable. Note however that our main application in Section \ref{sec:day_application} consists in recovering Theorem \ref{thm:main} by means of the Day convolution for dioperads and therefore depends on an enriched version of Theorem \ref{thm:day}; in this sense, this second argument for Theorem \ref{thm:main} does not constitute a complete proof.
\end{rk}

\subsection{Recollections on \texorpdfstring{$*$}{*}-autonomous categories}
\label{section:recollections_*-autonomous}

We start by recalling the definition of $*$-autonomous categories, which were introduced and studied in \cite{Barr-StarAutonomousCatsBook}.

\begin{df}\label{df:starAutCats}
    ($*$-autonomous categories)
    A \emph{$*$-autonomous category} is a closed symmetric monoidal category $(\zC,\otimes,[-,-])$ endowed with a \emph{dualizing object}, that is an object $K \in \zC$ such that for every $X\in \zC$ the canonical morphism 
    \[
        X \longrightarrow [[X,K],K]
    \]
    is an isomorphism.
\end{df}

Given a $*$-autonomous category $\zC$, we write $(-)^*\colon \zC^\mathrm{op}\to \zC$ for the functor defined by the formula $X^*=[X,K]$, which is an involution by assumption. Moreover, it satisfies the following property: there is a natural isomorphism
    \begin{equation}
        \label{eqn:def_*-autonomous}
        \zC(X \otimes Y , Z^* ) \cong \zC(X, (Y \otimes Z)^* )
    \end{equation}
for every three objects $X,Y,Z \in \zC$ (see \cite[Section 6]{Barr-nonsymmetric_star-autonomous} for a proof in the non-symmetric setting).

\begin{ex}
    Any rigid symmetric monoidal category defines a $*$-autonomous category, with dualizing object the monoidal unit $\1$.
\end{ex}

\begin{ex}
    There are several examples of $*$-autonomous categories of sheaves coming from geometry. To name one, the bounded derived category of constructible $\ell$-adic sheaves on a scheme of finite type over a field, endowed with its dualizing complex, is a $*$-autonomous category by \cite[Example 1.6]{Boyarchenko-Drinfeld2013}. This is an incarnation of Verdier duality, which motivates the choice of the authors to use the terminology \emph{Grothendieck--Verdier categories} for $*$-autonomous categories. 
\end{ex}

\subsubsection{\texorpdfstring{$*$}{*}-autonomous categories as linearly distributive categories with duals}
One can view $*$-autonomous categories as particular cases of symmetric \emph{linearly distributive categories}, which were originally introduced and studied under the name of \emph{weakly distributive categories} in \cite{Cockett_Seely_WeaklyDistributiveCategories} (in their non-symmetric version). 
More precisely, we have the following

\begin{construction}
    \label{const:lin_dis}
    Given  a $*$-autonomous category $(\zC,\otimes, *)$, we can define a second symmetric monoidal structure $\parr$ on $\zC$ by the formula
    \begin{equation*}
        X \parr Y := (X^* \otimes Y^*)^*
    \end{equation*}
    with unit $1^*$.
    There is a natural transformation
    \begin{align*}
        \label{eqn:deltas_lin_distr}
        \delta\colon & X\otimes (Y\parr Z) \to (X\otimes Y) \parr Z
    \end{align*}
    making $(\zC,\otimes,1,\parr,1^*,\delta)$ into a symmetric linearly distributive category; we refer to \cite[Proposition 4.11.]{Fuchs-Schaumann-Schweigert_Grothendieck-Verdier_duality_bimodules} for the precise definition of $\delta$ and a proof of this statement. 
\end{construction}

In turn, we can identify $*$-autonomous as those symmetric linearly distributive categories for which every object has a dual, in the appropriate sense (see \cite[Theorem 2.41]{Demirdilek-Schweigert_Frobenius}); in particular, for every object $X$ in a $*$-autonomous category $\zC$, there are  evaluation and coevaluation morphisms of the form $\ev \colon X\otimes X^*\to 1^*$ and $\coev \colon 1 \to X^*\parr X$. Using this perspective, the isomorphism \eqref{eqn:def_*-autonomous}  can be translated into a natural bijection
\begin{equation*}
    (-)^\sharp \colon \zC(X\otimes Y, Z) \cong \zC(X,Y^*\parr Z) \colon (-)^\flat.
\end{equation*}

\begin{rk}
    Given $f\colon X\otimes Y \to Z$,  the corresponding morphism $f^\sharp \colon X\to Y^*\parr Z$ is given by the composition
    \[
        X 
        \xrightarrow{\coev \otimes \id} 
        (Y^*\parr Y) \otimes X 
        \stackrel{\delta}{\longrightarrow}
        Y^*\parr (Y\otimes X)
        \xrightarrow{\id \parr f}
        Y^*\parr Z .
    \]
    Conversely, the inverse construction sends a morphism $g\colon A\to B\parr C$ to the composite
    \[
        g^\flat\colon 
        A\otimes B^*
        \xrightarrow[]{g\otimes \id}
        (B\parr C)\otimes B^*
        \cong
        (C\parr B)\otimes B^*
        \xrightarrow[]{\delta}
        C\parr (B\otimes B^*)
        \xrightarrow[]{\id \parr \ev}
        C.
    \]
\end{rk}

\subsubsection{Symmetric linearly distributive categories and underlying dioperads}\label{sec:UndelyingDioperad}
In \cite[Section 2.2]{Cockett_Seely_WeaklyDistributiveCategories}, the authors define a functor from linearly distributive categories to non-symmetric dioperads. In our symmetric setting, there is an analogous construction, which we now recall.

\begin{construction}
    Given a symmetric linearly distributive category $\zC$, its \emph{underlying dioperad}, denoted $\overline{\zU}\zC$, has colors given by the objects of $\zC$ and operations given by the formula
    \begin{equation*}
        \overline{\zU}\zC((X_i); (Y_j)) := \zC\left( \otimes_i X_i, \parr_j Y_j\right).
    \end{equation*}
    The dioperadic composition is given as follows: given operations $\alpha \colon (X_i) \to (Y_j) \cup (Z)$ and $\beta \colon (Z)\cup (V_k) \to (W_\ell)$, their composite  $\beta\circ_Z\alpha \colon (X_i)\cup(V_k)\to (Y_j)\cup(W_\ell)$ in $\overline{\zU}\zC$ is defined as the following composition in $\zC$:
    \[
        \otimes X_i \otimes V_k
        \xrightarrow[]{\alpha\otimes \id}
        (\parr Y_j \parr Z) \otimes V_k
        \xrightarrow[]{\delta}
        \parr Y_j \parr (Z \otimes V_k) 
        \xrightarrow{\id\parr \beta}
        \parr Y_j \parr W_\ell.
    \]
\end{construction}

Given a $*$-autonomous category $\zC$, its \emph{underlying dioperad} $\overline{\zU}\zC$ is defined as the underlying dioperad of the symmetric linearly distributive category associated to $\zC$ by Construction \ref{const:lin_dis}.

\subsection{Day convolution}
The following result has been conjectured in \cite[Remark 4.4]{Egger_StarAutonomousFctCategories}.

\begin{thm}
    \label{thm:day}
    Let $\zC$ be a $*$-autonomous category and $\zD$ be a dioperad. Then there exists a dioperad $\Fun(\zC,\zD)$ with the following universal property: for every dioperad $\zO$, there is an equivalence of categories
    \begin{equation*}
        \Diop(\zO,\Fun(\zC,\zD)) \simeq 
        \Diop(\zO\times \overline{\zU}\zC, \zD).
    \end{equation*}
\end{thm}

The rest of the section is devoted to the proof of this theorem.

\subsubsection{Construction of the dioperad \texorpdfstring{$\Fun(\zC,\zD)$}{Fun(C,D)}.}

Let us define explicitly the dioperad $\Fun(\zC,\zD)$. Its colors are functors $\zC\to \zD$ between the categories underlying $\zC$ and $\zD$. Given functors $F_1,\dots, F_n, G_1,\dots, G_m \in \Fun(\zC,\zD)$, we define the set of operations as follows: consider the two maps
$s,t$
\begin{equation}
    \label{eqn:def_equalizer}
    \begin{tikzcd}
        \prod\limits_{x\in \zC} \:\lim\limits_{\substack{{a_1\otimes \dots \otimes a_n \to x}\\ {x \to b_1\parr \dots \parr b_m}}} \zD(F_1a_1, \dots, F_na_n; G_1b_1, \dots, G_mb_m) \\
        \prod\limits_{y\to z\in \zC} \:\lim\limits_{\substack{{a_1\otimes \dots \otimes a_n \to y}\\ {z \to b_1\parr \dots \parr b_m}}} \zD(F_1a_1, \dots, F_na_n; G_1b_1, \dots, G_mb_m)
        \arrow[shift right=3, from=1-1, to=2-1, "s" left]
        \arrow[shift left=3, from=1-1, to=2-1, "t"]
    \end{tikzcd}
\end{equation}
induced by taking source and target. 

\begin{df}
    \label{df:operations_FunCD}
    An element in the domain of $s$ and $t$ in \eqref{eqn:def_equalizer} is called a \emph{natural operation} if it lies in their equalizer. We define the set $\Fun(\zC,\zD)(F_1,\dots, F_n; G_1,\dots, G_m)$ as the equalizer $\mathrm{eq}(s,t)$.
\end{df}

\begin{rk}
    For further reference, we make explicit the following property of natural operations in $\Fun(\zC,\zD)$, which follows from the definition of the limits involved in formula \eqref{eqn:def_equalizer}. Given a natural operation $\alpha\in \Fun(\zC,\zD)((F_i);(G_j))$ and morphisms $\otimes a_i \xrightarrow{f} x \xrightarrow{g} \parr b_j \xrightarrow{\parr u_j}\parr b'_j$ where $u_j\colon b_j\to b'_j$, we have the following equality
    \begin{equation}
        \label{eqn:rewriting_lim}
        \alpha_{f, \parr u_j \circ g} 
        = 
        (Lu_j) \circ_{(Lb_j)} 
        \alpha_{f,g},
    \end{equation}
    and similarly for morphisms $a_i\to a'_i$. 
\end{rk}

\paragraph{Description of the composition in $\Fun(\zC,\zD)$.}
Let us now describe how the composition is defined. Given operations $\alpha\colon (F_i)\to (G_j)\cup (L)$ and $\beta\colon (L) \cup (H_k)\to (K_\ell)$, we construct an element $$\beta\circ_L\alpha \colon (F_i) \cup (H_k) \to (G_j) \cup (K_\ell)$$ as follows.
Consider an object $x\in \zC$ as well as families of objects $a_i, b_j, c_k, d_\ell \in \zC$ indexed as the operations above, set $a = \otimes_i a_i$, $c = \otimes_k c_k$, $b = \parr_j b_j$ and $d = \parr_\ell d_\ell$. Let $a \otimes c \stackrel{f}{\to} x \stackrel{g}{\to} b \parr d$ be two morphisms. We now define the component of $\beta\circ_L \alpha $ corresponding to this data. 

Let $x' = (c^*\parr x) \otimes b^*$. By duality, the morphism $f$ induces a morphism $f^\sharp \colon a\to c^*\parr x$. 
Taking the component of $\alpha$ associated to $f^\sharp$ and $\delta \circ (\id \otimes \coev) \colon c^*\parr x \to (c^*\parr x)\otimes (b^*\parr b) \to x'\parr b$ yields an operation 
\[ 
    \alpha_{f^\sharp, \delta \circ (\id \otimes \coev)} 
    \colon (F_i(a_i))\to (G_j(b_j))\cup (L(x')).
\]
Similarly, we can set $x'' := c^*\parr (x\otimes b^*)$. By duality $g$ induces a morphism $g^\flat \colon x\otimes b^*\to d$; then the component of $\beta$ associated to $g^\flat$ and to $(\ev\otimes \id)\circ \delta \colon c\otimes x'' \to (c\otimes c^*)\parr (x\otimes b^*)\to x\otimes b^*$ yields an operation
\[ 
    \beta_{g^\flat, (\ev\otimes \id)\circ \delta} \colon
    (H_k(c_k)) \cup (L(x'')) \to (K_\ell(d_\ell)).
\]
Finally, using the natural transformation $\delta \colon x'\to x''$, we define the desired component of $\beta\circ_L \alpha$  as the following composition in $\zD$~:
\begin{equation}
    \label{eqn:def_composition}
    \beta_{g^\flat, (\ev\otimes \id)\circ \delta} 
    \circ_{L(x')}
    L(\delta)
    \circ_{L(x'')}
    \alpha_{f^\sharp, \delta \circ (\id \otimes \coev)}
    \colon 
    (F_i(a_i)) \to (K_\ell(d_\ell)).
\end{equation}

It is a lengthy but straightforward exercise to show that the composition defined above is well-defined, associative and unital with respect to the identity natural transformation.

\subsubsection{Universal property of \texorpdfstring{$\Fun(\zC,\zD)$}{Fun(C,D)}}

\paragraph{The evaluation morphism. }
We construct a morphism of dioperads
\begin{equation*}
    \mathsf{ev} \colon \Fun(\zC,\zD)\times \overline{\zU}\zC\to \zD.
\end{equation*}

At the level of underlying categories, this is simply the usual evaluation functor $(F,x)\mapsto F(x)$.
At the level of operations, given sequences of objects $(F_i,a_i)$ and $ (G_j,b_j)$ in $\Fun(\zC,\zD)\times \overline{\zU}\zC$, we define the desired map
\begin{equation}
    \label{eqn:def_eval}
    \Fun(\zC,\zD)((F_i);(G_j)) \times \zC(\otimes a_i;\parr b_j)\longrightarrow
    \zD((F_ia_i); (G_jb_j)).
\end{equation}
as the projection to the component given by $x= \otimes_i a_i$ in the product \eqref{eqn:def_equalizer} (or to the component $x=\parr_j b_j$, which is equivalent by naturality).

\begin{lm}
    \label{lm:ev_is_map_diop}
    The assignment \eqref{eqn:def_eval} defines a morphism of dioperads.
\end{lm}
We provide details for the proof of this result, as it illustrates well the typical arguments involved in the rest of the proof of Theorem \ref{thm:day}.

\begin{proof}[Proof of Lemma \ref{lm:ev_is_map_diop}]
    Consider operations $f\colon (a_i)\to (b_j)\cup(x)$ and $g\colon (c_k)\cup(x)\to (d_\ell)$ in $ \overline{\zU}\zC$, as well as 
    $\alpha\colon (F_i)\to (G_j)\cup(L)$ and $\beta\colon (L)\cup(H_k)\to (K_\ell)$ in $\Fun(\zC,\zD)$. We want to show that
    \begin{equation}
        \label{eqn:equality_ev}
        \mathsf{ev}(\beta\circ_L \alpha, g\circ_x f) 
        = \mathsf{ev}(\beta,g)\circ_{\mathsf{ev(L,x)}} \mathsf{ev}(\alpha,f).
    \end{equation}
    As above, we set $a:=\otimes a_i$, $c:=\otimes c_k$, $b:= \parr b_j$ and $d:=\parr d_\ell$. Moreover, throughout this proof we abreviate $X\otimes Y$ as $XY$, and remove parenthesis using the convention that $XY\parr Z$ should be interpreted as $(X\otimes Y)\parr Z$. Using these notations, the dioperadic composition $g\circ_x f$ can be written as
    \begin{equation*}
        g\circ_x f\colon ac \xrightarrow{f\cdot \id}
        (b\parr x)c \xrightarrow{\delta}
        b\parr xc \xrightarrow{\id\parr g}
        b\parr d.
    \end{equation*}
    Using the naturality property of $\beta\circ_L\alpha$ (Definition \ref{df:operations_FunCD}) and unravelling the definition of composition \eqref{eqn:def_composition}, the left hand side of equation \eqref{eqn:equality_ev} is given by
    \begin{align}
        \label{eqn:computation_ev_1}
        \mathsf{ev}(\beta\circ_L \alpha, g\circ_x f) 
        \nonumber
        & = (\beta\circ_L\alpha)_{ac\xrightarrow{f\cdot \id}(b\parr x)c\xrightarrow{(\id\parr g)\circ \delta} b\parr d}\\
        & = \beta_{cx''\xrightarrow{\ev\circ \delta}(b\parr x)cb^*\xrightarrow{((\id\parr g)\id\circ\delta)^\flat}d} \circ_{Lx''} L\delta \circ_{Lx'} \alpha_{a \xrightarrow{(f\cdot \id)^\sharp} c^*\parr (b\parr x)c\xrightarrow{\delta\circ \coev}x'\parr b}.
    \end{align}
    where $x'=(c^*\parr (b\parr x)c)b^*$ and $x''=c^*\parr (b\parr x)cb^*$. The core of the proof consists in a suitable rewriting of the components of $\beta$ and $\alpha$, using naturality in the sense of Definition \ref{eqn:def_composition} and property \eqref{eqn:rewriting_lim}.

    We start with the term $\alpha$. Observe that there is a commutative diagram
    \begin{equation*}
        \begin{tikzcd}[column sep = 8 em]
    & a & {c^*\parr ac} \\
    & {b\parr x} & {c^*\parr (b\parr x)c} \\
            {x\parr b} & {(b\parr x)b^*\parr b} & {x'\parr b}
            \arrow["{\delta\circ(\coev\cdot \id)}", from=1-2, to=1-3]
            \arrow["f"', from=1-2, to=2-2]
            \arrow["{(f\cdot \id)^\sharp}"description, from=1-2, to=2-3]
            \arrow["{\id\parr f}", from=1-3, to=2-3]
            \arrow["{\delta\circ(\coev\cdot \id)}"', from=2-2, to=2-3]
            \arrow[equals, bend right = 10, from=2-2, to=3-1]
            \arrow["{\delta\circ\coev_b}"', from=2-2, to=3-2]
            \arrow["{\delta\circ\coev_b}"', from=2-3, to=3-3]
            \arrow["{(\ev_b\circ\delta)\parr \id}"{description}, from=3-2, to=3-1]
            \arrow["{\delta\circ(\coev\cdot \id)\parr \id}"{description}, from=3-2, to=3-3]
        \end{tikzcd}
    \end{equation*}
    Commutativity of the right part of this diagram allows to write the component of $\alpha$ appearing in \eqref{eqn:computation_ev_1} as
    \begin{align}
        \label{eqn:ev_1}
        \alpha_{a \xrightarrow{(f\cdot \id)^\sharp} c^*\parr (b\parr x)c\xrightarrow{\delta\circ \coev}x'\parr b}
        \nonumber
        &= \alpha_{a \xrightarrow{f} b\parr x \to x'\parr b} \\ 
        &= L(\delta\circ\coev_c)\circ_{L((b\parr x)b*)}
        \alpha_{a \xrightarrow{f} b\parr x\to (b\parr x)b^*\parr b}
    \end{align}
    where the implicit maps are those in the diagram, 
    while commutativity of the left part yields
    \begin{align}
        \label{eqn:ev_2}
        \alpha_{a\xrightarrow{f}b\parr x = b\parr x}
        = L(\ev\circ \delta)\circ_{L((b\parr x)b^*)} 
        \alpha_{a \xrightarrow{f} b\parr x \to (b\parr x)b^*\parr b}.
    \end{align}

    We proceed similarly for the components of $\beta$ using the following commutative diagram:
    \begin{equation*}
        \begin{tikzcd}[column sep = 8em]
            xc & {(c^*\parr xc)c} & {cx''} \\
               & xc & {(b\parr x)cb^*} \\
               & d & {(b\parr d)b^*}
               \arrow["{(\delta\circ\coev)\id}", from=1-1, to=1-2]
               \arrow[equals, bend right = 10, from=1-1, to=2-2]
               \arrow["{\ev_c\circ\delta}", from=1-2, to=2-2]
               \arrow["{\delta\circ\ev_b}"', from=1-3, to=1-2]
               \arrow["{\ev_c\circ\delta}", from=1-3, to=2-3]
               \arrow["g"' right, from=2-2, to=3-2]
               \arrow["{\delta\circ\ev_b}"', from=2-3, to=2-2]
               \arrow["{((\id\parr g)\id\circ\delta)^\flat}"{description}, from=2-3, to=3-2]
               \arrow["{(\id\parr g)\id\circ \delta}", from=2-3, to=3-3]
               \arrow["{\delta\circ\ev_b}", from=3-3, to=3-2]
        \end{tikzcd}
    \end{equation*}
    which gives the identities
    \begin{align}
        \label{eqn:ev_3}
        \nonumber
        \beta_{cx''\xrightarrow{\ev\circ\delta}(b\parr x)cb^*\xrightarrow{((\id\parr g)\id\circ\delta)^\flat} d}
        &= \beta_{cx''\to xc\to d}\\
        &= \beta_{(c^*\parr xc)c\to xc\to d} \circ_{L(c^*\parr xc)} L(\ev_b\circ \delta)
    \end{align}
    together with 
    \begin{equation}
        \label{eqn:ev_4}
        \beta_{xc=xc \xrightarrow{g}d} = 
        \beta_{(c^*\parr xc)c \to xc\to d} \circ_{L(c^*\parr xc)} L(\delta\circ\coev_c).
    \end{equation}
    The last ingredient is commutativity of the following diagram
    \begin{equation}
        \label{eqn:lm_ev_last_diag}
        \begin{tikzcd}[row sep = scriptsize]
            {(b\parr x)b^*} & {x'} \\
                            && {x''} \\
            x && {c^*\parr xc}
            \arrow["{\delta\circ\coev_c}"', from=1-1, to=1-2]
            \arrow["{\ev_b\circ\delta}", from=1-1, to=3-1]
            \arrow["\delta"', from=1-2, to=2-3]
            \arrow["{\ev_b\circ\delta}"', from=2-3, to=3-3]
            \arrow["{\delta\circ\coev_c}", from=3-1, to=3-3]
        \end{tikzcd}
    \end{equation}
    Putting all the above pieces together we obtain the desired result:
    \begin{align*}
        \mathsf{ev}(\beta\circ_L \alpha, g\circ_x f) 
        & \stackrel{\eqref{eqn:computation_ev_1}}{=} \beta_{cx''\xrightarrow{\ev\circ \delta}(b\parr x)cb^*\xrightarrow{((\id\parr g)\id\circ\delta)^\flat}d} \circ_{Lx''} L\delta \circ_{Lx'} \alpha_{a \xrightarrow{(f\cdot \id)^\sharp} c^*\parr (b\parr x)c\xrightarrow{\delta\circ \coev}x'\parr b}\\
        &\stackrel{\eqref{eqn:ev_1}, \eqref{eqn:ev_3}}{=} \beta_{(c^*\parr xc)c\to xc\to d} \circ
        L(\ev_b\circ \delta)
        \circ L\delta\circ
        L(\delta\circ\coev_c)\circ
        \alpha_{a \xrightarrow{f} b\parr x\to (b\parr x)b^*\parr b}\\
        &\stackrel{\eqref{eqn:lm_ev_last_diag}}{=} 
        \beta_{(c^*\parr xc)c \to xc\to d} \circ
        L(\delta\circ\coev_c) \circ L(\ev\circ \delta)\circ
        \alpha_{a \to b\parr x \to (b\parr x)b^*\parr b} \\
        &\stackrel{\eqref{eqn:ev_2}, \eqref{eqn:ev_4}}{=}
        \beta_{xc=xc \xrightarrow{g}d} \circ 
        \alpha_{a\xrightarrow{f}b\parr x = b\parr x}\\
        &= \mathsf{ev}(\beta,g)\circ_{\mathsf{ev(L,x)}} \mathsf{ev}(\alpha,f).
    \end{align*}
\end{proof}

\paragraph{Proof of the universal property. } 
We turn to establishing the universal property of the dioperad $\Fun(\zC,\zD)$. Let $\zO$ be a dioperad. 
We consider the functor
\begin{equation*}
    \Phi \colon \Diop(\zO,\Fun(\zC,\zD)) \longrightarrow
    \Diop(\zO\times \overline{\zU}\zC, \zD).
\end{equation*}
given by $\Phi(-) = \ev \circ (- \times \id_{\baru\zC})$.

We now construct an quasi-inverse $\Psi$ to $\Phi$. Given a map of dioperads $\varphi \colon \zO\times \overline{\zU}\zC\to \zD$, we define $\rho:= \Psi(\varphi)\colon \zO \to \Fun(\zC,\zD)$ as follows. For every color $p\in \zO$, we set $\rho(p)$ to be the functor $\varphi(p,-)$. For a morphism $\alpha \colon (p_i)\to (q_j)$ in $\zO$, we define $\rho(\alpha)$ componentwise via the formula 
\[
    \rho(\alpha)_{f,g} := \varphi(\alpha, g\circ f),
\]
where $f\colon \otimes a_i \to x$ and $g \colon x\to \parr b_j$ are two morphisms in $\zC$. It is straightforward to see that $\rho(\alpha)$ is indeed an element of $\Fun(\zC,\zD)((\rho(p_i)); (\rho(q_j)))$. The fact that the assignment $\rho$ is indeed a morphism of dioperads $\zO\to \Fun(\zC,\zD)$ follows from arguments similar to those in the proof of \ref{lm:ev_is_map_diop}, which we omit.

We now extend the assignment $\Psi$ to a functor $\Diop(\zO,\Fun(\zC,\zD))\to \Diop(\zO\times \overline{\zU}\zC,\zD)$. Given a polynatural transformation $t\colon \varphi \Rightarrow \varphi'$ of morphisms of dioperads $\zO\times \overline{\zU}\zC\to \zD$, we define $\Psi(t)\colon \Psi(\varphi)\Rightarrow\Psi(\varphi')$ via the formula 
$\Psi(t)_p := (t_{p,x})_{x\in \zC}$
for every $p\in \zO$. The fact that $\Psi(t)$ is a polynatural transformation follows easily from compatibility between the dioperadic composition of operations in $\Fun(\zC,\zD)$ and the usual composition of functors. With these definitions, it is immediate to check that $\Psi$ is a functor.

The fact that $\Phi$ and $\Psi$ are quasi-inverse is then a simple verification. This concludes the proof of Theorem \ref{thm:day}.

\begin{rk}[Enriched version of Theorem \ref{thm:day}]
    \label{rk:day_enriched}
    The statement and proof of the convolution theorem generalize to the enriched case; the result then takes the following form. For any $*$-autonomous $\zV$-category $\zC$ and any $\zV$-dioperad $\zD$, there exists a $\zV$-dioperad $\Fun(\zC,\zD)$ with the following universal property: for every $\zV$-dioperad $\zO$, there is an equivalence of categories
    \begin{equation*}
        \label{eqn:enriched_day_convolution}
        \Diop_\zV(\zO,\Fun(\zC,\zD)) \simeq 
        \Diop_\zV(\zO\otimes \overline{\zU}\zC, \zD),
    \end{equation*}
    where the symbol $\otimes$ in the right hand side denotes the objectwise and aritywise  tensor product of $\zV$-dioperads (generalizing the Hadamard tensor product in operadic literature).
\end{rk}

\subsection{Application to \texorpdfstring{$d$}{d}-duality contexts}
\label{sec:day_application}
To conclude, we sketch how Theorem \ref{thm:day} on convolution for dioperads provides a different approach to proving Theorem \ref{thm:main}.

As in Section \ref{par:setting}, we fix a $d$-duality context $F \colon \zD \rightleftharpoons\zC \colon G$  for some invertible object $d\in \zV$, and use the notations from this section.

Recall that there is a $\zV$-dioperad $\mathrm{Frob}$, the \emph{Frobenius dioperad}, which has a single color and objects of operations $\mathrm{Frob}(m,n)\cong \1_\zV$ for every $m,n\geq 0$. 
This dioperad is usually presented as generated by an associative commutative  binary operation $\nabla$ with a unit $\eta$ and a coassociative cocommutative operation $\Delta$ with a counit $\varepsilon$, that satisfy a compatibility relation analogous to the \eqref{eqn:Frobenius_relation}. It will be convenient to use the following alternative description of this dioperad. 
\begin{lm}
    \label{lm:presentation_frob_duality}
    A $\mathrm{Frob}$-algebra in a $\zV$-dioperad $\zO$ is equivalent to the data of a commutative algebra $(A, \nabla,\eta)$  together with an operation $\varepsilon$ in $\zO(A;\emptyset)$, such that $\chi := \varepsilon\circ \nabla$ is the evaluation of a duality. In other words, $\chi $ has the property that there exists some operation $\gamma$ in $\zO(\emptyset;A,A)$ such that the two composites of $\chi $ with $\gamma$ along $A$ are $\id_A$.
\end{lm}
\begin{proof}
    This is a well-known result, see for example \cite[Theorem 1.6]{Street_Frobenius} (which uses the language of Frobenius monads).
\end{proof}

Note that the rigid symmetric monoidal $\zV$-category $\zC$ is in particular $*$-autonomous, with the property that the two underlying dioperads $\overline{\zU}\zC$ and $\zU\zC$ are isomorphic. By Theorem \ref{thm:day} (or rather its version for enriched categories, see Remark \ref{rk:day_enriched}) applied to the terminal $\zV$-dioperad $\mathrm{Frob}$, we have an equivalence of categories
\begin{equation}
    \label{eqn:convolution_applied_Frob}
    \Diop_\zV(\zU\zC,\zU\zD\{d\}) \cong \Alg_\mathrm{Frob}\left( \Fun(\zC,\zU\zD\{d\}) \right).
\end{equation}
Therefore, extending the functor $G\colon \zC\to \zD$ to a map of dioperads $\zU\zC\to \zU\zD\{d\}$ amounts to promoting it to a Frobenius algebra in the convolution dioperad $\Fun(\zC,\zU\zD\{d\})$. Since $G$ already admits a lax symmetric monoidal structure $(\nabla, \eta)$, by Lemma \ref{lm:presentation_frob_duality} it suffices to provide an operation $\varepsilon$ with the desired property. 
To describe it explicitly, we see by formula \eqref{eqn:def_equalizer} that it suffices to provide, for any composable morphisms $f\colon a\to x$ and $g\colon x\to \1_\zC$ in $\zC$, a morphism $G(a) \to \1_\zD$ of degree $-d$ in $\zD$, with appropriate naturality properties. To do so, we set $\varepsilon_{f,g}$ to be the composite $\xi\circ G(g\circ f)$, where $\xi$ is the given $d$-orientation. 

To show that $(G,\nabla,\eta,\varepsilon)$ actually yields a Frobenius algebra, by Lemma \ref{lm:presentation_frob_duality} it suffices to exhibit a suitable operation $\gamma$ in $\Fun(\zC,\zU\zD\{d\})(G,G;\emptyset)$. 
Recall that since  we are given a $d$-orientation $\xi$, the morphism $\beta$ of Definition \ref{df:orientation} is non-degenerate; we let   $\kappa_X\colon \1_\zD\to G(X)\otimes G(X^\vee)$ denote the associated coevaluation morphism of degree $d$.
Now given morphisms $f\colon \1_\zC\to x$ and $g\colon x\to b_1\otimes b_2$, we set $\gamma_{f,g}$ to be the composite
\[\begin{tikzcd}
	\1 & {G(\1)} & {G(b_1 b_2)} & {G(b_1)G(b_1^\vee)G(b_1b_2)G(b_2) G(b_2^\vee)} \\
	{G(b_1)G(b_2)} && {G(b_1)G(b_2)G(b_2^\vee b_1^\vee) G(b_1 b_2)} & {G(b_1)G(b_2)G(b_2^\vee) G(b_1^\vee) G(b_1 b_2)}
	\arrow["{\nabla_0}", from=1-1, to=1-2]
	\arrow["{G(g\circ f)}", from=1-2, to=1-3]
	\arrow["{\kappa_{b_1}\cdot \id\cdot \kappa_{b_2}}", from=1-3, to=1-4]
	\arrow["\cong"{marking, allow upside down}, draw=none, from=1-4, to=2-4]
	\arrow["{\id \cdot \beta_{b_1b_2}}"', from=2-3, to=2-1]
	\arrow["{\id\cdot \nabla_{b_2^\vee,b_1^\vee}\cdot \id}"', from=2-4, to=2-3]
\end{tikzcd}\]
where we omit the $\otimes$ symbols for convenience.
With these definitions, using techniques similar to those of Theorem \ref{thm:main}, one can verify that $\chi:=\varepsilon\circ \nabla$ and $\gamma$ form a duality.

The following result summarizes the above construction.
\begin{prop}
    \label{prop:orientation_from_convolution}
    Any $d$-duality context yields a Frobenius algebra structure in the convolution dioperad $\Fun(\zC,\zU\zD\{d\})$ on the right adjoint $G$.
\end{prop}
Moreover, it is easy to see that under the equivalences \eqref{eqn:convolution_applied_Frob} and Proposition \ref{prop:twisted_Frob_monoidal}, this structure corresponds to the $d$-twisted Frobenius monoidal structure of Theorem \ref{thm:main}.


\bibliographystyle{alpha}
\bibliography{bibliography}

\end{document}